\documentclass[12pt]{amsart}
\usepackage[utf8]{inputenc}
\usepackage{amsfonts,graphics,amsmath,amsthm,amscd, amssymb,latexsym,enumitem} \usepackage[all]{xy}  \usepackage{mathtools} \usepackage{pigpen}\usepackage{mathrsfs}\usepackage{manfnt}\usepackage[utf8]{inputenc}\usepackage{tikz}\usetikzlibrary{arrows,decorations.markings,plotmarks,decorations.markings}\usepackage{tikz-cd}
\usepackage{comment}
\usepackage{epsfig}
\usepackage{flafter}
\usepackage{typearea}
\typearea{11}
\usepackage[pagewise]{lineno}
\theoremstyle{plain}
\newtheorem{theorem}{Theorem}[section]
\newtheorem{corollary}[theorem]{Corollary}

\newtheorem{lemma}[theorem]{Lemma}
\newtheorem{claim}[theorem]{Claim}

\newtheorem{proposition}[theorem]{Proposition}
\newtheorem{definition-lemma}[theorem]{Definition-Lemma}

\newtheorem{defn}[theorem]{Definition}
\newtheorem{remark}[theorem]{Remark}
\newtheorem{conjecture}[theorem]{Conjecture}

\def\ol#1{\overline{#1}}

\def\ideal#1.{I_{#1}}
\def\ring#1.{\mathcal {O}_{#1}}
\def\fring#1.{\hat{\mathcal {O}}_{#1}}
\def\proj#1.{\mathbb P(#1)}
\def\pr #1.{\mathbb P^{#1}}
\def\af #1.{\mathbb A^{#1}}
\def\Hz #1.{\mathbb F_{#1}}
\def\Hbz #1.{\overline{\mathbb F}_{#1}}
\def\pic#1.{\operatorname {Pic}\,(#1)}
\def\pico#1.{\operatorname{Pic}^0(#1)}
\def\picg#1.{\operatorname {Pic}^G(#1)}
\def\ner#1.{NS (#1)}
\def\rdown#1.{\llcorner#1\lrcorner}
\def\rup#1.{\ulcorner#1\urcorner}
\def\cone#1.{\operatorname {NE}(#1)}

\def\ccone#1.{\overline{\operatorname {NE}}(#1)}
\def\coef#1.{\frac{(#1-1)}{#1}}
\def\vit#1.{D_{\langle #1 \rangle}}
\def\mm#1.{\overline {M}_{0,#1}}
\def\H1#1.{H^1(#1,{\ring #1.})}
\def\ac#1.{\overline {\mathbb F}_{#1}}

\def\adj#1.{\frac {#1-1}{#1}}
\def\spn#1.{\overline{#1}}
\def\ses#1.#2.#3.{0\to #1\to #2\to #3 \to 0}
\def\pek#1.#2.{\Cal P^{#1}(#2)}
\def\plk#1.#2.{\Cal P^{\leq #1}(#2)}
\def\ev#1.{\operatorname{ev_{#1}}}
\def\bminv#1.{(\nu_1,s_1;\nu_2,s_2;\dots ;\nu_{#1},s_{#1};\nu_{r+1})}
\def\zinv#1.{(\nu_1,s_1;\nu_2,s_2;\dots ;\nu_{#1},s_{#1};0)}
\def\iinv#1.{(\nu_1,s_1;\nu_2,s_2;\dots ;\nu_{#1},s_{#1};\infty)}
\def\map#1.#2.{#1 \longrightarrow #2}
\def\rmap#1.#2.{#1 \dasharrow #2}
\def\emb#1.#2.{#1 \hookrightarrow #2}

\def\Supp{\operatorname{Supp}}

\def\dim{\operatorname{dim}}

\def\Null{\operatorname{Null}}

\def\supp{\operatorname{Supp}}

\def\C{\mathbb C}
\def\Q{\mathbb Q}

\def\wh{\widehat}

\def\ddbar{\sqrt{-1}\partial\overline\partial}

\def\e{\Cal E}

\def\e1{E_1}
\def\e2{E_2}
\def\del{\bar {\partial}}
\def\OO{\mathcal O}

\def\bbeta{\boldsymbol \beta}

\def\ggamma{\boldsymbol \gamma}

\def\wt{\widetilde}
\def\ep{\varepsilon}
\def\ddc{\sqrt{-1}\partial\overline\partial}

\newcommand{\po}{\ar@{}[dr]|{\text{\pigpenfont R}}}
\newcommand{\pb}{\ar@{}[dr]|{\text{\pigpenfont J}}}

\newcommand\R{{\mathbb{R}}}
\author{Christopher Hacon}
\address{Department of Mathematics\\
University of Utah\\
155 S 1400 E\\
Salt Lake City, Utah 84112, USA}
\email{hacon@math.utah.edu}
\thanks{Christopher Hacon was partially supported by the NSF research grants no: DMS-1952522, DMS-1840190, DMS-2301374.}

\author{Mihai P\u aun}
\address{Institut für Mathematik\\ Universität Bayreuth \\ 95440 Bayreuth, Germany
}
\email{mihai.paun@uni-bayreuth.de}
\thanks{Mihai P\u aun gratefully acknowledges support from the DFG}

\title{On the Canonical Bundle Formula and Adjunction for Generalized K\"ahler Pairs}
\begin{document}

 \maketitle

 \begin{abstract}
    In this article we prove analogs of Kawamata's canonical bundle formula, Kawamata subadjunction and plt/lc inversion of adjunction for generalized pairs on K\"ahler varieties. We also show that a conjecture of \cite{BDPP13} in dimension $n-1$ implies that the cone theorem holds for any $n$-dimensional K\"ahler generalized klt pair $(X,B+\bbeta)$.
\end{abstract}

\tableofcontents

 Generalized pairs have been playing an increasingly prominent role in higher dimensional birational geometry (see eg. \cite{Birkar21} and references therein). Their analytic counterparts were introduced in \cite{DHY23} (see Definition \ref{d-gpair}) where it is shown that the minimal model program for compact K\"ahler generalized klt 3-fold pairs holds. Note that even in the projective case,  Definition \ref{d-gpair} is more general than the usual definition of generalized pairs since the "nef" part is only assumed to be a positive (1,1) form (instead of a nef divisor).
 This is extremely useful in the K\"ahler context as, in many instances, we can replace the use of an arbitrary ample (or big) divisor by a (modified) K\"ahler class. In particular, using this extra flexibility, \cite{DHY23} shows the finiteness of minimal models for compact K\"ahler generalized klt 3-fold pairs of general type, and that klt Calabi-Yau K\"ahler 3-folds are connected by finite sequences of flops.
 The theory of generalized pairs makes sense in all dimensions and it is hoped that many results from the projective minimal model program will carry through to this setting. In this paper we perform the first steps in this direction.
 We show that adjunction and inversion of adjunction hold for generalized pairs (both in the plt and lc cases), we prove a canonical bundle formula for generalized klt K\"ahler pairs and we show that assuming the BDPP conjecture in dimension $n-1$, then the cone theorem for generalized pairs holds in dimension $n$ (and in particular it holds unconditionally in dimension $4$). More precisely, we show the following.
 \begin{theorem}\label{t-inv}Let $(X,B+\bbeta )$ be a generalized pair and $S$ a component of $B$ of coefficient 1 with normalization $\nu :S^\nu \to S$. Then
    $(X,B+\bbeta )$ is generalized lc (resp. generalized plt) on a neighborhood of $S$ iff $({S^\nu},B_{S^\nu} +\bbeta _{S^\nu})$ is generalized lc (resp. $S$ is normal and $({S},B_{S} +\bbeta _{S})$ is generalized  klt).
\end{theorem}
Next we turn our attention to the following generalization of Kawamata's adjunction theorem for generalized pairs (cf. \cite[Theorem 1]{Kawamata98}).
\begin{theorem}\label{t-Ksad} Let $(X,B+\bbeta)$ be a generalized log canonical pair such that $(X,B'+\bbeta')$ is a generalized klt pair and $W\subset X$ is a minimal log canonical center of $(X,B+\bbeta)$. Then $W$ is normal and $(K_X+B+\bbeta)|_W=K_W+B_W+\bbeta _W$ is a generalized klt pair.
\end{theorem}
In order to prove this result, it is necessary to prove the following canonical bundle formula (cf. Theorem \ref{t-gp}).
 \begin{theorem}\label{t-gcbf}
 	Let $f:X\to Y$ be a projective morphism of compact normal K\"ahler varieties such that $f_*\OO_X=\OO_Y$ and $(X,B+\bbeta )$ is a generalized klt (or generalized lc) pair. If  $\gamma \in H^{1,1}_{\rm BC}(Y)$ is such that 
 	$[K_X+B+\bbeta _X ]=f^*\gamma$ then $\gamma =[K_Y+B_Y+\bbeta _Y]$ where $(Y,B_Y+\bbeta _Y)$ is a generalized klt (or generalized lc) pair.
 \end{theorem}
 Note that we expect semistable reduction to hold (unconditionally) for morphisms of compact analytic varieties, however the necessary references are not yet available. In the case of projective morphisms, we can deduce semistable reduction from the algebraic case (see \cite{AK00}, \cite{Karu99}).
 It is possible that the results of \cite{BdSB23} are already sufficient for our purposes, but the proof would appear to be more involved and we do not pursue it here. The proof of this result heavily uses Theorem \ref{L2/m1} (which is a generalization of a result of Guenancia \cite{Gue20}), which roughly speaking, states that $\bbeta _Y$ is pseudo-effective. By \cite[Theorem 2.36]{DHP22}, to show that 
 $\bbeta_Y$ is b-nef, it suffices to show that $\bbeta_Y|_Z$ is pseudo-effective for any subvariety $Z\subset X$. This can be checked by using semistable reduction and applying Theorem \ref{L2/m1}.

 Finally, assuming a key conjecture of Boucksom-Demailly-Paun-Peternell \cite[Conjecture 0.1]{BDPP13}, we show that the cone theorem for generalized klt pairs holds in arbitrary dimension (and unconditionally for pseudo-effective pairs in dimension $\leq 4$). This provides some evidence that the minimal model program holds in arbitrary dimension for generalized klt pairs. We refer the reader to \cite{DHY23} for the generalized klt minimal model program in dimensions $\leq 3$. 
\begin{conjecture}\label{c-BDPP13}
Let $X$ be a compact K\"ahler manifold. Then the canonical class
$K_X$ is pseudoeffective if and only if $X$ is not uniruled (i.e. not covered by rational
curves).\end{conjecture}
Note that the above conjecture is known to hold in dimension $\leq 3$.
Following ideas from \cite{CH20} and using Theorem \ref{t-gcbf} we then prove the following result.
\begin{theorem}\label{t-gkltcone} Assume that Conjecture \ref{c-BDPP13} holds in dimension $n$ (resp. in dimension $n-1$). Let $X$ be a compact $\Q$-factorial K\"ahler variety of dimension $n$ such that $(X,B+\bbeta )$ is generalized klt (resp. and $K_X+B+\bbeta_X$ is pseudo-effective), then
there are at most countably many rational curves $\{\Gamma _i\}_{i\in I}$ such that 
 \[\overline{\rm NA}(X)=\overline{\rm NA}(X)_{K_X+B+\bbeta _X \geq 0}+\sum _{i\in I}\mathbb R ^+[\Gamma _i],\]
 where 
 $0<-(K_X+B+\bbeta_X )\cdot \Gamma _i\leq 2n$. Moreover, if $B+\bbeta_X $ (or $K_X+B+\bbeta_X$) is big, then $I$ is finite.
\end{theorem}

We now turn to a more detailed description of some of the key results in this paper.
The most important results in this paper concern versions of the canonical bundle formula (see Theorems \ref{t-gcbf} and \ref{lelong}).
The typical set up for the canonical bundle formula is an algebraic fiber space $f:X\to Y$ where say $X,Y$ are normal projective varieties, $f_*\OO _X=\OO _Y$, and a log canonical pair $(X,B)$ such that $K_X+B\sim _{\Q,Y}0$.
We can then write \[K_X+B\sim _\Q f^*(K_Y+B_Y+M_Y)\] where $B_Y$ is the boundary part that measures the singularities of $f$, and the moduli part $M_Y$ is a $\Q$-divisor class which measures the variation of the fibers of $f$. For example if $f$ is an elliptic fbration, then $M_Y=j^*\OO _{\mathbb P ^1}(\frac 1 {12})$ where $j$ denotes the $j$ function, and if $X_P$ is a smooth  fiber of multiplicity $m$ over a codimension 1 point $P\in Y$ and $B=0$, then the coefficient of $B_Y$ along $P$ is $1-\frac 1m$.

The canonical bundle formula, roughly speaking, states that if the morphism $f$ is sufficiently well prepared (eg. $B=0$ and all fibers have simple normal crossings), then $(Y,B_Y)$ is log canonical and the  the moduli part $M_Y$ is a nef $\Q$-divisor. In particular $(Y,B_Y)$ is a generalized log canonical pair (and in fact this is the key motivation for introducing generalized pairs \cite{BZ16}).

 Results along this line are established in \cite{Kawamata98}, see \cite{Kollar07} for a detailed discussion.
 The positivity of the nef part is deduced from general positivity properties of pushforwards of the canonical bundle (see \cite{Kollar86}).
 The intuition here is that after performing several reductions, we can in fact assume that the moduli part coincides with $f_*\omega _{X/Y}$. 
 
While the canonical bundle formula has a large number of extremely important applications (eg. to sub-adjunction \cite{Kawamata98}), it is clear that in order to run proofs by induction on the dimension, it is important to establish versions of the canonical bundle formula that work for generalized pairs of the form $(X,B+M)$. This is achieved in \cite{Fil20}. Note that on projective varieties, nef classes are limits of ample divisors and hence one hopes the result for generalized pairs  \cite{Fil20} follow as a limit of the result for the usual pairs \cite{Kawamata98}.

In the K\"ahler context it is more practical to work with generalized pairs $(X,B+\beta)$ where $\beta \in H^{1,1}_{\rm BC}(X)$ is a nef class (see Definition \ref{d-gpair}) and $f:X\to Y$ is a holomorphic map of K\"ahler manifolds such that $K_X+B+\beta =f^*\gamma$ for some $\gamma \in H^{1,1}_{\rm BC}(Y)$. One can then define the boundary $B_Y$ and moduli parts $\beta _Y\in H^{1,1}_{\rm BC}(Y)$ just as in the projective case. Unluckily, it does not follow that $\beta $ is a limit of $\Q$-divisors and hence the arguments from the projective case do not apply in the K\"ahler case.

Our strategy is to first prove an analog of the positivity of $f_*\omega _{X/Y}$. We show that if $K_X+B+\beta =f^*\gamma$ and $(X,B+\beta) $ is generalized log canonical, then $K_{X/Y}+B+\beta$ is pseudo-effective (the precise statement is contained in Theorem \ref{t-gue}, which is an easy consequence of Theorem \ref{L2/m1} that generalizes \cite{Gue20}).
Once the morphism $f:X\to Y$ is sufficiently prepared (and $B$ is re-chosen appropriately), we have $K_{X/Y}+B+\beta=f^*\beta _Y$ and hence $\beta _Y$ is also pseudo-effective.
By \cite{DHP22}, it is known that to show that $\beta _Y$ is nef, it suffices to show that $\beta _Y|_W$ is nef where $W$ is (the normalization of) any subvariety of $Y$. To verify this we consider $f_W:X_W\to W$ (where $(\ldots )_W$ denotes restriction over $W$) and then apply Theorem \ref{t-gue} to the induced pair $(X_W,B_W+\beta _W)$. For the details of the proof see Theorem \ref{t-gp}. Note that for technical reasons we have to assume that $f$ is a projective morphism, but we expect the result will hold without this assumption.

Therefore, the technical heart of this paper is Theorem \ref{L2/m1} (which is the key ingredient in the proof of Theorem \ref{t-gue}).
To gain some intuition, note that in the setting of Theorem \ref{L2/m1} we consider $K_{X/Y}+B+\beta\equiv f^*\gamma+L$ where $L$ is a $\Q$-line bundle such that $\kappa (L|_{X_y})\geq 0$ for general $y\in Y$. Restricting over open subsets $U\subset Y$ we may trivialize $\gamma$ and treat $\beta |_{X_U}$ as a $\Q$-line bundle $F_U=\mathcal O_{X_U}(L-K_{X/Y}-B)$. Following \cite{PT18}, we construct a positive current $\Theta _U\geq 0$ fiber-wise from the $m$-th root of the sections of $mL|_{X_y}$ (for $m>0$ sufficiently big and divisible).
These currents then glue together to give a positive current $\Theta \equiv K_{X/Y}+B+\beta$.

Finally, we remark that it is possible to give direct arguments with analytic techniques to prove stronger versions of the canonical bundle formula Theorem \ref{t-gp}.
In fact Theorem \ref{lelong} shows that if we further assume that $\beta$ contains a smooth positive
representative, then we can conclude the stronger fact that $\beta _Y$  is a closed positive current with zero Lelong numbers.

{\bf Acknowledgment.} The authors would like to thank J. Cao, O. Das and M. Temkin for useful communications. Thanks go equally to S. Boucksom and C.-M. Pan for interesting discussions about Lelong numbers.
 \section{Preliminaries} Here we recall some definitions and results from \cite{DHY23}. We will say that $S$ is relatively compact if $S\subset S'$ is an open subset whose closure is compact. Similarly $\pi:X\to S$ is a proper morphism to a relatively compact space $S$ if there is a proper morphism $\pi':X'\to S'$ where  $S\subset S'$ is an open subset whose closure is compact and $X=\pi '^{-1}(S)$.
 \begin{defn}\label{d-gpair}
 	Let $\pi :X\to S$ be a proper morphism of normal K\"ahler
varieties such that $S$ is relatively compact and $X$ is a normal compact K\"ahler variety, $\nu :X'\to X$ a resolution of singularities, $B'$ an $\R$-divisor on $X'$ with simple normal crossings support, and $\beta '\in H^{1,1}_{\rm BC}(X')$, such that
\begin{enumerate}
\item $B:=\nu _*B'\geq 0$,
\item $[ \beta ']\in H^{1,1}_{\rm BC}(X')$ is nef over $S$, and 
\item $[K_{X'}+B'+\beta ']=\nu^* \gamma $, where $\gamma \in H^{1,1}_{\rm BC}(X)$.
\end{enumerate}
  Then we let $\beta =\nu _*\beta '$ and we say that $\nu: (X',B'+\beta')\to (X,B+\beta)$  is a generalized pair (over $S$).
We will often abuse notation and say that $(X/S,B+\beta)$ (or $(X,B+\beta)$) is a generalized pair (over $S$) and $\nu: (X',B'+\beta')\to (X,B+\beta)$ is a  resolution.
If moreover ${\rm Ex}(\nu)$ is a divisor such that ${\rm Ex}(\nu)+B'$ has simple normal crossings, then we say that $\nu$ is a log resolution and if ${\rm Ex}(\nu)$ supports a relatively ample divisor, then $\nu$ is projective.
We will often assume that $X=S$ and omit $\pi :X\to S$.
\end{defn}
\begin{remark}
Note that we can define the corresponding nef b-(1,1) form $\boldsymbol{\beta}:=\overline{\beta'}$ as follows. For any bi-meromorphic morphism $p :X''\to X'$ we define $\boldsymbol{\beta}_{X''}=p ^* \beta '$ and for any bi-meromorphic morphism $q :X''\to X'''$ we let $\boldsymbol{\beta}_{X'''}=q_*\boldsymbol{\beta}_{X''}$. Using the projection formula, one easily checks that $q_*\boldsymbol{\beta}_{X''}$ is well defined (i.e. $\boldsymbol{\beta}_{X'''}$ does not depend on the choice of the common resolution $X''$ of $X'$ and $X'''$) and that for any bi-meromorphic morphism $r:X_1\to X_2$ of birational models of $X$, we have $r_* \boldsymbol{\beta}_{X_1}=\boldsymbol{\beta}_{X_2}$.
We say that $\overline{\beta'}$ descends to $X'$. Note that for any bi-meromorphic morphism $p :X''\to X'$, we also have
$\overline{\beta'}=\overline{\beta''} $ where $\beta '' =\boldsymbol{\beta}_{X''}$, and so $\boldsymbol{\beta}$ also descends to $X''$.

Similarly, if $\nu :Y\to X'$ is a proper morphism, then write $K_Y+B_Y=\nu ^*(K_{X'}+B')$. For any proper morphism $\mu:Y\to Y'$ we have $B_{Y'}=\mu _* B_Y$. 
	In this way we have defined a b-divisor  $\mathbf B$ (whose trace $\mathbf B_Y$ on $Y$ is $B_Y$). Since the b-divisor $\mathbf {K+B}=\overline{K_{X'}+B_{X'}}$ and the b-$(1,1)$-form {\mathversion{bold}$\beta$}$=\overline{\beta _{X'}}$ descend to $X'$, we say that the generalized pair $(X,B+\beta)$ descends to $X'$.

     We will often denote the generalized pair $\nu: (X',B'+\beta')\to (X,B+\beta)$ by $(X,B+\bbeta )$ where $\bbeta =\overline{\beta '}$.
     Note that then $\beta '=\bbeta _{X'}$ and $B'=\nu ^*(K_X+B+\beta)-(K_{X'}+\beta ')$.
\end{remark}

We define the \emph{generalized discrepancies} $a(P;X,B+\bbeta)=-{\rm mult}_P(\mathbf B_Y)$, where $P$ is a prime divisor on a bimeromorphic model $Y$ of $X$. 
We say that $(X,B+\bbeta )$ is \emph{generalized klt} or generalized Kawamata log terminal (resp. \emph{generalized lc} or generalized log canonical) if for any log resolution $\nu:X'\to X$, we have $\lfloor \mathbf B_{X'} \rfloor\leq 0$, i.e. $a(P;X,B+\bbeta)>-1$ for all prime divisors $P$ over $X$ (resp. $a(P;X,B+\bbeta)\geq -1$ for all prime divisors $P$ over $X$). This can be checked on a single given log resolution.
We say that $(X,B+\bbeta )$ is \emph{generalized dlt} (divisorially log terminal) if there is an open subset $U\subset X$
such that $(U,(B+\bbeta)|_U)$ is a log resolution (of itself) and $-1\leq a(P;X,B+\bbeta)\leq 0$ for any prime divisor $P$ on $U$
and $-1< a(P;X,B+\bbeta)$ for any prime divisor $P$ over $X$ with center contained in $X\setminus U$.

\begin{lemma}\label{l-klt}
	Let $(X,B+\bbeta)$ be a generalized klt (resp. generalized dlt) variety. If $K_X+B$ is $\mathbb Q$-Cartier, 
	then $(X,B)$ is klt (resp. dlt).
\end{lemma}
\begin{proof}
	Let $f:X'\to X$ be a log resolution, $K_{X'}+B'+\bbeta _{X'}=f^*(K_X+B+\bbeta_X)$,
	where $\lfloor B'\rfloor\leq 0$, as $(X,B+\bbeta)$ is generalized klt.
	Let $K_{X'}+B^\sharp=f^*(K_X+B)$,
	then \[E:=K_{X'}+B'-f^*(K_X+B)\equiv f^*\bbeta_X-\bbeta_{X'}\] and so $E$ is exceptional and $-E$ is nef over $X$.  By the negativity lemma $E\geq 0$.
	But then \[B'=E+f^*(K_X+B)-K_{X'}=B^\sharp +E\] and so $\lfloor B^\sharp \rfloor\leq 0$, i.e. $(X,B)$ is klt. The statement about dlt singularities follows similarly.
\end{proof}
In dimension 2, the situation is particularly simple as shown by the following lemma.
\begin{lemma}\label{l-sfs}
If $(X,B+\bbeta)$ is a generalized klt, dlt, lc surface, then $K_X+B$ is $\mathbb R$-Cartier, $(X,B)$ is klt, dlt, lc and $\bbeta_X$ is nef.
\end{lemma}
\begin{proof}
	See \cite{DHY23}.
\end{proof}
The next result shows that, working locally over $X$, generalized klt pairs behave similarly to the usual klt pairs, and in particular they have rational singularities.
\begin{theorem}\label{t-gkltlocal}
 Let $(X,B+\bbeta)$ be a generalized klt pair, then $X$ has rational singularities and if we replace $X$ by a relatively compact Stein open subset, then the following hold: 
 \begin{enumerate}
     \item there exists a small bimeromorphic morphism $\mu :X^\sharp \to X$ such that $X^\sharp$ is $\mathbb Q$-factorial,
    \item if $K_{X^\sharp}+B^\sharp+\bbeta _{X^\sharp}=\mu ^*(K_X+B+\bbeta _X)$, then $\beta _{X^\sharp}\equiv _X \Delta ^\sharp$ so that $(X^\sharp,B^\sharp+\Delta ^\sharp)$ is klt, and
    \item if $\Delta =\mu _*\Delta  ^\sharp$, then $(X,B+\Delta)$ is klt.
   \end{enumerate}
 \end{theorem}
 \begin{proof} This follows from \cite{DHY23} but we include a proof for the convenience of the reader. Note that rational singularities is a local property and hence follows from (3) and \cite[Theorem 3.12]{Fuj22}.

 (1-2)  Let $\nu :X'\to X$ be a projective log resolution of $(X,B+\beta)$ and write $K_{X'}+B'+\beta '=\nu ^*(K_X+B+\beta)$ so that $\bbeta =\overline {\beta '}$. Let $E$ be the reduced exceptional divisor and for $0<\epsilon \ll 1$, let $B^*=(B')^{>0}+\epsilon E$ and $F=(B')^{<0}+\epsilon E$, then $K_{X'}+B^*+\beta '=\nu ^*(K_X+B+\beta)+F$ where the support of $F$ equals the set of all exceptional divisors, and $(X',B^*+\beta ')$ is generalized klt.
 In particular $\beta'\equiv _X F-(K_{X'}+B^*)$ where $F-(K_{X'}+B^*)$ is an $\mathbb R$-divisor, nef over $X$. As $\nu$ is projective and $X$ is Stein, we may assume that $F-(K_{X'}+B^*)$ is big. But then $\beta '\equiv_X \Delta'$, where $\Delta ' >0$ is an effective $\mathbb R$-divisor such that
 $(X',B^*+\Delta ')$ is klt. We may therefore run the relative $K_{X'}+B^*+\Delta '$ mmp (\cite{Fuj22} and \cite{DHP22}) and hence we may assume that we have a birational map $\psi: X'\dasharrow X^\sharp$ such that if $F^\sharp =\psi _*F$, $B^\sharp =\psi _*B^*$, $\beta^\sharp =\psi _*\beta '$  and $\Delta ^\sharp =\psi _*\Delta '$, then
 \[ F^\sharp \equiv _X K_{X^\sharp }+B^\sharp +\beta^\sharp \equiv_X K_{X^\sharp }+B^\sharp +\Delta ^\sharp \] is nef over $X$ so that $F^\sharp =0$ by the negativity lemma. Therefore $\mu:X^\sharp\to X$ is a small bimeromorphic morphism, $B^\sharp=\mu ^{-1}_*B$ and $X^\sharp$ is $\mathbb Q$-factorial. 
 Clearly $(X^\sharp,B^\sharp+\Delta ^\sharp)$ is klt.
 Note that each step of the above mmp preserves the 
 numerical equivalence $\beta ^\sharp\equiv _X\Delta ^\sharp$, and in particular $K_{X^\sharp}+B^\sharp+\beta ^\sharp=\mu ^*(K_X+B+\beta)$.
 
 (3) By the base point free theorem \cite[Theorem 4.8]{Nak87}, we have $K_{X^\sharp}+B^\sharp+\Delta ^\sharp \sim _{\mathbb Q,X}0 $ and the claim follows.\end{proof}
 The following result is a technical result that is useful in many situations, especially when proving results by induction on the dimension. It shows that up to replacing $X$ by a higher model, the locus of non-klt singularities is contained inside the reduced boundary of a  carefully chosen strongly $\Q$-factorial dlt pair. Recall that a $\Q$-line bundle is a reflexive rank 1 sheaf such that there exists a positive integer $m$ such that the reflexive hull $L^{[m]}:=(L^{\otimes m_0})^{\vee \vee}$ is a  line bundle.
 A variety $X$ is strongly $\Q$-factorial if every reflexive rank 1 sheaf is a $\Q$-line bundle.
 \begin{theorem}[DLT models]\label{t-dltmodel} Let $(X,B+\bbeta )$ be a generalized pair, where $X$ is relatively compact Stein. Then there exists a projective birational morphism $f^{\rm m}:X^{\rm m}\to X$ such that $X^{\rm m}$ is strongly $\Q$-factorial, all exceptional divisors $P$ have discrepancy $a(X,B+\bbeta ,P)\leq -1$ and
 $(X^{\rm m}, {(f^{\rm m})}^{-1}_*B+{\rm Ex}(f^{\rm m}))$ is generalized dlt.
\end{theorem}
\begin{proof}
The proof is similar to the proof in the case of the usual generalized pairs (see \cite[Theorem 3.2]{Fil20}), which in turn is based on ideas of Hacon (see \cite{KK10}). We include the argument for the convenience of the reader. Recall that $X$ is strongly $\Q$-factorial if for every reflexive rank 1 sheaf $F$ there exists an integer $m>0$ such that $(F^{\otimes m})^{**}$ is locally free (see \cite[Definition 2.2]{DH20}).

Let $f:X'\to X$ be a log resolution of the generalized pair $(X,B+\bbeta)$ and write $K_{X'}+B'+\beta '=f^*(K_X+B+\beta)$ where $\beta =\bbeta _X$ and $\beta '=\bbeta _{X'}$. 
We may assume that $f$ is defined by a sequence of blow ups over centers of codimension $\geq 2$  in $X$, and hence $f$ is a projective morphism and so we have $C\geq 0$ an $f$-exceptional divisor such that $-C$ is relatively ample.
We write $B'=f^{-1}_*\{ B\}+E^++F-G$ where $E^+,F,G$ are supported on the divisors of discrepancy $a\leq -1,\ -1<a <0,\ a >0$ respectively and we let $E={\rm red}(E^+)$ be the reduced divisor with the same support as $E$.
For any $0<\epsilon, \mu, \nu<1$, we have
\[E+(1+\nu)F-\mu C+\beta '=(1-\epsilon \mu)E+(1+\nu)F+\mu(\epsilon E-C)+\beta '.\]
Note that $-\mu C+\beta '\equiv _X -\mu C-(K_{X'}+B')$ is an ample $\R$-divisor (over $X$) and hence for $0<\epsilon \ll 1$,
$\mu(\epsilon E-C)+\beta '$ is also numerically equivalent to an ample $\R$-divisor (over $X$). Since $X$ is Stein, we may write \[-\mu C+\beta '\equiv _X H_{1}\qquad  {\rm and}\qquad
\mu(\epsilon E-C)+\beta '\equiv _X H_{2}\] where $B'+H_{1}+H_{2}$ has simple normal crossing support and $\lfloor H_{1}\rfloor =\lfloor H_{2}\rfloor =0$. 
Let \[\Delta_{\epsilon, \mu,\nu}:=f^{-1}_*\{ B\}+(1-\epsilon \mu)E+(1+\nu)F+H_{2}\equiv _Xf^{-1}_*\{ B\}+E+(1+\nu)F+H_{1},\] then $(X',\Delta_{\epsilon, \mu,\nu})$ is klt for $0<\nu\ll 1$, and by \cite{DHP22} there is a $\Q$-factorial minimal model $\psi:X'\dasharrow X^{\rm m}_{\epsilon, \mu,\nu}$ over $X$ so that
$K_{X^{\rm m}_{\epsilon, \mu,\nu}}+\Delta_{\epsilon, \mu,\nu}^{\rm m}$ is nef over $X$. By the equations above, this is also a minimal model for
the dlt pair $(X',f^{-1}_*\{ B\}+E+(1+\nu)F+H_{1})$. 
Let $B^{\rm m}_{\epsilon, \mu,\nu}=\psi_*(f^{-1}_*\{ B\}+E+F)$, then $\Delta^{\rm m} _{\epsilon, \mu,\nu}=B^{\rm m}_{\epsilon, \mu,\nu}+\psi_*(\nu F+H_1)$, and 
$(X^{\rm m}_{\epsilon, \mu,\nu}, B^{\rm m}_{\epsilon, \mu,\nu})$ is dlt.
Define \[N:= K_{X^{\rm m}_{\epsilon, \mu,\nu}}+B_{\epsilon, \mu,\nu}^{\rm m}+\nu F^{\rm m}_{\epsilon, \mu,\nu} + H^{\rm m}_{1,\epsilon, \mu,\nu}\equiv _X K_{X^{\rm m}_{\epsilon, \mu,\nu}}+\Delta_{\epsilon, \mu,\nu}^{\rm m}\]
\[T:=K_{X^{\rm m}_{\epsilon, \mu,\nu}}+B_{\epsilon, \mu,\nu}^{\rm m}+(E^+-E)^{\rm m}_{\epsilon, \mu,\nu}-G^{\rm m}_{\epsilon, \mu,\nu}+\beta ^{\rm m}_{\epsilon, \mu,\nu}\equiv_X 0 .\]
We then have
\[T-N\equiv _X \mu C^{\rm m}+(E^+-E)^{\rm m}_{\epsilon, \mu,\nu}-G^{\rm m}_{\epsilon, \mu,\nu}-\nu F^{\rm m}_{\epsilon, \mu,\nu}=:D^{\rm m}_{\epsilon, \mu,\nu},\]
where $-D^{\rm m}_{\epsilon, \mu,\nu}$ is nef and the pushforward of $D^{\rm m}_{\epsilon, \mu,\nu}$ to $X$ is effective. By the negativity lemma, $D^{\rm m}_{\epsilon, \mu,\nu}\geq 0$.  
The divisors $C$, $E^+-E$, $F$ and $G$ are independent of ${\epsilon, \mu,\nu}$, thus if $0<\mu \ll \nu \ll 1$, then $G^{\rm m}_{\epsilon, \mu,\nu}=\nu F^{\rm m}_{\epsilon, \mu,\nu}=0$. Then let $X^{\rm m}:=X^{\rm m}_{\epsilon, \mu,\nu}$, then $X^{\rm m}$ is strongly $\Q$-factorial as it is the output of a minimal model program (see \cite[Lemma 2.5]{DH20}) and  $(X^{\rm m}, {(f^{\rm m})}^{-1}_*B+{\rm Ex}(f^{\rm m}))$ is generalized dlt. 
\end{proof}
\subsection{Boundary and moduli parts}
Throughout this section we will assume that $\pi:Z\to S$ is a proper morphism of relatively compact normal analytic varieties, $f:X\to Z$ is a proper morphism of normal K\"ahler varieties such that $f_*\OO _X=\OO _Z$ and $(X/S,B+\bbeta )$ a generalized pair which is generalized log canonical over an open subset of $Z$. 
Recall that by assumption there is a log resolution $\nu :X'\to X$  of $(X,B+\bbeta )$, i.e. 
\begin{enumerate}
\item $\bbeta=\overline {\bbeta _{X'}}$ (that is $\bbeta$ descends to $X'$), 
\item $\bbeta _{X'}$ is nef over $S$, \item ${\rm Ex}(\nu)$ is a divisor and $\nu ^{-1}(B)\cup {\rm Ex}(\nu)$ has simple normal crossings.\end{enumerate}
We let $K_{X'}+B_{X'}+\bbeta _{X'}=\nu ^*(K_X+B+\bbeta _X)$ and we say that $K_{X'}+B_{X'}+\bbeta _{X'}$ is the log-crepant pull-back of $K_X+B+\bbeta _X$.
\begin{defn} For any prime divisor $Q$ on $Z$, let \[a_Q=a_Q(X,B+\bbeta)= {\rm sup}\{t \in \mathbb R|(X,B+tf^*Q +\bbeta) \ {\rm is\ glc\ over}\ \eta _Q\}.\]
In this definition "glc" means that there is an analytic open subset $Z^0\subset Z$ intersecting $Q$ such that $(X,B+a_Qf^*Q +\bbeta)$ is glc but not gklt over $Z^0$. Note that as $Z$ is normal, $Z_{\rm sing}$ has codimension at least 2 and so  $Q$ is Cartier in codimension 1. Since $(X,B+\bbeta )$ is a generalized log canonical pair, $f$ is proper, and $Z$ is relatively compact, it is easy to see that $a_Q=1$ for all but finitely many divisors $Q\subset Z$.
We can then define the {\it boundary divisor} $B_Z=B(X/Z,B+\bbeta) =\sum (1-a_Q)Q$.
\end{defn}
\begin{remark} If $\eta :X'\to X$ and $\mu :Z'\to Z$ are proper bimeromorphic maps of normal varieties and $f':X'\to Z'$ is holomorphic,  $K_{X'}+B'+\bbeta_{X'}=\eta ^*(K_X+B+\bbeta _X)$, then we say that $(X',B'+\bbeta)$ is the induced generalized pair. If $Q'=\mu ^{-1}_*Q$, then it is easy to see that $a_Q(X,B+\bbeta)=a_{Q'}(X',B'+\bbeta)$. In particular, if $B_{Z'}$ is the boundary divisor for $(X',B'+\bbeta)$, then $\mu _* B_{Z'}=B_Z$, i.e. the boundary divisor defined by the above formula is in fact a b-divisor which we denote by $\mathbf B ^Z$ (so that $B_Z=\mathbf B ^Z_Z$ and $B_{Z'}=\mathbf B ^Z_{Z'}$). 
\end{remark}
\begin{defn} If $K_X+B+\bbeta _X \equiv f^*\gamma$ for some $\del, \partial$ closed form $\gamma$, then we define the {\it moduli part} $\beta _Z:=\gamma -(K_Z+B_Z)$  of $(X/Z,B_{X}+\bbeta_{X})$. If $K_Y+B_Y+\bbeta _Y$ is the
log-crepant pull-back of $K_X+B+\bbeta _X$ and $\beta _{Z'}$ is the moduli part of $(X'/Z',B_{X'}+\bbeta_{X'})$, then it is easy to see that $\mu _* \beta _{Z'}=\beta _Z$ and so we have a b-(1,1) form $\bbeta ^Z$ such that 
$\beta _Z=\bbeta ^Z_Z$. \end{defn}
\begin{defn} 
The pair $(X/Z, B+\bbeta )$ is said to be BP stable 
(over $Z$) if $\mathbf K + \mathbf B=\overline {(K_Z+B_Z)}$ i.e. if $K_Z+B_Z$ is $\R$-Cartier and for any contraction $f ' : X'\to  Z'$ which is birationally
equivalent to $f$ and such that the induced maps $\mu :Z'\to Z$ and $\eta :X' \to X$
are projective birational morphisms, then
$K_{Z'} + \mathbf B^Z_{Z'} = \mu ^*(K_Z + \mathbf B^Z_Z )$. Note that if $K_X+B+\bbeta _X \equiv f^*\gamma$, then the moduli part also descends to $Z$ i.e. $\bbeta ^Z=\overline {\beta _Z}$
\end{defn}

Suppose now that $\eta :X'\to X$ is a proper generically finite morphism. We define the pull-back $\bbeta ':=\eta ^*\bbeta$ as follows. Let $\nu :Y\to X$ be a log resolution of $(X,B+\bbeta)$ and $Y'$ be a resolution of the normalization of the main component of $Y\times _XX'$ so that $\rho :Y'\to Y$ is a generically finite holomorphic map and  $\nu ':Y'\to X'$ is a bimeromorphic map. Then let $\bbeta '=\overline {({\rho }^*\bbeta _{Y})}$ and $K_{Y'}+B_{Y'}={\rho }^*(K_{Y}+B_Y)$. Since $\bbeta '_{Y'}={\rho }^*\bbeta _{Y}$ is nef over $S$,  
$(Y',B_{Y'}+\bbeta '/S)$ defines a generalized pair.   
Now let $B'=\nu '_*B_{Y'}$, then \[K_{X'}+B'+\bbeta'_{X'}=\nu'_*(\rho ^*(\nu^*(K_X+B_X+\bbeta _X)))=\eta ^*(K_X+B_X+\bbeta _X)\] by the projection formula.
We will say that $(X',B'+\bbeta ')$ is the {\it log crepant pull-back} of $(X,B+\bbeta)$.
\begin{lemma}\label{l-desc} If $\bbeta$ descends to $X$, then $\bbeta '$ descends to $X'$.\end{lemma}
\begin{proof} By assumption $\nu ^*\bbeta_X = \bbeta _{Y}$. Then \[\nu '^*(\eta ^* \bbeta_X )=\rho ^*\nu ^*\bbeta_X  =\rho ^* \bbeta _{Y}=\bbeta ' _{Y'}.\] Thus $\bbeta '=\overline {\eta ^*\bbeta_X }$, i.e. $\bbeta'$ descends to $X'$. 
\end{proof}

Given a generically finite map $\mu:Z'\to Z$ and $h:X'\to X\times _Z Z'$ a proper bimeromorphic map from a normal variety to the main component, then we obtain a base change diagram  
\begin{equation}\label{diag-bsch} 
\begin{tikzcd}
X' \arrow{r}{{\eta}} \arrow[swap]{d}{f'} &  X \arrow{d}{f}\\
Z' \arrow{r}{\mu } & Z
\end{tikzcd}
\end{equation}
\begin{lemma}\label{l-fin} Let $ B_Z=B(X/Z,B+\bbeta )$ be the boundary divisor for $f:(X,B+\bbeta )\to Z$ and $B_{Z'}=B(X'/Z',B_{X'}+\bbeta ' )$ be the boundary divisor for $f':(X',B_{X'}+\bbeta ' )\to Z '$ where $\bbeta'=\eta ^*\bbeta$ and $K_{X'}+B_{X'}+\bbeta _{X'}$ is the log-crepant pull-back of $K_X+B+\bbeta _X$. If $\mu$ is finite, then $K_{Z'}+B_{Z'}=\mu ^*(K_Z+B_Z)$.
\end{lemma}
\begin{proof} 
See \cite[Theorem 3.2]{Amb99}.
\end{proof}

\begin{defn}A morphism $f:X\to Z$ is {\it weakly semistable} \cite[Definition 0.1]{AK00} if
\begin{enumerate}
\item $X$ and $Z$ admit toroidal structures $U_X \subset X$ and $U_Z \subset Z$, with $U_X=f^{-1}(U_Z)$;
\item  with this structure, the morphism $f$ is toroidal;
\item  the morphism $f$ is equidimensional;
\item  all the fibers of the morphism $f$ are reduced; and
\item  $Z$ is nonsingular.
\end{enumerate}
See \cite[Section 1]{AK00} for a discussion on toroidal morphisms.
If $(X,B)$ is a pair, then we say that $(X,B)$ has {\it good horizontal divisors} if for every $x\in X$ there exists a local model $X_\sigma =X_{\sigma '}\times \mathbb A ^l$ where the horizontal divisors (i.e. the divisors dominating $Z$) are exactly the pull-backs of the coordinate hyperplanes in $\mathbb A ^l$ \cite[Definition 9.1]{Karu99}.   \end{defn}
We recall the following facts.
\begin{theorem}\label{t-ws}
\begin{enumerate}
\item Weak semistability is preserved by base change (see \cite[Lemma 8.3]{Karu99}).
\item Let $f:X\to Z$ be a projective morphism of quasi-projective varieties and $W\subset X$ a closed subset a proper subscheme, then there exists a diagram \begin{equation}\label{diag-mod} 
\begin{tikzcd}
U_{X'}\subset X' \arrow{r}{{\mu_X}} \arrow[swap]{d}{f'} &  X \arrow{d}{f}\\
U_{Z'}\subset Z' \arrow{r}{\mu_Z} & Z
\end{tikzcd}
\end{equation}
such that $X',Z'$ are nonsingular, $\mu_{Z}$ and $\mu_X$ are birational, $f'$ is toroidal and $f^{-1}(W)\subset X'\setminus U_{X'}$ is a snc divisor \cite[Theorem 2.1]{AK00}.
\item There exists a finite surjective morphism $Z''\to Z'$ such that
denoting by $X''$ the normalization of $X \times _{Z'}Z''$, we have that $U_{X''}\subset X''$ and $U_{Z''}\subset Z''$ are toroidal embeddings,
the projection $f'' : X''\to Z''$ is an equidimensional toroidal morphism with reduced fibers, and
$(f'')^{-1}(U_{Z''})=U_{X''}$.
If $(X,B)$ is a pair, then we may assume $(X'',B'')\to Z''$ is weakly semistable with good horizontal divisors where $B''$ is the support of the inverse image of $B$ and the $X''\to X$ exceptional divisors see \cite[Proposition 5.1]{AK00} and \cite[Theorem 9.5]{Karu99}.
\end{enumerate}\end{theorem}
\begin{lemma}\label{l-semi} Suppose  that $(X/Z,{\rm Supp}(B))$ is weakly semistable with good horizontal divisors, $\bbeta$ descends to $X$ and $(Z,\Sigma)$ is a simple normal crossings pair such that $f(B^v)\subset \Sigma$. Then, the corresponding boundary divisor $B_Z$ descends to $Z$, i.e. $(X/Z,B)$ is BP-stable.  \end{lemma}
\begin{proof} Since $\bbeta $ descends to $X$, we can disregard $\bbeta$, and the claim follows from the usual argument for pairs, see eg. \cite[4.13]{Fil20}.\end{proof}
\begin{proposition}\label{p-BPs} Let $f:(X,B+\bbeta)\to Z$ be a locally projective morphism of relatively compact normal analytic varieties such that $(X,B+\bbeta )$ is a pair which is glc over an open subset of $Z$. Then $\mathbf B ^Z$ descends to some model $Z'$ i.e. if $f':X'\to Z'$ is bimeromorphic to $f:X\to Z$ and $(X',B'+\bbeta)$ is the induced pair, $\bbeta =\overline{\bbeta _{X'}}$ then $(X'/Z',B')$ is BP-stable.
\end{proposition}
\begin{proof} We may assume that $\bbeta$ descends to $X$ and hence we may disregard $\bbeta$ in what follows. 
The question is local over $Z$ and hence we may assume that $f$ is projective and $Z$ is relatively compact and Stein. 
By GAGA (see \cite[Apendix B, C]{AT19} for the details) we may assume that $f:X\to Z$ is a projective morphism of quasi-projective varieties.

Let $f':X'\to Z'$ and $f'':X''\to Z''$ be the morphisms defined by Theorem \ref{t-ws}. We may assume that $\pi _X^{-1}({\rm Supp}(B))+{\rm Ex}(\pi _X)$ has good horizontal divisors (over $Z''$) where $\pi _X:X''\to X$. 
Let $K_{X''}+B''=\pi _X^*(K_{X}+B)$, then $(X''/Z'',B'')$ is BP-stable by Lemma  \ref{l-semi}.
We must show that if $\nu_{Z'}:Z_1'\to Z'$ is any birational map, then $K_{Z'_1}+ B_{Z'_1}=\nu _{Z'}^*(K_{Z'}+ B_{Z'})$ where $B_{Z'_1}$ is the corresponding boundary divisor. 
Since $Z''\to Z'$ is finite, then $K_{Z''}+ B_{Z''}=\mu _{Z'}^*(K_{Z'}+ B_{Z'})$ by Lemma \ref{l-fin} (where $\mu _{Z'}:Z''\to Z'$ and $B_{Z''}$ is the corresponding boundary divisor).
Let $Z''_1$ be the normalization of $Z''\times _{Z'}Z'_1$, then $\mu_{Z'_1}:Z''_1\to Z_1'$ is finite and hence 
$\mu_{Z'_1}^*(K_{Z'_1}+B_{Z_1'})=K_{Z''_1}+ B_{Z_1''}$.
Let $\nu_{Z''} :Z_1''\to Z''$. Since $(X''/Z'',B'')$ is BP-stable and $ \mu_{Z'}\circ \nu _{Z''}=\nu _{Z'}\circ \mu_{Z'_1}$, then
\[\mu_{Z'_1}^*(K_{Z'_1}+ B_{Z_1'})=K_{Z''_1}+ B_{Z_1''}=\nu_{Z''} ^*(K_{Z''}+ B_{Z''})=\nu_{Z''} ^*(\mu_{Z'}^*(K_{Z'}+ B_{Z'}))=\mu_{Z'_1}^*(\nu_{Z'} ^*(K_{Z'}+ B_{Z'})).\]
Pushing forward, it follows that $K_{Z'_1}+ B_{Z_1'}=\nu_{Z'} ^*(K_{Z'}+ B_{Z'})$.

\end{proof}
\begin{theorem}\label{t-aws} Let $(X,B)$ be a pair and $f:X\to Z$ be a projective morphism of compact complex manifolds with connected fibers, then there exists a birational morphism $Z'\to Z$ and a finite morphism $Z''\to Z'$ and a morphism $f'':X''\to Z''$ birational to $X\times _ZZ''$ such that $(X''/Z'',B'')$ is weakly semistable with good horizontal divisors where $B''$ is the support of the inverse image of $B$ and the $X''\to X$ exceptional divisors.\end{theorem}
\begin{proof} Let
$\phi:(\mathcal X,\mathcal B) \to \mathcal Z$ the corresponding component of the Hilbert scheme. Replacing $\mathcal Z$ by a resolution of the image of $Z$ and replacing $Z$ by a higher model, we may 
assume that $Z\to \mathcal Z$ is surjective, generically finite of smooth manifolds, and $(X,B)=(\mathcal X,\mathcal B)\times _{\mathcal Z} Z$ over an open subset of $Z$. By Theorem \ref{t-ws}, there 
are 
\begin{enumerate} 
\item birational morphisms $\mu _{\mathcal Z}:\mathcal Z'\to \mathcal Z$ and $\mu _{\mathcal X}:\mathcal X'\to \mathcal X$ such that the induced map $\phi ':\mathcal X'\to \mathcal Z'$ is a toroidal morphism of smooth varieties,
\item a finite surjective morphism $\mathcal Z''\to \mathcal Z'$ such that denoting by $\mathcal X''$ the normalization
of $\mathcal X''\times _{\mathcal Z'}\mathcal Z''$ then $\phi '':\mathcal X''\to \mathcal Z''$ is an equidimensional toroidal morphism with reduced fibers and  $(\mathcal X'', \mathcal B'') \to \mathcal  Z''$ is weakly semistable with good
horizontal divisors where $\mathcal B''$ is the support of the inverse image of $\mathcal B$ and the $\mathcal X''\to \mathcal X$
exceptional divisors. 
\end{enumerate}
We let $Z'=Z\times _{\mathcal Z}\mathcal Z'$ and $Z''$ be an appropriate resolution of $Z\times _{\mathcal Z}\mathcal Z''$ (so that the inverse image of $\mathcal Z''\setminus U_{\mathcal Z''}$ is a divisor with simple normal crossings) and $X''=X\times _Z Z''$,  
$B''=B\times _{Z}Z''$, then   $f'':X''\to Z''$ is weakly semistable and $(X'',B'')$ has good horizontal divisors where $B''$ is the inverse image of $B$ plus the $X''\to X$ exceptional divisors.\end{proof}
\begin{remark} The morphism $f'':X''\to Z''$ is quasi-smooth in the following sense (see \cite{Taka} pages 1736-1737).
For every $x''\in X''$, there exists an open subset $x''\in U''\subset X''$ such that $(U'',{\rm Supp}(D'')|_{U''})=(\tilde U'',D_{\tilde U''})/G$ where $(\tilde U'',D_{\tilde U''})$ is a log smooth toric variety and $G$ is a finite abelian group and $f''\circ \pi:(\tilde U'',D_{\tilde U''})\to Z''$  is toric and flat and $\pi :\tilde U''\to U''$ is the quotient map.
In particular we can pick local coordinates $(x_1,\ldots , x_{n+m})$ on $U''$ and $(t_1,\ldots , t_m)$ on $Z''$ such that $(f''\circ \pi)^*(t_i)=\prod x_{j}^{k_{i,j}})$ where the $k_{i,j}$ are non-negative integers such that
$k_{i,j}\ne 0$  for at most one $i$ for each index $j$.
\end{remark}
\section{Adjunction of generalized pairs}
In this section we will prove a version of Kawamata's canonical bundle formula and of Kawamata's sub-adjunction for generalized K\"ahler pairs. We will also  prove plt and log canonical inversion of adjunction in this setting. 
\subsection{Towards a canonical bundle formula}
Let $f:X\to Z$ and $Z\to S$ be proper morphisms of normal relatively compact K\"ahler varieties such that $f_*\OO _X=\OO _Z$ and $(X/S,B+\bbeta )$ is a generalized  pair.
Suppose that $K_X+B+\bbeta _X=f^*\gamma $ for some $\del, \partial$ closed current $\gamma$ on $Z$.
Let $\ggamma=\overline \gamma$ so that $\ggamma _{Z'}=\mu ^*\gamma$ for any bimeromorphic map $\mu :Z'\to Z$.
We define the boundary b-divisor $\mathbf B^Z$ as above and for any base change diagram \eqref{diag-bsch}, we define $B_{Z'}$ accordingly.
In particular if $\mu :Z'\to Z$ is birational, then $B_{Z'}=\mathbf B ^Z_{Z'}$.
We also define the moduli part $\bbeta ^Z$ via
\[\bbeta ^Z=\ggamma-(\mathbf K + \mathbf B^Z),\qquad i.e.\qquad  K_{Z'}+\mathbf B^Z _{Z'}+\bbeta ^Z_{Z'}=\ggamma_{Z'}.\]
One expects, similarly to the algebraic case, that $\mathbf B ^Z,\bbeta ^Z$ descend to some model $Z'$, $(Z',\mathbf B ^Z_{Z'}) $ is klt (lc) and that $\bbeta ^Z_{Z'}$ is nef. In other words we conjecture the following.
\begin{conjecture} Let $f:(X,B+\bbeta )\to Z$ be a generalized klt (lc) pair as above, then $(Z,B_Z+\bbeta ^Z)$ is a generalized klt (lc) pair.
\end{conjecture}
We will prove this conjecture under additional conditions.
As a first step, we prove the following result which is a generalization of the main result of \cite{Gue20}.
\begin{theorem}\label{t-gue} Let $f:X\to Z$ be a surjective projective map with connected fibers between normal compact  K\"ahler varieties, $(X,B+\bbeta)$ a generalized pair which is log canonical over an open subset of $Z$, $Z$ is smooth, $\gamma$ a real (1,1)-class on $Z$, and $L$ a $\Q$-line bundle such that $h^0(X_z,L^{[m]}|_{X_z})\ne 0$ for $m\gg 0$ sufficiently divisible, $z\in Z$ general.
If  $K_X+B+\bbeta_X =f^*\gamma +L$, then $ K_{X/Z}+B+\bbeta_X$  is pseudo-effective.
\end{theorem}
\begin{proof} Let $\nu :X'\to X$ be a resolution 
and write  $K_{X'}+B'+\bbeta _{X'}=\nu ^*(K_X+B+\bbeta _X)+E$ where $B',E\geq 0$ are effective divisors with no common components. We may assume that $B'$ has simple normal crossings support. 
Since 
$E\geq 0$ is $\nu$-exceptional, by the projection formula
it suffices to show that $K_{X'/Z}+B'+\bbeta_{X'}$ is pseudo-effective. 
Let $B'=B^h+B^v$ where the components of $B^h$ dominate $Z$ and the components of $B^v$ do not dominate $Z$. 
Note that $(X,B^h)$ is log canonical and it suffices to show that $K_{X'/Z}+B^h+\bbeta_{X'}$ is pseudo-effective. 
For every $m>0$ we will write \[mF^m:=\lfloor B^h\rfloor +\{mB^h\}\qquad {\rm and}\qquad mB^m:= m(B^h-F^m)=\lfloor mB^h\rfloor-\lfloor B^h\rfloor.\]
In particular $(X',B^m)$ is klt and $F^m$ has fixed support with coefficients in $[0,1/m]$. 
Let $H '$ be a relatively ample divisor and $\omega$ be a K\"ahler form on $Z$ such that ${f'}^*\omega +H'$ is K\"ahler. For any $\epsilon>0$ there is an $m>0$ such that $\beta  ^{m,\epsilon}:=\bbeta_{X'}  +\epsilon ({f'}^*\omega +H ') +F^m$ is K\"ahler and $L^{\epsilon}:=\nu ^*L+E-B^v+\epsilon H'$ is a $\Q$-line bundle such that $h^0(X'_z,(L^{\epsilon})^{[m]}|_{X'_z})\ne 0$ for $m\gg 0$ sufficiently divisible, $z\in Z$ general. We may write
\[K_{X'}+B^m+\beta ^{m,\epsilon}=f'^*(\gamma+\epsilon \omega ) +L^{\epsilon }.\]
By Theorem \ref{L2/m1}, $K_{X'/Z}+B^m+\beta ^{m,\epsilon}$ is pseudo-effective.
Since being pseudo-effective is a closed condition, and \[K_{X'/Z}+B^h+\bbeta _{X'} =\lim \left(K_{X'/Z}+B^m+\beta ^{m,\epsilon}\right) ,\] it follows that  $K_{X'/Z}+B^h+\bbeta_{X'}$ is pseudo-effective.
\end{proof}
\begin{theorem}\label{t-gp} Let $(X,B+\bbeta )$ be a generalized klt (lc) pair, $f:X\to Z$ a surjective projective morphism of normal compact K\"ahler varieties with connected fibers, such that $K_X+B+\bbeta _X=f^*\gamma $ for some $\del, \partial$ closed current $\gamma$ on $Z$. Then $(Z,B_Z+\bbeta ^Z)$ is a generalized klt (lc) pair, i.e. \begin{enumerate} \item $\bbeta ^Z$ is b-nef, that is $\bbeta ^Z$ descend to some model $Z'$ so that $\bbeta ^Z_{Z'}$ is nef, and
\item $(Z',\mathbf B ^Z_{Z'}) $ is  sub-klt (sub-lc). \end{enumerate}

\end{theorem}
The proof of this result is somewhat technical but the strategy that we will now illustrate is fairly natural. We already know that the moduli part $\bbeta ^Z$ descends to some model, so assume for simplicity that it descends to $Z$. Similarly assume that $\beta=\bbeta _X$ is nef and $B$ has simple normal crossings. We must show that $\beta_Z:=\bbeta ^Z_Z$ is nef. By \cite{DHP22}, it suffices to show that $\beta _Z|_{W}$ is pseudo-effective on (the normalization of) any subvariety $W\subset Z$. When $W=Z$, this follows from Theorem \ref{t-gue}. If $W\ne Z$ the situation is more delicate. Assume further that $f$ is weakly semistable  and that $B_Z$ is a reduced divisor supported on $Z\setminus U_Z$ (cf. Theorems \ref{t-ws}, \ref{t-aws}). 
If $T$ is any component of $B_Z$, then there is a component $S$ of $B^{=1}$ dominating $T$.
By adjunction $(K_X+B+\beta )|_S=K_S+B_S+\beta _S=(f|_S)^*(\gamma |_T)$ and so $\beta _T:=\beta (S/T,B_S+\beta _S)$ is nef by induction on the dimension. Since $\beta _Z|_T=\beta _T$, then $\beta _Z|_W$ is nef and hence pseudo-effective for any $W\subset T$.
Finally, suppose that $W\not \subset {\rm Supp}(B_Z)$. Let $X_W=X\times _ZW$, then we expect that $(X_W,B_W+\beta _W)$ is a generalized pair with mild singularities and $K_{X_W}+B_W+\beta _W=(f|_W)^*(\gamma |_{W})$ so that by induction on the dimension $\beta _Z|_W=\beta (X_W/W,B_W+\beta _W)$ is nef.
Of course there are many technical issues that we will have to address in the proof.
\begin{proof}[Proof of Theorem \ref{t-gp}] 
By Proposition \ref{p-BPs}, the boundary b-divisor $\mathbf B ^Z$ for $(X,B+\bbeta /Z)$  descends to a model $Z'$, and hence so does the moduli part $\bbeta^Z$.
It is well known (and easy to see) that $(Z',\mathbf B ^Z_{Z'}) $ is sub-klt (sub-lc).
Thus, it suffices to show that, after possibly replacing $Z'$ by a higher model, $\bbeta ^Z_{Z'}$ is nef.

Let $f'':X''\to Z''$ and $f':X'\to Z'$ be the bi-meromorphic models of $f:X\to Z$ defined in Theorem \ref{t-aws} and let $K_{X''}+B''+\bbeta _{X''}$ and $K_{X'}+B'+\bbeta _{X'}$ be the log pull backs of $K_{X}+B+\bbeta _{X}$. We may assume that $\bbeta$ descends to $X'$. It suffices to show that $\beta _{Z'}$ is nef where 
$\beta _{Z'}:=\bbeta ^Z_{Z'}$ is the moduli part of $(X'/Z',B'+\bbeta _{X'})$.
Let $B_{Z''}$ and $\beta _{Z''}:=\bbeta ^Z_{Z''}$ be the boundary and moduli parts of $(X''/Z'',B''+\bbeta _{X''})$.
Since $\mu _{Z'}:Z''\to Z'$ is finite, then $\beta _{Z''}=\mu _{Z'}^*\beta _{Z'}$
 Thus it suffices to show that $\beta _{Z''}$ is nef (see eg. \cite[Lemma 2.38]{DHP22}).

Let $B''=B^+-B^-$ where $B^+,B^-$ are effective with no common components. 
We may assume that $B_{Z''}$ is the reduced divisor supported on $Z''\setminus U_{Z''}$, 
and we let $F$ be the sum of the components of $(f'')^*(B_{Z''})$ that are not contained in $\lfloor B^+\rfloor$.
Pick $0<\delta \ll 1$, $H''$ a relatively ample divisor and let $\psi: X''\dasharrow \bar X$
be a relative minimal model for $({X''}, B^++\delta F+\bbeta _{X''}+\epsilon H'')$ where $0<\epsilon \ll \delta$. 
We note here that 
\[K_{X''}+B^++\delta F+N_\epsilon\equiv K_{X''}+B^++\delta F+\bbeta _{X''}+\epsilon H''\equiv B^-+\delta F,\] where
$N_\epsilon :=-(K_{X''}+B'')+\epsilon H''\equiv _{Z''}\bbeta _{X''}+\epsilon H''$ is a relatively ample $\Q$-divisor and hence 
the corresponding minimal model exists by \cite{Fuj22}.

We claim that $\psi_*(B^-+\delta F)=0$. To see this, note that since for general $z\in Z$, $K_{X_z}+B_z+\bbeta _{X}|_{X_z}\equiv 0$, then 
$K_{X_{z''}}+B^+_{z''}+\bbeta _{X''}|_{X''_{z''}}\equiv B^-_{z''}$ where $z''$ denotes a preimage of $z$ on $Z''$. As $B^-_{z''}$ is exceptional for $X''_{z''}\to X_z$ and $F_{z''}=0$, then
by standard arguments $\psi$ will contract every component of $B^-_{z''}$. 
In particular $K_{\bar X_{z''}}+\bar B_{z''}+\bbeta _{\bar X}|_{\bar X_{z''}}\equiv 0$ and so
$K_{\bar X_{z''}}+\bar B_{z''}+\bbeta _{\bar X}|_{\bar X_{z''}}+\epsilon \bar H|_{\bar X_{z''}}$ is nef for all $0\leq \epsilon \ll 1$. 
Since $\bar B^-+\bar F$ is vertical, up to adding a vertical divisor which is $\equiv _{Z''}0$, we may assume that 
for any divisor $P$ in the support of $\bar f_*(\bar B^-+\bar F)$, every divisor $Q$ dominating $P$ has ${\rm mult}_Q(\bar B^-+\bar F)\geq 0$ and ${\rm mult}_Q(\bar B^-+\bar F)= 0$ 
for one such component. By \cite[Lemma 2.9]{Lai11}, if ${\rm mult}_Q(\bar B^-+\bar F)> 0$ for one such $Q$, then there is a component $Q'$ dominating $P$ which is contained in ${\mathbf B}_-(\bar B^- +\delta \bar F/Z'')$. Since $\bar B^-+\delta \bar F+\epsilon H''$ is nef over $Z''$ (for $0<\epsilon \ll 1$), it follows that no such divisor exists and hence $\bar B^-+\bar F$ is  exceptional (over $Z''$).
Since $X''\to Z''$ is weakly semistable, $\bar B^-+\delta \bar F=0$ 
and so $K_{\bar X}+\bar B+\bbeta _{\bar X}\equiv \bar f ^*\gamma ''$. Note that $\bar B\geq \bar f^*B_{Z''}$.
 
By \cite[Theorem 2.36]{DHP22}, it suffices to show that for any subvariety $W\subset Z''$ with normalization $W^\nu\to W$, then  $\bbeta _{Z''}|_{W^\nu}$ is pseudo-effective.  
Let $T\subset Z''$ be any component of $B_{Z''}$. 
There exists a component $S$ of $ (\bar B )^{=1}$ such that ${\rm mult}_S(\bar B)=1$ and $S$ dominates $T$.
We replace $S$ by a minimal stratum of $ (\bar B )^{=1}$ dominating $T$. Then $S$ is normal and by adjunction $(K_{\bar X}+\bar B+\bbeta _{\bar X})|_S=K_{S}+B_S+\beta _S$ where $(S,B_S+\beta _S)$ is generalized dlt,
$(K_{Z''}+B_{Z''})|_T=K_T+B_T$ where $(T,B_T)$ is log canonical, $B_T=(B_{Z''}-T)|_T$. Then \[K_S+B_S+\beta _S=g^*(K_{Z''}+B_{Z''}+\beta _{Z''})|_T=g^*(K_T+B_T+\beta _{T})\] where $\beta _S=\bbeta _{\bar X}|_S$, $g=\bar f|_S$. By standard arguments (see for example the claim in the proof of Theorem 1.1 of \cite{FG14}), we may assume that $g$ has connected fibers. Since $(K_{Z''}+B _{Z''})|_T=K_T+B_T$, it follows that $\beta _{Z''}|_T=\beta _T$ is the mobile part of $(S,B_S+\beta _S/T)$. By induction on the dimension, $\beta _T$ is nef. 

Thus, we may now assume that $W$ is not contained in $Z''\setminus U_{Z''}$.
Replacing $\bar B$ by $\bar B-\bar f^*B _{Z''}$ we will also assume that $ B _{Z''}=0$. 
It suffices to show that  $\bbeta _{Z''}|_{\tilde W}$ is pseudo-effective where $\mu _W:\tilde W\to W$ is a resolution or equivalently that $(K_{\tilde X''/\tilde Z}+B_{\tilde X''}+\bbeta _{\tilde X''})|_{\tilde X''_{\tilde W}}$ is pseudo-effective (see Theorem \ref{mp}.ii). We may assume that $\tilde W\subset \tilde Z$ where $\mu _{Z''}:\tilde Z\to Z''$ is a resolution. We write $\tilde X=\bar X\times _{Z''}\tilde Z$, $\tilde X''= X''\times _{Z''}\tilde Z$ 
and $\tilde X_{\tilde W}=\bar X\times _{Z''}\tilde W$,  $\tilde X''_{\tilde W}=X''\times _{Z''}\tilde W$.
We let $\mu _{\bar X}^*(K_{\bar X/Z''}+\bar B+\bbeta_{\bar X})=K_{\tilde X/\tilde Z}+\tilde B+\bbeta_{\tilde X}$ 
and
$K_{\tilde X''/\tilde Z}+\tilde B''+\bbeta_{\tilde X''}=\tilde \psi ^*(K_{\tilde X/\tilde Z}+\tilde B+\bbeta_{\tilde X})$ where $\mu _{\bar X}:{\tilde X}\to \bar X$ and $\tilde \psi:\tilde X''\dasharrow \tilde X$.
By our construction $\bbeta _{X''}$ is nef over $Z''$ and $\bbeta =\overline {\bbeta_{X''}}$. 
Note that $\tilde X''_{\tilde W}$ is not contained in the support of $\tilde B''$ and that $\tilde X''_{\tilde W}$ and $\tilde X_{\tilde W}$ are normal. This last claim follows since $\tilde X''_{\tilde W}$ and $\tilde X_{\tilde W}$ are regular in codimension 1 and all corresponding fibers are S$_2$. 
We now let
\[(K_{\tilde X''/\tilde Z}+B_{\tilde X''}+\bbeta _{\tilde X''})|_{\tilde X''_{\tilde W}}=K_{\tilde X''_{\tilde W}/\tilde W}+B_{\tilde X''_{\tilde W}}+\bbeta _{\tilde X''_{\tilde W}}\]
where $K_{\tilde X''_{\tilde W}/\tilde W}=K_{\tilde X''/\tilde Z}|_{\tilde X''_{\tilde W}}$ (see \cite[Theorem 2.68]{Kollar22}), $B_{\tilde X''_{\tilde W}}=B_{\tilde X''}|_{\tilde X''_{\tilde W}}$ and $\bbeta _{\tilde X''_{\tilde W}}=\bbeta _{\tilde X''}|_{\tilde X''_{\tilde W}}$ is nef. 
Similarly we let $(K_{\tilde X/\tilde Z}+B_{\tilde X}+\bbeta _{\tilde X})|_{\tilde X_{\tilde W}}=K_{\tilde X_{\tilde W}/\tilde W}+B_{\tilde X_{\tilde W}}+\bbeta _{\tilde X_{\tilde W}}$.
Note that $K_{\tilde X''_{\tilde W}/\tilde W}+B_{\tilde X''_{\tilde W}}+\bbeta _{\tilde X''_{\tilde W}}=\psi _{\tilde W}^*(K_{\tilde X_{\tilde W}/\tilde W}+B_{\tilde X_{\tilde W}}+\bbeta _{\tilde X_{\tilde W}})$ and so we have a generalized pair \[\psi _{\tilde W}:(\tilde X''_{\tilde W},B_{\tilde X''_{\tilde W}}+\bbeta _{\tilde X''_{\tilde W}})\to (\tilde X_{\tilde W},(\psi _{\tilde W})_*(B_{\tilde X''_{\tilde W}}+\bbeta _{\tilde X''_{\tilde W}})).\] Since $(\psi _{\tilde W})_*(B_{\tilde X''_{\tilde W}})\geq B_{\tilde X_{\tilde W}}$, 
then $(\tilde X_{\tilde W},(\psi _{\tilde W})_*(B_{\tilde X''_{\tilde W}}+\bbeta _{\tilde X''_{\tilde W}}))$ is generalized log canonical over an open subset of $\tilde W$.
By Theorem \ref{t-gue}, $K_{\tilde X''_{\tilde W}/\tilde W}+B_{\tilde X''_{\tilde W}}+\bbeta _{\tilde X''_{\tilde W}}$ is pseudo-effective. This concludes the proof.
\end{proof}

\begin{remark} We believe that if moreover $\bbeta$ admits a smooth positive representative, then so does $\bbeta ^Z$ (cf. Theorem \ref{lelong}).
\end{remark}
\subsection{Adjunction}
In what follows, for ease of exposition, we will denote  a generalized pair $(X/Z,B+\bbeta )$ simply by $(X,B+\bbeta )$.
\begin{defn}\label{d-ad}Suppose that $(X,B+\bbeta )$ is a generalized pair and $S$ is a component of $B$ of coefficient 1 with normalization $\nu :S^\nu \to S$, then we define a generalized pair  \[K_{S^\nu}+B_{S^\nu} +\bbeta ^{S^\nu}=(K_X+B+\bbeta )|_{S^\nu}\] as follows.
Let $f:X'\to X$ be a log resolution of  $(X,B+\bbeta )$ so that $X'$ is smooth, $B'$ has simple normal crossings, $\bbeta_{X'}$ is nef, $\bbeta$ descends to $X'$,  and $K_{X'}+B'+\bbeta _{X'}\equiv f^*(K_X+B+\bbeta _X )$. Let $S'=f^{-1}_*S$, then $S'$ is smooth and
by the usual adjunction for sub-klt pairs, we can write $K_{S'}+B_{S'}=(K_{X'}+B')|_{S'}$ where $B_{S'}=(B'-S')|_{S'}$ has simple normal crossings. We may assume that $\beta'$ is smooth and we let $\bbeta ^{S^\nu}=\overline{\beta  _{S'}}$ where $\beta  _{S'}:=\beta'|_{S'}$ is the induced nef (1,1) form. Notice that  $[K_{S'}+B_{S'}+\beta  _{S'}]=g^*(\gamma|_{S^\nu})$ where
$g:S'\to S^\nu$ is the induced morphism and $\gamma =[K_X+B+\beta ]$. 
Let $K_{S^\nu}+B_{S^\nu} +\beta _{S^\nu}:=g_*(K_{S'}+B_{S'}+\beta _{S'})$, then $g: ({S'},B_{S'}+\beta _{S'})\to (S,B_{S^\nu} +\beta _{S^\nu})$ defines a generalized pair; equivalently $(S',B_{S'}+\bbeta ^{S'})$ is a generalized pair.
\end{defn}
If $K_X+B$ is $\mathbb R$-Cartier, then we will write $K_{X'}+B^\sharp=f^*(K_X+B)$ and $\beta _{X'}=f^*\beta_X -E$ where $E\geq 0$ is exceptional.
Note however that $E_{S'}:=E|_{S'}$ may not be $g$-exceptional where $g:S'\to S^\nu$.
If $K_{S^\nu}+B_{S^\nu} =(K_X+B)|_{S^\nu}$ is the usual adjunction, then \[(K_X+B+\beta_X )|_{S^\nu}=K_{S^\nu}+B_{S^\nu} + g_*E_{S'}+g_*\beta_{S'}.\] 
By Lemma \ref{l-sfs}, if $(X,B+\bbeta)$ is generalized lc in codimension 2, then $K_X+B$ is $\mathbb R$-Cartier in codimension 2 and hence the above formula can always be used to compute $(K_X+B+\bbeta )|_{S^\nu}$.

It is clear that if $(X,B+\bbeta )$ is generalized log canonical (resp. generalized plt) then $({S^\nu},B_{S^\nu} +\beta _{S^\nu})$ is generalized log canonical (resp. generalized klt). 
We now will verify that the reverse implication also holds. The following statement is often referred to at {\it plt inversion of adjunction}.

\begin{theorem}\label{t-pltinv}Let $(X,B+\bbeta )$ be a generalized pair and $S$ a component of $B$ of coefficient 1 with normalization $\nu :S^\nu \to S$. Then
    $(X,B+\bbeta )$ is generalized plt on a neighborhood of $S$ iff  $S$ is normal and $({S},B_{S} +\bbeta ^{S})$ is generalized  klt.
\end{theorem}
\begin{proof}
   Let $f:X'\to X$ be a log resolution of the generalized pair $(X,B+\bbeta )$. We may assume that $f$ is a projective morphism. We write $K_{X'}+B'+\bbeta_{X'}=f^*(K_X+B+\bbeta_X)$ and $S'=f^{-1}_*S$. 
   We have a short exact sequence
   \[0\to \OO _{X'}(-\lfloor B'\rfloor)\to \OO _{X'}(-\lfloor B'\rfloor +S')\to \OO _{S'}(-\lfloor B'\rfloor+S')\to 0.\]
   Since, $-\lfloor B'\rfloor\equiv _X K_{X'}+\{B'\}+\bbeta_{X'}$ where $(X',\{B'\})$ is klt and $\bbeta_{X'}$ is nef (and big) over $X$, then $R^1f_* \OO _{X'}(-\lfloor B'\rfloor)=0$ and so we have a surjection
   \[ \phi:f_* \OO _{X'}(-\lfloor B'\rfloor +S')\to f_* \OO _{S'}(-\lfloor B'\rfloor+S').\] 
   
   If $(X,B+\bbeta )$ is generalized plt (on a neighborhood of $S$), then $ -\lfloor B'\rfloor +S'\geq 0$ is effective and exceptional (over a neighborhood of $S$) and hence  $\lfloor B_{S'}\rfloor =(\lfloor B'\rfloor-S')|_{S'}\leq 0$ so that 
   $(S^\nu, B_{S^\nu}+\bbeta ^S)$ is generalized klt. We also have that $f_* \OO _{X'}(-\lfloor B'\rfloor +S')=\OO _X$ and so $\phi$ factors through 
   \[ f_* \OO _{X'}(-\lfloor B'\rfloor +S')=\OO _X\to \OO _S\subset \nu _* \OO _{S^\nu}=  f_*\OO _{S'}\subset f_* \OO _{S'}(-\lfloor B'\rfloor +S')\] and therefore $\OO _S= \nu _*\OO _{S^\nu}$ and $S$ is normal.
   
   If $(S, B_{S}+\beta _{S})$ is generalized klt (and in particular $S$ is normal), then $0\leq -\lfloor B_{S'}\rfloor =(-\lfloor B'\rfloor+S')|_{S'}$ and $f_* \OO _{S'}(-\lfloor B'\rfloor+S')=\OO _S$, and so we have a surjection $f_* \OO _{X'}(-\lfloor B'\rfloor +S')\to \OO _S$. Since $f_* \OO _{X'}(-\lfloor B'\rfloor +S')\subset \OO _X$, it follows that $f_* \OO _{X'}(-\lfloor B'\rfloor +S')= \OO _X$ on a neighborhood of $S$ and so $-\lfloor B'\rfloor +S'\geq 0$ over a neighborhood of $S$, i.e. $(X,B+\bbeta )$ is generalized plt on a neighborhood of $S$.
  
\end{proof}
The following result is known as {\it log canonical inversion of adjunction}. In the usual pair setting, it was first addressed in \cite{Kawakita07} and refined in\cite{Hacon14} and \cite{Fil20}. The following proof is based on the ideas of \cite{Hacon14}.
\begin{theorem}\label{t-lcinv}Let $(X,B+\bbeta )$ be a generalized pair and $S$ a component of $B$ of coefficient 1 with normalization $\nu :S^\nu \to S$. Then
    $(X,B+\bbeta )$ is generalized lc on a neighborhood of $S$ iff   $({S^\nu},B_{S^\nu} +\bbeta ^{S^\nu})$ is generalized  lc.
\end{theorem}

\begin{proof}
Following the arguments above, it is easy to see that if $(X,B+\bbeta )$ is generalized lc on a neighborhood of $S$, then  $({S^\nu},B_{S^\nu} +\bbeta ^{S^\nu})$ is generalized  lc. Thus it suffices to show that if $({S^\nu},B_{S^\nu} +\bbeta ^{S^\nu})$ is generalized  lc, then $(X,B+\bbeta )$ is generalized lc on a neighborhood of $S$.
The question is local over $X$ and so we may assume that $X$ is a relatively compact Stein variety.

Let $\mu:Y\to X $ be a generalized dlt model given by Theorem \ref{t-dltmodel} so that $\mu$ is projective, $Y$ is $\Q$-factorial, all exceptional divisors have discrepancy $a\leq -1$, $(Y, B'_Y+\bbeta )$ is generalized dlt, and $B'_Y:=\mu ^{-1}_*B+{\rm Ex}(\mu)\leq B_Y$ where $K_Y+B_Y+\bbeta _Y=\mu ^*(K_X+B+\bbeta_X)$.
Let $S_Y:=\mu ^{-1}_*S$, $B'_Y=S_Y+\Gamma$ and $B_Y=S_Y+\Gamma+\Sigma$. Let $H$ be a sufficiently ample divisor and run the $K_Y+S_Y+\Gamma+\bbeta _Y$ mmp with scaling of $H$.
\begin{claim} There is a sequence of flips and contractions $\phi_i:Y_i\dasharrow Y_{i+1}$ and real numbers $s_0=1$, $s_i\geq s_{i+1}\geq 0$ such that $K_{Y_i}+S_i+\Gamma _i+\bbeta _{Y_i}+sH_i$ is nef over $X$ for $s_i\geq s\geq s_{i+1}$. 
If the minimal model program terminates, then we may assume that $s_{n+1}=0$, otherwise the we have $\lim s_i=0$.\end{claim}
\begin{proof} For any $s>0$, $\bbeta _Y+sH\equiv _X-(K_Y+B_Y)+sH$ is an ample $\R$-divisor and hence\[S_Y+\Gamma +\bbeta _Y+sH\equiv_X \Delta _s \] where $(Y,\Delta _s)$ is klt.
By \cite{DHP22} or \cite{Fuj22}, for any $0<\epsilon \ll 1$, we can run the $K_Y+\Delta _\epsilon $ mmp over $X$ with scaling of $(1-s)H$. The claim now follows easily.
\end{proof}
We may assume that for $i\gg i_0$, all $\phi _i$ are flips. 
\begin{claim} For any $t>0$,  there is an $\R$-divisor $\Theta _t\equiv _X \Gamma +\bbeta _Y +tH$ such that $(Y,S_Y+\Theta _t)$ is plt.\end{claim}
\begin{proof} Since $H$ is ample and $\bbeta_Y$ is nef, then $\bbeta _Y+tH\equiv _X-(K_Y+B_Y)+tH$ is an ample $\R$-divisor and hence  \[S_Y+\Gamma +\bbeta _Y+tH\equiv_X S_Y+\Theta _t \] where $(Y,S_Y+\Theta _t)$ is plt.
\end{proof}
If $t\leq s_i$, then $(Y_i, S_i+\Theta _{t,i})$ is plt and in particular $S_i$ is normal. Suppose that $\Sigma _i\cap S_i\ne \emptyset$, then 
\[\mu _i^*(K_X+B+\bbeta_X )|_{S_i}=(K_{Y_i}+S_i+\Gamma _i+\Sigma _i+\bbeta _{Y_i})|_{S_i}=K_{S_i}+{\rm Diff}_{S_i}(\Gamma _i+\Sigma _i)+\bbeta _{Y_i}|_{S_i}\]
where $K_{S_i}+{\rm Diff}_{S_i}(\Gamma _i+\Sigma _i)$ is not log canonical as ${\rm mult}_P(\Gamma _i+\Sigma _i)>1$, and this implies 
that ${\rm mult}_Q{\rm Diff}_{S_i}(\Gamma _i+\Sigma _i)>1$ where $Q$ is any component of $S_i\cap P$. But then $K_{S_i}+{\rm Diff}_{S_i}(\Gamma _i+\Sigma _i)+\bbeta _{Y_i}|_{S_i}$ is not generalized log canonical and so neither is $({S^\nu},B_{S^\nu} +\bbeta ^{S^\nu})$.
Therefore $\Sigma _i\cap S_i= \emptyset$ for all $i\geq 0$.

Fix $m>0$ such that $m\Sigma$ is an integral divisor and $s_i>\frac 1 m \geq s_{i+1}$. Then
\[H_i-m\Sigma _i -S_i\equiv _X K_{Y_i}+\Theta _{\frac 1 m , i}+(m-1)(K_{Y_i}+S_i+\Gamma _i+\frac 1 m H_i+\bbeta _{Y_i})\]
where $K_{Y_i}+\Theta _{\frac 1 m , i}$ is klt and $K_{Y_i}+S_i+\Gamma _i+\frac 1 m H_i+\bbeta _{Y_i}$ is nef and big over $X$ so that 
$R^1(\mu _i)_* \OO _{Y_i}(H_i-m\Sigma _i-S_i)=0$ and hence we have a surjection
\[(\mu _i)_* \OO _{Y_i}(H_i-m\Sigma _i)\to (\mu _i)_* \OO _{S_i}(H_i-m\Sigma _i)=(\mu _i)_* \OO _{S_i}(H_i). \]
Note that as $Y_{i_0}\dasharrow Y_i$ is a small bimeromorphic map, then the sheaves
\[ (\mu _i)_* \OO _{Y_i}(H_i-m\Sigma _i)=(\mu _{i_0})_* \OO _{Y_{i_0}}(H_{i_0}-m\Sigma _{i_0})\subset (\mu _{i_0})_* \OO _{Y_{i_0}}(H_{i_0})\]
are contained in $\mathcal I _V\cdot (\mu _{i_0})_* \OO _{Y_{i_0}}(H_{i_0})$ for $m\gg 0$ where $V=\mu _{i_0}(\Sigma _{i_0})$ and if $V\cap S\ne \emptyset$, then this contradicts the above surjection.
Therefore $\mu _{i_0}(\Sigma _{i_0})\cap S=\emptyset$ and so $(X,B+\bbeta )$ is generalized log canonical on a neighborhood of $S$.
  
\end{proof}
\begin{proof}[Proof of Theorem \ref{t-inv}] Immediate from Theorems \ref{t-pltinv} and \ref{t-lcinv}.
\end{proof}
\begin{proposition}\label{t-normallcc}
    Let $(X,B+\bbeta)$ be a generalized log canonical pair and $(X,B_0+\bbeta _0)$ a generalized klt pair. If $V$ a minimal log canonical center of $(X,B+\bbeta)$,
    then $V$ is normal.
\end{proposition}
\begin{proof}
The question is local on $X$ and hence we may assume that $X$ is relatively compact and Stein.
    By the usual tie breaking arguments, we may assume that there is a unique log canonical place for $(X,B+\bbeta )$ and that this place dominates
    $V$. 
    Thus there is a log resolution of $(X,B+\bbeta )$, $f:X'\to X$ such that $K_{X'}+B'+\bbeta_{X'} =f^*(K_X+B+\beta)$ where $B'=S'+\Delta '$ and $S'$
    is the unique LC place for $(X,B+\bbeta )$ so that $\lfloor \Delta '\rfloor \leq 0$ is an exceptional divisor.
    By the proof of Theorem \ref{t-pltinv}, we have a surjection \[ \phi:f_*\OO _{X'}(-\lfloor \Delta '\rfloor )\to f_*\OO _{S'}(-\lfloor \Delta '\rfloor ).\] Since $-\lfloor \Delta '\rfloor$ is effective and exceptional, then $f_*\OO _{X'}(-\lfloor \Delta '\rfloor )=\OO _X$.
    Since $S'$ is smooth, $S'\to V$ factors through the normalization $\nu :V^\nu\to V$ and so $\phi$ factors via $f_*\OO _{X'}(-\lfloor \Delta '\rfloor )=\OO _X\to \OO _V$
    and the natural inclusions \[\OO _V\subset  \nu _*\OO _{V^\nu}\subset f_*\OO _{S'}\subset f_*\OO _{S'}(-\lfloor \Delta '\rfloor ).\] It follows that $\OO _V\cong \nu _*\OO _{V^\nu}$ and hence $V$ is normal.
\end{proof}
\begin{theorem} Let $(X,B+\bbeta)$ be a generalized log canonical pair then ${\rm nklt}(X,B+\bbeta)$ is seminormal.
\end{theorem}
\begin{proof} The corresponding result for log canonical pairs is contained in \cite{Amb98}. We will follow the approach of \cite{Kollar07}.
Let $\nu :X'\to X$ be a log resolution of $(X,B+\bbeta)$ and write $K_{X'}+B'+\bbeta _{X'}=\nu ^*(K_X+B+\bbeta _X)$. If $B'=S-A+\{ B'\}$ where $S,A$ are effective Weil divisors without common components, then $S$ is seminormal as it is a divisor with simple normal crossings and hence $f|_S:S\to W$ factors through the seminormilazation $h:W^{\rm sn}\to W$ via $g:S\to W^{\rm sn}$. We have a short exact sequence
\[0\to \OO _{X'}(A-S)\to \OO _{X'}(A)\to \OO _S(A|_S)\to 0.\]
Since $A-S\equiv _X K_{X'}+\{B'\}+\bbeta _{X'}$ and $\bbeta _{X'}$ is nef and big over $X$, it follows that $R^1\nu _* \OO _{X'}(A-S)=0$ and hence
$\nu _* \OO _{X'}(A)\to \nu _*\OO _S(A|_S)$. Since $A$ is $\nu $-exceptional, $\nu _* \OO _{X'}(A)=\OO _X$ and hence $\nu _*\OO _S(A|_S)=\OO _W$.
But $\nu _*\OO _S(A|_S)\supset h_*g_*\OO _S=h_*\OO _{W^{\rm sn}}$ and so $h_*\OO _{W^{\rm sn}}=\OO _W$ i.e. $h:W^{\rm sn}\to W$ is an isomorphism.
\end{proof}



We will now prove that Theorem \ref{t-Ksad} follows from Theorem \ref{t-gcbf}.
\begin{proof}[Proof of Theorem \ref{t-Ksad}] 
We may assume that $W$ has codimension $\geq 2$.
By Proposition \ref{t-normallcc}, $W$ is normal and by a standard tie breaking argument, we may assume that there is a unique log canonical place 
for an auxiliary pair $(X,B^\sharp+\bbeta^\sharp)$. Let $f:X'\to X$ be the corresponding dlt model (Theorem \ref{t-dltmodel}), then $f$ has a unique exceptional divisor $S'$ and $f^*(K_X+B+\bbeta _X)=K_{X'}+S'+B'+\bbeta _{X'}$ where $(X',S'+B'+\bbeta ')$ is plt. Note that by the proof of Proposition \ref{t-normallcc}, $S'\to W$ has connected fibers.
By Theorem \ref{t-pltinv}, $(S', B_{S'}+\bbeta _{S'})$ is generalized klt where $K_{S'}+B_{S'}+\bbeta _{S'}=(K_{X'}+S'+B'+\bbeta _{X'})|_{S'}$.
Let $\nu :S'\to W$, then $K_{S'}+B_{S'}+\bbeta _{S'}=\nu ^*((K_X+B+\bbeta _X)|_W)$ and so by Theorem \ref{t-gcbf}, $(K_X+B+\bbeta _X)|_W\equiv K_W+B_W+\bbeta _W$ where $(W,B_W+\bbeta _W)$ is generalized klt.
\end{proof}
\begin{remark}\label{r-ad} The above arguments show that if $(X,B+\bbeta)$ and $(X,B'+\bbeta ')$ are generalized pairs, $V\subset X$ is a subvariety and $U\subset X$ an open subset such that $(U,(B+\bbeta )|_U)$ is generalized lc, $(U,(B'+\bbeta ')|_U)$ is generalized klt and $V\cap U$ is a minimal log canonical center of $(U,(B+\bbeta )|_U)$, then $(K_X+B+\bbeta)|_{V^\nu}=K_{V^\nu}+B_{V^\nu}+\bbeta ^{V^\nu}$ where $V^\nu \to V$ is the normalization and $(V^\nu, B_{V^\nu}+\bbeta ^{V^\nu})$ is a generalized pair.

\end{remark}


\section{Cone Theorem for generalized klt pairs}

The results in this section are inspired by \cite{CH20}. They suggest that one of the main obstructions to the higher dimensional minimal model program for K\"ahler varieties is  Conjecture \ref{c-BDPP13}.

\begin{proposition}\label{p-ray}Assume Conjecture \ref{c-BDPP13} in dimension $\leq n-1$.  Let $X$ be a compact $\Q$-factorial K\"ahler $n$-fold such that $(X,B+\bbeta)$ is generalized klt, $K_X+B+\bbeta _X$ is pseudo-effective, and $\omega $ be a K\"ahler form such that $\alpha:=[K_X+B+\bbeta _X+\omega]$ is nef and big but not K\"ahler. Then there is an $\alpha $-trivial  rational curve $C$ such that $0<-(K_X+B+\bbeta _X)\cdot C=\omega \cdot C\leq 2\dim X$.
\end{proposition}
\begin{proof} By Proposition \ref{pro:null-non-kahler} (see also \cite[Theorem 1.1]{CT15} and \cite[Theorem 3.17]{Bou04} in the smooth case) the restricted non-Kähler locus
${ E_{nK}^{as}} (\alpha)$ coincides with the null-locus  ${\rm Null}(\alpha)$, and 
there exists a K\"ahler current $\eta$ with weak analytic singularities in the class $\alpha$ such that
the Lelong set coincides with ${ E_{nK}^{as}} (\alpha)$. 
Since $\alpha $ is not K\"ahler, then ${\rm Null}(\alpha)$ has a positive dimensional component.
Let $Z$ be a maximal dimensional irreducible component of ${\rm Null}(\alpha)$ and $c$ be the log canonical threshold of $(X,B+\bbeta _X )$ with respect to $\eta$ on a neighborhood of general points of $Z$. This means that 
if we pick a log resolution $\nu :X'\to X$ such that $K_{X'}+B_{X'}+\bbeta _{X'}=\nu ^*(K_X+B+\bbeta _X)$ where $\bbeta _{X'}$ is nef and
$\nu ^*\eta=\eta' +F$ where $F$ is an effective $\R$-divisor, $\eta'\geq 0$ and $F+B_{X'}$ has simple normal crossings, then $Z$ is an irreducible component of $\nu ((B_{X'}+cF)^{=1})$ and $Z$ is not contained in $\nu ((B_{X'}+cF)^{>1})$. If $\boldsymbol{\eta}=\overline {\eta '}$, then $Z$ is a generalized log canonical center of the generalized pair   
$(X,B+c\nu _*F+\bbeta +c\boldsymbol{\eta})$.
By Remark \ref{r-ad}, \[(K_X+B+c\nu _*F+\bbeta  +c\boldsymbol{\eta})|_{Z^\nu}=K_{Z^{\nu}}+B_{Z^{\nu}}+\gamma _{Z^{\nu}}\] where $({Z^{\nu}},B_{Z^{\nu}}+\gamma _{Z^{\nu}})$ is a generalized pair.

By assumption we have $k:=\nu _{\rm num}(\alpha |_{Z^\nu})<\dim Z$ so that $(\alpha |_{Z^\nu})^k\not \equiv  0$ and $(\alpha |_{Z^\nu})^{k+1}\equiv   0.$
But then \[ (K_{Z^{\nu}}+B_{Z^{\nu}}+\gamma _{Z^{\nu}})\cdot \alpha _{Z^{\nu}}^k\cdot \omega _{Z^{\nu}}^{\dim Z-k-1}=- \alpha _{Z^{\nu}}^k\cdot \omega _{Z^{\nu}}^{\dim Z-k}< 0 \]
where $\alpha _{Z^{\nu}}=\alpha |_{Z^{\nu}}$ and $\omega _{Z^{\nu}}=\omega |_{Z^{\nu}}$.
Since $B_{Z^{\nu}}\geq 0$ and $\gamma _{Z^{\nu}}$ is pseudo-effective, then $K_{Z^{\nu}}$ is not pseudo-effective and hence neither is $K_{Z'}$ for any resolution $Z'\to Z^\nu$.

Consider now the MRC fibration $Z'\to Y$, which is non-trivial as $K_{Z'}$ is not pseudo-effective. Passing to a higher model we may assume that it is a morphism with general fiber $F$. Note that $F$ is rationally connected and hence $h^2(\OO _F)=0$ and so $F$ is algebraic.
Arguing as above, for any $\epsilon>0$ we have  \[ (K_{Z'}+B_{Z'}+\gamma _{Z'}+(1-\epsilon)\omega _{Z'}+t\alpha _{Z'})\cdot \alpha _{Z'}^k\cdot \omega _{Z'}^{\dim Z-k-1}=\]\[(K_{Z^{\nu}}+B_{Z^{\nu}}+\gamma _{Z^{\nu}}+(1-\epsilon)\omega _{Z^\nu})\cdot \alpha _{Z^{\nu}}^k\cdot \omega _{Z^{\nu}}^{\dim Z-k-1}=-\epsilon  \alpha _{Z^{\nu}}^k\cdot \omega _{Z^{\nu}}^{\dim Z-k}<0,\]
where $\omega _{Z'}=\omega |_{Z'}$ and $\alpha _{Z'}=\alpha |_{Z'}$.
Since $\gamma _{Z'}$ is pseudo-effective and  $B_{Z'}^{<0}$ is $Z'\to Z^\nu$ exceptional, it follows that $(K_{Z'}+(1-\epsilon)\omega _{Z'}+t\alpha _{Z'})\cdot \alpha _{Z'}^k\cdot \omega _{Z'}^{\dim Z-k-1}<0$ 
and hence $K_{Z'}+(1-\epsilon)\omega _{Z'}+t \alpha_{Z'}$ is not pseudo-effective 
for any $t >0$. Since $Y$ is not uniruled, $K_{Y}$ is pseudo-effective and hence by \cite[Theorem 5.2]{CH20},   $K_F+(1-\epsilon) \omega_{F}+t \alpha_{F}$ is also not pseudo-effective for any $t >0$ and in particular $\alpha_{F}$ is not big. By the cone theorem, there are finitely many $K_F+(1-\epsilon) \omega_{F}$ negative extremal rays, and there is a non-empty finite collection of $K_F+(1-\epsilon) \omega_{F}$ negative extremal rays that are $\alpha$ trivial.
Let $\eta :F\to \bar F$ be the induced non-trivial morphism contracting this face. Then $ \alpha_{F}=\eta ^*\alpha_{\bar F}$ where $\alpha_{\bar F}$ is ample on $\bar F$.
If $\eta $ is birational, then $\alpha_{F}$ is big which is a contradiction.
Thus, $\eta$ is of fiber type and hence $F$ is covered by $\alpha_{F}$-trivial rational curves $C$. Note that by bend and break, we may assume that $0<-K_F\cdot C\leq 2\dim F$ and hence $(1-\epsilon)\omega _F\cdot C< 2\dim F$. But then $Z'$ is covered by $\alpha_{F}$-trivial rational curves and finally $Z$ is covered by $\alpha$-trivial rational curves $C$ such that $0<\omega \cdot C< \frac 2{1-\epsilon}\dim X$.
Since these curves belong to finitely many numerical classes, we may assume that $0<\omega \cdot C \leq 2\dim X$.
Finally, we observe that 
\[0<-(K_X+B+\bbeta _X)\cdot C=\omega \cdot C\leq 2\dim X.\]

\end{proof}
\begin{corollary}\label{c-cone} Let $(X,B+\bbeta )$ be a compact K\"ahler $4$-fold generalized klt pair such that $K_X+B$ is pseudo-effective. Then there are at most countably many rational curves $\{\Gamma _i\}_{i\in I}$ such that $-(K_X+B+\bbeta _X)\cdot \Gamma _i\leq 8$ for all $i\in I$ and  \[\overline{\rm NA}(X)=\overline{\rm NA}(X)_{(K_X+B+\bbeta _X) \geq 0}+\sum _{i\in I}\mathbb R ^+[\Gamma _i].\] \end{corollary} \begin{proof} This follows by standard arguments from Proposition \ref{p-ray} (see eg the proof of \cite[Theorem 1.3]{DH23}). \end{proof}


\section{Null loci}
In this section we generalize the main result of \cite{CT15} to the singular setting. To start with, we recall the notion of Lelong number of a closed positive current $T\geq 0$ on a normal complex space, since it plays a crucial role in the formulation of 
Proposition \ref{pro:null-non-kahler} below.

\subsection{Closed positive currents on normal complex spaces}

Let $\Omega\subset \C^n$ be the unit ball, and let $\phi$ be a psh function on $\Omega$. Then 
\[T:= \ddbar \phi\]
defines a closed positive current of $(1,1)$-type on $\Omega$. For example, if we take
$\phi= \log |f|^2$, with $f$ holomorphic, then up to a multiple, $T$ is equal to the 
current of integration along the analytic set $f=0$ (taking the multiplicities into account).
Thus, closed positive $(1,1)$ currents can be seen as natural generalizations of effective divisors. 
\smallskip

\noindent It turns out that the function
\[r\to \sup_{|z|=r}\phi(z) \]
is convex increasing of $\log r$, i.e. if $x:= \log r$, then the function $\displaystyle x\to \sup_{|z|= e^{x}}\phi(z)$
is convex increasing. It therefore follows that the limit
\[\nu(T, 0):= \lim\inf_{z\to 0}\frac{\phi(z)}{\log|z|}\]
exists, and it is called the Lelong number of $T$ at $z=0$. In the example above, this is precisely the multiplicity of the divisor $(f=0)$ at $0$. Moreover, one can see that the equality
\begin{equation}\label{new77}
\nu(T, 0)= \sup\{\nu\geq 0 | \phi(z)\leq \nu\log|z|+ \mathcal O(1)\}
\end{equation}
holds true.
\smallskip

\noindent We collect next a few facts about currents which will be needed later on.

\begin{theorem}\label{gen}\cite{DeBook}
Let $X$ be a complex manifold, and let $T\geq 0$ be a closed positive current of $(1,1)$-type on $X$. 
We consider an analytic hypersurface $A\subset X$, and let $\chi_A$ be the characteristic function 
of $A$. Then the following assertions hold true.
\begin{enumerate}

\item[\rm (1)] The function $a\to \nu(T, a)$ defined on the analytic set $A$ is constant in the complement of an at most countable union of analytic subsets of $A$, and it defines the generic Lelong number of $T$ along $A$.

\item[\rm (2)] The currents $\chi_AT$ and $\chi_{X\setminus A}T$ obtained by multiplying $T$ with the characteristic function of $A$ and its complement, respectively are closed (and of course, positive). Moreover, we have 
\[\chi_AT= \nu(T, A)[A],\]
where $\nu(T, A)$ is the generic Lelong number of $T$ along $A$.
\end{enumerate}
\end{theorem}
\medskip

\noindent Next, consider a surjective map $f:X\to Y$ between two compact complex manifolds. Given a 
closed $(1,1)$ current $T$ on the base $Y$, the pull-back $f^\star T$ is a well-defined, closed current on $X$ (this is not necessarily true for currents of other bi-degrees). The following result clarifies the connection between the Lelong numbers of $T$ and those of its inverse image
$f^\star T$.

\begin{theorem}\label{pull}\cite{Fav99}
Under the assumptions above, there exists a positive constant $C> 0$ such that we have 
\[C\nu(f^\star T, x)\leq \nu(T, y)\leq \nu(f^\star T, x),\]
for all $x\in X$ and $y= f(x)$.
\end{theorem}

\noindent Note that the right-hand side inequality in Theorem \ref{pull} follows immediately form the definition
\eqref{new77}. Also, this sort of comparison inequalities is far from true in case of currents obtained by direct images, i.e. if one wishes to compare the Lelong numbers of 
a current $\Theta$ on $X$ with those of its direct image $f_\star \Theta$. For example, let $f:X\to Y$
be the blow-up of a point $y$ of $Y$. We assume that $Y$ is Kähler; then given any Kähler 
metric $\omega$ on $X$, the direct image $f_\star \omega$ has a positive Lelong number at $y$. 

\noindent Still in this context (i.e. $f$ is the blow-up at a point and $T\geq 0$ is a closed positive current on $Y$), we have the equality
\begin{equation}\label{new78}
f^\star T= \nu(T, y)[E]+ R 
\end{equation}
where $E$ is the exceptional divisor of the blow-up $f$, and $R\geq 0$ is a closed positive current whose
generic Lelong number along $E$ is equal to zero. In particular, we have
\begin{equation}\label{new79}
\chi_Ef^\star T= \nu(T, y)[E]. 
\end{equation}
\medskip

\noindent Let $f:X\to Y$ be a surjective map between compact complex manifolds, and let $\rho$ be a real, $(1,1)$ cohomology class on the target manifold $Y$. 
We have the following well-known remark.

\begin{lemma}\label{lem: inverse}Given any $(1,1)$ closed positive current 
\[T\in f^\star\rho,\] there exists a $(1,1)$ closed positive current $R$ on $Y$ such that 
\begin{equation}\label{new85}
T= f^\star R.\end{equation}\end{lemma}

\begin{proof} The matter is indeed clear: according to our hypothesis there exists an $L^1$ function $\phi$ on $X$ and a smooth representative $a\in \rho$ such that the following equality 
\[T= f^\star(a)+ \ddbar \phi\]
holds. Now the restriction $\displaystyle T|_{X_y}$ of the current $T$ to the general fibers $X_y$ of $f$ is a well-defined, closed positive current. On the other hand, we have
\[T|_{X_y}= \ddbar \phi |_{X_y}\]
which shows that $\displaystyle \phi |_{X_y}$ must be constant on $X_y$. Therefore our assertion follows.\end{proof}
\medskip

\noindent The notion of Lelong number of a closed positive current on a normal space will be needed in order to formulate the main result of this section. We recall it next.
\begin{defn}\label{defn_singl}
\cite{Dem82}.
Let $X$ be a normal complex space, and let $T\geq 0$ be a closed positive $(1,1)$-current on $X$. We consider a positive function $\varphi\in \mathcal C^2(X, \R_+)$, such that $\log \varphi$ is psh and such that $\Supp(T)\cap (\varphi< R)$ is relatively compact in $X$ for all 
$0< R\ll 1$ sufficiently small. The limit
\[\nu(T, \varphi):= \lim_{r\to 0}\frac{1}{(2\pi r)^{2n-2}}\int_{\varphi< r}T\wedge (\ddbar\varphi)^{n-1}\]
is called the Lelong number of $T$ with respect to $\varphi$.
\end{defn}
 
\noindent Let $y\in X$ be an arbitrary point. If we consider an embedding
\begin{equation}\label{eq202}
(X, y)\xhookrightarrow{} (\C^N, 0)
\end{equation}
then the coordinate functions $(z_i)_{i=1,\dots, N}$ on $\C^N$ restricted to $X$ induce a generating system $(g_i)_{i=1,\dots, N}$ of the maximal ideal of the ring $\mathcal O_{X, y}$. The function 
\begin{equation}\label{eq203}
\varphi_y:= \sum_i |g_i|
\end{equation}
is defined on some small open subset $U$ containing $x$, and then the Lelong number of $T$ at $y$ is defined as follows
\begin{equation}\label{eq204}
\nu(T, y):= \nu(T|_U, \varphi_y),
\end{equation}
where $T|_U$ is the restriction of $T$ to $U$, so the RHS is defined as in Definition \ref{defn_singl}.  

\begin{remark} It is not immediate that the limit in Definition \ref{defn_singl} exists, but this is a consequence of the 
Jensen formula
established in \cite{Dem82}, Théorème 3. Moreover, note that  
the Lelong number $\nu(T, y)$ is independent of the embedding \eqref{eq202}, as consequence of Théorème 4 in \emph{loc. cit.}
\end{remark}

\noindent Let $n$ be the dimension of $X$.
By composing the embedding map \eqref{eq202} with a generic linear projection 
on $\C^n$, we obtain a proper, finite map
\begin{equation}\label{eq209}
p:(X, y)\to (\C^n, 0)
\end{equation}
such that $p^{-1}(0)= {y}$. The function
\begin{equation}\label{eq210}
\wt \varphi_y:= \sum |p_j|
\end{equation}
(where $p_j$ are the components of $p$) verifies the inequalities 
\[\varphi_y^M\leq \wt \varphi_y\leq C\varphi_y\]
locally near $x$, for some positive constants $C$ and $M$. We can assume that it holds on the open subset $U$. By the comparison theorem 
 for Lelong numbers (cf. Th\'eor\`eme 4, page 46 in \cite{Dem82}), we have 
\begin{equation}\label{eq211}
\frac{1}{M}\nu(T, \wt \varphi_y)\leq \nu(T, y)\leq \nu(T, \wt \varphi_y).
\end{equation}

\begin{remark}\label{slope}
If $y\in X$ is a regular point, the equality
\[\nu(T, y)= \lim\inf_{z\to y}\frac{\varphi_T}{\log |z-y|}\]
holds, and it provides an alternative definition for the 
Lelong number of $T$ at $y$, as we have already mentioned. Simple examples (\cite{BEGZ10}, Appendix A) show that as soon as 
$y\in X_{\rm sing}$, the relation above is no longer verified in general. 
However, we always have the inequality
$\displaystyle \nu(T, y)\geq \lim\inf_{z\to y}\frac{\varphi_T}{\log |z-y|}$. 
\end{remark}
\smallskip

In connection with these topics, the following result was obtained very recently in \cite{P24}. 
\begin{lemma}\label{slope, III} \cite{P24}
Let $X$ be a normal complex space, and let $T\geq 0$ be a closed positive current on $X$. The following equivalence
\begin{equation}\label{eq212}
\nu(T, y)> 0 \iff \lim\inf_{z\to y}\frac{\varphi_T}{\log |z-y|}> 0
\end{equation}
holds true for any point $y\in X$.
In other words, the Lelong number of $T$ at $y$ and the slope of its potential (i.e. the RHS of \eqref{eq212})
are simultaneously positive or zero. 
\end{lemma}

\begin{remark}\label{slope, II}
Of course, one expects an inequality of type
\begin{equation}\label{eq2212}
\nu(T, y)\leq C\lim\inf_{z\to y}\frac{\varphi_T}{\log |z-y|}
\end{equation}
to be true for some constant $C> 0$, uniform on compact subsets of $X$.
\end{remark}

\bigskip

\noindent To finish this subsection we consider the following the set-up.
\begin{itemize}

\item $X$ is a normal, compact Kähler space and $\pi: \wh X\to X$ is generically finite,
such that $\wh X$ is also normal (and Kähler).

\item $T= \alpha+ \ddbar \varphi$ is a closed current on $X$, where $\alpha$ 
is smooth and locally $\ddbar$-exact, and such that 
\[T\geq \gamma\]
for some smooth locally $\ddbar$-exact $(1,1)$-form $\gamma$.

\item We assume moreover that $\pi^\star T= [E]+ \theta+ \ddbar\rho,$ where $E$ is effective, $\rho$ is a bounded real function on $\wh X$, and the form $\theta$ is locally given by the Hessian of a function bounded from above. 
\end{itemize}
\smallskip

\noindent Then we claim that the following inequality
\begin{equation}\label{eq100}
\theta+ \ddbar\rho\geq \pi^\star\gamma\end{equation}
holds true (this will be useful in the next sections). This is seen as follows: we only have to verify \eqref{eq100} locally near a point $x_0\in \supp(E)$ in the support of the divisor $E$. Let $U$ be an open subset of $X$
such that $\pi(x_0)\in U$ and such that $\gamma$ restricted to $U$ is
given by the Hessian of the smooth function $f_\gamma$. Assume that there exists an open subset $V$ containing $x_0$ such that we have 
\[\theta|_{V}= \ddbar f_\theta\]
where $f_\theta$ is bounded from above.
It then follows that the function 
\begin{equation}\label{eq101}
f_\theta- f_\gamma\circ\pi+ \rho|_{V_0}\end{equation}
is psh, where $\displaystyle V_0:= \big(V\cap\pi^{-1}(U)\big)\setminus \supp E$. On the other hand, the function in \eqref{eq101}
is bounded from above on $V_0$. Since $\wh X$ is normal, it extends as psh function 
locally near this point, by result due to Hartogs in the smooth case, see \cite{Dem85}, Théorème 1.7 for the version we need here. This is the analogue of the fact that 
bounded holomorphic functions on normal spaces extend.
\smallskip

\noindent In conclusion, the $(1,1)$-current obtained by taking the $\ddbar$ of the function
\eqref{eq101} is positive -- but this is simply 
\[\theta+ \ddbar\rho- \pi^\star\gamma|_V,\]
and our claim is proved.

\subsection{Main results} Prior to stating our results we set a few notations and conventions. 
\medskip

\begin{defn}\label{defn-null} Given a real $(1,1)$-class $\alpha$ we denote by ${\rm Null}(\alpha)$ the \emph{null locus} of $\alpha$, which is given by the union of analytic subsets $V\subset X$ such that ${\int _V\alpha ^{\dim V}=0}$.
\end{defn}
\smallskip

\noindent We next introduce and establish basic properties of a class of closed positive currents on normal varieties
which will play the role of \emph{Kähler currents} in \cite{CT15}.

\subsubsection{Currents with admissible singularities} We introduce the following class of singularities.
\begin{defn}\label{defn-adm-sing} Let $X$ be a normal, compact Kähler space, and let $\varphi:X\to [-\infty, \infty[$ be a function on $X$. We say that $\varphi$ has admissible singularities if 
\[\varphi= \max{(\varphi_1,\dots, \varphi_k)}\]
where each $\varphi_j$ has analytic singularities in the sense of \cite{Dem92} i.e. it can be locally expressed as 
\[\gamma\log(\sum |f_k|^2)\]
modulo a bounded function. In the expression above, $\gamma\geq 0$ is a real number and the functions $(f_k)$ are holomorphic.
\end{defn}

\noindent A more flexible version of this notion reads as follows.

\begin{defn}\label{defn-was}
Let $X$ be a normal, compact Kähler space, and let 
\[T=\alpha+ \ddbar \varphi\]
be a closed positive current of type $(1,1)$ on $X$, where we denote by $\alpha$ a smooth, real $(1,1)$-form on $X$, which is locally $\ddbar$--exact. We say that $T$ has weak analytic singularities if there exists: 
\begin{itemize} 

\item a biholomorphic map $\pi: \wh X\to X$ such that $\wh X$
is normal, and
\item a closed positive current 
\[\wh T= \pi^\star \alpha+ \ddbar \psi\geq 0\]
on $\wh X$ such that $\psi$ has admissible singularities and such that we have $\pi_\star \wh T= T.$
\end{itemize}
\end{defn}
\begin{remark}\label{rmk-sing} As consequence of the fact that the current $\wh T$ is assumed to belong to the class $\pi^\star \alpha$, we show that
\emph{the function $\psi$ above (in the second point of Definition \ref{defn-was}) is constant on every connected component of 
positive dimensional fibers of $\pi$}. This can be seen as follows: assume that the restriction of $\psi$ to a fiber $F$ of $\pi$ is not identically $-\infty$.
Then by Théorème 1.10 in \cite{Dem85} combined with the fact that $\psi$ has admissible singularities 
we infer that $\psi|_F$ is a psh function defined on a compact analytic space -- hence, it must be constant by the maximum principle.
So, there exists a function $\varphi_1$ on $X$
such that $\varphi_1\circ\pi= \psi$, and moreover, the difference $\varphi- \varphi_1$
is smooth (because it belongs to the kernel of the operator $\ddbar$).\footnote{Since
$\psi$ has admissible singularities, we expect that this should be the case for $\varphi_1$ (and $\varphi$) as well, but it is not clear how such a statement can be proved.}
It follows that Definition \ref{defn-was} is equivalent to the existence of a birational map
$\pi:\wh X\to X$ together with a current $\wh T= \pi^\star \alpha+ \ddbar \psi\geq 0$ on $\wh X$ such that 
$\pi^\star T= \wh T$. In other words, the "singular analogue" of Lemma \ref{lem: inverse} 
holds true.
\end{remark}
\smallskip

\noindent It turns out that this class of currents behaves very well under a few natural operations which will be needed in the proof of Proposition \ref{pro:null-non-kahler} below. In particular we have the following statement.  

\begin{lemma}\label{max_dir}
Let $X$ and $Y$ be normal compact Kähler spaces, and let 
$p: Y\to X$ be a holomorphic map. Let $\beta$ be a smooth, real and closed $(1,1)$--form on $X$.
We have the following assertions.
\begin{enumerate}
\smallskip

\item[\rm (a)] Let $T$ be a current with weak analytic singularities on $X$. Then the inverse image 
$p^\star T$ has weak analytic singularities.
\smallskip

\item[\rm (b)] For $i=1, 2$ let $T_i:= \beta+ \ddbar\varphi_i$ be two currents with weak analytic singularities 
in the class induced by the smooth form $\beta$. If we define $\varphi:= \max(\varphi_1, \varphi_2)$,
then the current $\displaystyle T:= \beta+ \ddbar\varphi$ has weak analytic singularities.
\smallskip

\item[\rm (c)] Assume moreover that $p$ is birational. If $\Theta\in p^\star(\beta)$ is a closed positive current with weak analytic singularities in the class $p^\star(\beta)$ on $Y$, then 
the direct image $T:= p_\star\Theta$ has weak analytic singularities.
\end{enumerate}
\end{lemma}

\begin{proof} 
Concerning the first point (a), consider the map
$\pi_X: \wh X\to X$ given in Definition \ref{defn-was}, so that the inverse image 
\[\wh T:= \pi_X^\star T\] 
is a closed positive current on $\wh X$, whose potential has admissible singularities. We can construct holomorphic maps 
\[\pi_Y:\wh Y\to Y, \qquad \wh p: \wh Y\to \wh X\]
such that the equality $p\circ\pi_Y= \pi_X\circ \wh p$ holds, the space $\wh Y$ is normal
and $\pi_Y$ is birational.

Consider the inverse image $\Theta:= \wh p^\star \wh T$.
By Definition \ref{defn-adm-sing}, it is clear that $\Theta$ has admissible singularities.
On the other hand, by Remark \ref{rmk-sing}, we may assume that 
\[\Theta= \wh p^\star \big(\pi_X^\star T\big)= \pi_Y^\star(p^\star T)\]
which shows that the current $p^\star T$ has weak analytic singularities.
\smallskip

\noindent Point (b) is a direct consequence of (a), so we will not give any further details.
\smallskip

\noindent For assertion (c) we argue as follows.
Given that $\Theta$ has weak analytic singularities, there exist a map $\pi_Y:\wh Y\to Y$ and a closed positive current \[\wh \Theta\in \pi_Y^\star(p^\star\beta)= (p\circ \pi_Y)^\star \beta\] as in Definition \ref{defn-was}, 
such that 
\[\pi_{Y\star}\wh \Theta= \Theta\]
and moreover we have $\wh \Theta= (p\circ \pi_Y)^\star\beta+ \ddc \psi$
for a function $\psi$ with admissible singularities. The map $p\circ \pi_Y: \wh Y\to X$
is birational, and we clearly have
\[(p\circ \pi_Y)_\star\wh \Theta= p_\star\Theta.\]
It therefore follows that the direct image $p_\star\Theta$ has weak analytic singularities. \end{proof}

\medskip

\begin{remark}\label{rmk_ana}
Assume that $X$ is a manifold and that $T\in \alpha$ is a Kähler current with weak analytic singularities. By the regularisation results in \cite{Dem92}, the class $\alpha$ contains a current with analytic singularities, meaning that the 
equality
\[\varphi|_{U_i}= \gamma_i\log\big(\sum_\alpha|f_{i\alpha}|^{2}\big)+ \psi_i\]
holds, where the $f_{i\alpha}$ are holomorphic and $\psi_i$ is bounded. 
We expect that this still holds in our context, i.e. in case $X$ is a normal, compact K\"ahler space. 
\end{remark}

\begin{remark}\label{rmk_ana, I}
It follows from Lemma \ref{max_dir} that in the definition of a current with weak analytic singularities (\ref{defn-was}) we can assume that the space $\wh X$ is non-singular. 
\end{remark}
\medskip


\medskip

\noindent We prove next another property of currents with weak analytic singularities. In the proof below, we use the generic notation $"C"$ for a constant that can change from one line to another. 

\begin{lemma}\label{ana} 
Let $T$ be a current with weak analytic singularities on a normal compact Kähler space $X$.
Then the set $E_+(T):= \bigcup_{c>0} E_c(T)$ is a closed, analytic subset of $X$.
\end{lemma}
\begin{proof} We are using the notations in the previous Remark \ref{rmk-sing}. In particular we assume that $\wh X $ is non-singular, and moreover the set 
$(\varphi= -\infty)$ coincides with the image of $(\psi= -\infty)$ via the map $\pi$,
since we have 
\begin{equation}\label{eq300}
\psi= \varphi\circ \pi\end{equation}
modulo a bounded quantity.
Given that $\psi$ has admissible singularities, 
it follows that we have the equality
\[(\psi= -\infty)= E_+(\wh T),\]
and therefore we have $E_+(T)\subset \pi\big(E_+(\wh T)\big)$, since $E_+(T)\subset (\varphi= -\infty)$
as a consequence of upper-semicontinuity of $\varphi$.

\noindent Actually, more is true, namely the equality
\[E_+(T)= \pi\big(E_+(\wh T)\big)\]
holds. To see this, it would be enough to show the existence of a constant $C>0$ such that we have   
\begin{equation}\label{eq301}
C\nu (T, y)\geq \nu (\pi^\star T, x)= \nu (\wh T, x)
\end{equation}
where $y\in X$ is an arbitrary point and $x\in \pi^{-1}(y)$.
If we admit this for the moment, the proof of our lemma is complete.
\smallskip

\noindent In order to establish inequality \eqref{eq301}, we proceed as follows.  
Consider a local parametrization $\tau: (X, y)\to (\C^n, 0)$. It is a proper, finite map such that $\tau^{-1}(0)= y$ (as sets). In this context we have the following important estimate, cf. \cite{Dem82}, Th\'eor\`eme 6
\begin{equation}\label{eq219}
C\nu(T, y)\geq \nu(\Theta, 0),
\end{equation}
which we have already mentioned in \eqref{eq211}, where the constant $C$ here is very explicit (depending on a certain multiplicity associated to the map $\tau$) and 
$\Theta:= \tau_\star T$ is the direct image of $T$ with respect to the proper map $\tau$. 

\noindent By the main result in \cite{Fav99}, we have 
\begin{equation}\label{eq220}
C\nu(\Theta, 0)\geq \nu\big((\tau\circ\pi)^\star \Theta, x\big),
\end{equation}
which combined with \eqref{eq219} gives 
\begin{equation}\label{eq221}
C\nu(T, y)\geq \nu\big((\tau\circ\pi)^\star \Theta, x\big).
\end{equation}

It would therefore be sufficient to show that the inequality
\begin{equation}\label{eq222}
\tau^\star(\tau_\star T)\geq T,
\end{equation}
because then it follows that $\displaystyle \nu\big((\tau\circ\pi)^\star \Theta, x\big)\geq \nu(\pi^\star T, x)$. 

Let $T|_U= \ddbar \varphi_T$ be the local expression of the current $T$. Then we have 
\[\tau_\star T= \ddbar \psi\]
where $\displaystyle \psi(z):= \sum_{w\in \tau^{-1}(z)}\varphi_T(w)$ is the trace of the local potential $\varphi_T$ of $T$. It follows that the following formula 
\begin{equation}\label{eq223}
\tau^\star(\tau_\star T)= \ddbar \psi\circ \tau\geq \ddbar \varphi_T,
\end{equation}
holds. Indeed, the difference
\begin{equation}\label{eq224}
\psi\circ \tau (w)- \varphi_T(w)= \sum_{x\in F_w^\star}\varphi_T(x)
\end{equation}
is a psh function on $U$, where 
$\displaystyle F_w^\star := \{x\in U : \tau(x)= \tau(w), x\neq w\}$. The argument for this last claim is 
as follows: on the unramified locus of $\tau$ things are clear, and on the other hand 
the RHS of \eqref{eq224} is uniformly bounded from above.
Our proof is finished.
\end{proof}

\begin{remark}
The inequality \eqref{eq222} does not holds in general. For example, if instead of being finite and proper 
the map $\tau$ is the blow-up of $\mathbb C^2$ at $0$, then \eqref{eq222} certainly fails in case $T$ is the current of integration on the exceptional divisor. 
\end{remark}

\medskip

\noindent Finally, we introduce the following notion.

\begin{defn}\label{wsnK}
Let $X$ be a normal compact Kähler space, and let $\alpha$ be a nef and big real $(1,1)$-class on $X$
(in the Bott-Chern cohomology). The restricted non-Kähler locus of $\alpha$ is the following set
\[E_{nK}^{as}(\alpha):= \bigcap_{T\in \alpha} E_+(T)\]
where $T$ above is assumed to be a Kähler current with weak analytic singularities, and $E_+(T)\subset X$ is the (analytic) subset of $X$
for which the Lelong numbers of $T$ are strictly positive.
\end{defn}
\smallskip

\begin{remark} In the case of a non-singular Kähler space $X$, one defines $E_{nK}(\alpha)$ as the intersection of 
$E_+(T)$ for all Kähler currents $T\in\alpha$. 
Thus, the difference between $E_{nK}^{as}(\alpha)$ and $E_{nK}(\alpha)$ is that in the definition of the former we restrict ourselves to currents with weak analytic singularities. If $X$ is non-singular,  then we have $E_{nK}^{as}(\alpha)= E_{nK}(\alpha)$, thanks to the regularisation results in \cite{Dem92}. In the general case of a normal space, things are less clear, but we can at least say that 
\[E_{nK}(\alpha)\subset E_{nK}^{as}(\alpha)\subset E_{nK}(\alpha)\cup X_{\rm sing}\]
holds true. We actually expect the first inclusion to be an equality.
\end{remark}

\smallskip

\noindent The following statement will be important in what follows.

\begin{corollary} Let $X$ be a normal compact complex Kähler space, and let $\alpha$ be a nef and big 
$(1,1)$-class. Then $E_{nK}^{as}(\alpha)$ is an analytic subset of $X$, and there is a K\"ahler  current with weak analytic singularities $T\in \alpha$ such that $E_+(T)= E_{nK}^{as}(\alpha).$
\end{corollary}
\begin{proof}
By (b) of Lemma \ref{max_dir}, given two K\"ahler currents with weak analytic singularities $T_i\in \alpha$, we can construct $T\in\alpha$ such that 
\[E_+(T)\subset  E_+(T_1)\cap E_+(T_2)\]
and moreover $T$ is again a Kähler current with weak analytic singularities. We can therefore construct a sequence $T_k\in \alpha$ of such currents, for which the following assertions
\[E_+(T_{m+1})\subset E_+(T_{m}), \qquad E_{nK}^{as}(\alpha)\subset  \bigcap_k E_+(T_k)\]
are true for each $m\geq 1$. Since $X$ is compact and $E_+(T_{m})$ are analytic sets, the corollary follows because by notherian induction there exists an integer $m_0$ such that $E_+(T_{m+1})= E_+(T_{m})$ for 
all $m\geq m_0$.
\end{proof}

\medskip

\noindent In this context, we have the following statement, which represents the main result of this section.

 \begin{theorem}\label{pro:null-non-kahler}
       Let $X$ be a compact normal K\"ahler variety, and let $\alpha$ be a smooth (1,1)-form, which is locally $\ddbar$-exact and such that the corresponding class is nef and big. Then $E_{nK}^{as}(\alpha)= \Null(\alpha)$. In particular,
the set $\Null(\alpha)$ is analytic.   
\end{theorem}
\smallskip

\begin{remark}
As we have already mentioned in Remark \ref{slope}, in singular setting the Lelong number of a positive current can be different from the slope of its potential at a given point. 
Therefore one might wonder why the former notion appears in Proposition \ref{pro:null-non-kahler} and not the latter. The explanation is given by Lemma \ref{slope, III}: considering slopes instead of Lelong numbers would lead to the same set $E_{nK}^{as}(\alpha)$.
\end{remark}

   \begin{proof} 
   The arguments which will follow combine \cite{CT15} (where this statement is established in case $X$ is a manifold), with additional inputs from \cite{DHP22}.
\medskip

\noindent \emph{Step 1: the inclusion $\Null(\alpha)\subset E^{as}_{nK}(\alpha)$ holds.} Indeed, let $V\subset X$ be an  irreducible component 
of $\Null(\alpha)$ so that
\begin{equation}\label{eq80}\int_{V_{\rm reg}}\alpha^d= 0,\end{equation}
where $d$ is the dimension of $V$. If $V\not\subset E_{nK}^{as}(\alpha)$, then by the definition of this subset, there exists a 
Kähler current $\Theta\in \alpha$, whose potentials have weak analytic singularities and whose Lelong number 
at a generic point of $V$ is equal to zero. These two properties show that the restriction $\displaystyle \Theta|_V$
is a Kähler current, so in particular $\alpha|_V$ is nef and big. This contradicts the equality \eqref{eq80}. 
\medskip

\noindent \emph{Step 2: the set $E^{as}_{nK}(\alpha)$ does not have isolated points.} Let 
\[T:= \alpha+ \ddbar \phi\]
be a Kähler current in the class $\alpha$ (in the equality above we abusively denote by $"\alpha"$ a smooth representative of this class). Assume that the function $\phi$ has weak analytic singularities and moreover 
$x\in X$ is an isolated point in the set $\phi^{-1}(-\infty)$. Then we can remove the pole of $\phi$ at $x$ as follows (see \cite{Dem92} and the references therein). 

Consider $(X, x)\subset U\subset (\C^N, 0)$ a local embedding of $X$, such that \[\alpha|_{U\cap X}= \ddbar \tau_x\]
for some smooth function $\tau_x$. The sum $\displaystyle \tau_x+ \phi$ is the restriction of a function 
$\psi$ defined on $U$ and whose Hessian is bigger than a positive multiple of the Euclidean metric on $\C^N$. We now define
\[\widetilde \psi:= \max(\psi, C_1\Vert Z\Vert^2- C_2)\]
where $C_1$ is strictly positive and $C_2\gg 0$ is large enough, so that 
\[\psi|_{\partial(U\cap X)}> (C_1\Vert Z\Vert^2- C_2)|_{\partial(U\cap X)}\]
holds. This is indeed possible, since the restriction of $\psi$ to a small enough neighborhood of the boundary of $U\cap X$ in $X$ is smooth. We have denoted by $Z$ the coordinates in $\C^N$ and for each open set $\Omega$ we denote by 
$\partial (\Omega)$ its boundary.  

We then define $\widetilde \phi$ the function on $X$ given by $\widetilde \psi|_{U\cap X}- \tau_x$ on $U\cap X$ and $\phi$
on $X\setminus U$. This function has weak analytic singularities, and the current
\[\widetilde T:= \alpha+ \ddbar \widetilde\phi\]
is greater than a small multiple of the Kähler metric on $X$. Moreover, $x\not \in \widetilde \phi^{-1}(-\infty)$.
Therefore, the set $E^{as}_{nK}(\alpha)$ cannot contain isolated points, all its irreducible components must have dimension at least one.

\medskip

\noindent \emph{Step 3: the inclusion $E^{as}_{nK}(\alpha)\subset {\rm Null}(\alpha)$ holds.} Let $V\subset E^{as}_{nK}(\alpha)$ be any irreducible component. We have to show that $\displaystyle \int_{V_{\rm reg}}\alpha^d= 0$, where $d\geq 1$ is the dimension of $V$. Assume that this equality does not hold. Then given that the restriction $\alpha|_V$ is nef, the 
only alternative is 
\begin{equation}\label{eq81}\int_{V_{\rm reg}}\alpha^d> 0\end{equation}
and we show next that \eqref{eq81} leads to a contradiction. This will be done if we are able to construct a 
Kähler current in the class $\alpha$, whose potential has weak analytic singularities and such that 
it is bounded locally at some point of $V$.

A first important reduction is that we can assume that $V$ is non-singular. Indeed, as shown in the proof of \cite[Theorem 2.29]{DHP22} there exists a modification $p:\wh X\to X$ such that the following hold.
\begin{itemize}
\item The complex Kähler space $\wh X$ is normal. 

\item The proper transform $\wh V$ of $V$
is a smooth submanifold of $\wh X$. 

\item The map $p$ is an isomorphism locally over an open subset of $V$. 
\end{itemize}
If we are able to construct a Kähler current $\Theta\in p^\star\alpha$ with weak analytic singularities
such that $\Theta$ is bounded locally at a very general point of $\wh V$, then we are done 
by considering the direct image $p_\star\Theta$, cf. Lemma \ref{max_dir}, (b).
Thus replacing $X$ with $\wh X$, we can assume from this point on that $V$ is non-singular. 

Thanks to inequality \eqref{eq81}, it follows that we can construct a Kähler current 
\[\Theta_V:= \alpha|_V+ \ddbar f\]
such that $f:V\to [-\infty, 0]$ has \emph{analytic singularities}, cf. Remark \ref{rmk_ana}. 

On the other hand, the class $\alpha$ is nef and big on $X$, so there exists a Kähler current 
\[\Theta:= \alpha+ \ddbar F\]
where $F$ has weak analytic singularities along $V\cup Z$, where $Z\subset X$ is an analytic subset of $X$ which does not contain $V$. 

 The idea is to remove the singularities of $\Theta$ at the generic point of $V$ (as we did in Step 2) by using $\Theta_V$, so that the resulting current will provide the sought-after contradiction. However, in the actual context additional complications arise due the singularities of $F$.
\medskip

 \noindent \emph{A particular case.} As "warm-up" for the rest of the proof, we provide here an argument in case the function $f$ is smooth. By the proof of \cite[Theorem 4]{Dem90} there exists a 
 smooth extension $\wt f$ defined in an open subset $U$ of $V$ such that 
 \[\alpha|_U+ \ddbar \wt f\]
is a Kähler current on $U$. Let now $\psi_Z$ be a quasi-psh function on $X$, with analytic poles along $Z$. If $\delta_0> 0$ is sufficiently small, then 
\begin{equation}\label{eq82}\alpha|_U+ \ddbar (\wt f+ \delta_0\psi_Z|_U)\end{equation}
is a Kähler current on $U$, and it has poles along the analytic set $Z\cap U$. 

Next we will use the hypothesis "$\alpha$ nef" in order to diminish the order of poles of $F$. This is necessary, as otherwise we will be unable to glue $\Theta$ with the form in \eqref{eq82}. Indeed, for any 
positive $\ep> 0$ there exists a smooth function $G_\ep$ on $X$ such that 
\[T_\ep:= \alpha+ \ddbar G_\ep \geq -\ep \omega_X\]
where $\omega_X$ is a reference Kähler metric on $X$. Let $\delta_1> 0$ be a strictly positive real number such that $\displaystyle \Theta\geq \delta_1\omega_X$. The convex combination
\[\Theta_\ep:= (1-\frac \ep{\delta_1})T_\ep+ \frac \ep{\delta_1}\Theta\]
is a Kähler current in the class $\alpha$. It is easy to check that it is greater or equal than
\[-\ep(1-\frac \ep{\delta_1})\omega_X+ \frac{\ep}{\delta_1} \delta_1\omega_X= \frac{\ep^2}{\delta_1}\omega_X.\]

On the other hand, the singularities of $\Theta_\ep$ are of order $\mathcal O(\ep)$, in other words we can make them as small as we want/need.
Let $\displaystyle F_\ep:= (1-\ep/\delta_1)G_\ep+ \ep/\delta_1 F$ be the potential of $\Theta_\ep$. It follows that given a point $z\in Z\setminus V$, there exists an open subset $U_z$ containing $z$ such that the following inequality 
\begin{equation}\label{eq83}F_\ep> \delta_0\psi_Z\end{equation}
holds for all $0<\ep\ll 1$
small enough. This can be seen as consequence of Hilbert's Nullstellensatz, as follows.
\smallskip

We recall that $E_+(\Theta)$ is equal to $Z\cup V$. Moreover 
there exists a birational map $\pi:\wh X\to X$ such that $\pi^\star\Theta$ has admissible singularities,
meaning in particular that
locally near $w\in \pi ^{-1}(z)$ we have
\[F\circ\pi|_{(\wh X, w)}= \log(\sum |f_i|^{2r_i})+ b\]
where $b$ is bounded, $(f_i)$ is a set of holomorphic functions and $r_i\geq 0$ are
positive real numbers. 

In the proof of Lemma \ref{ana} we showed that 
$\displaystyle E_+(\Theta)= \pi\big(E_+(\pi^\star\Theta)\big)$, which in particular implies 
that we have $\displaystyle \pi^{-1}\big(E_+(\Theta)\big)\supset E_+(\pi^\star\Theta)$. In fact, the equality 
\begin{equation}\label{eq302}
\pi^{-1}\big(E_+(\Theta)\big)= E_+(\pi^\star\Theta)\end{equation}
holds: let $x\in \pi^{-1}\big(E_+(\Theta)\big)$, so that $\nu\big(\Theta, \pi(x)\big)> 0$.
We claim that $\nu(\pi^\star\Theta, x)> 0$ is strictly positive as well. If this is not the case,
the local potential of $\pi^\star\Theta$ is \emph{locally bounded from below} near $x$. By the analogue of \eqref{eq300} in our setting here, it follows that the potential of $\Theta$
is equally locally bounded from below near $\pi(x)$. But this contradicts the fact that 
the Lelong number of $\Theta$ at $\pi(x)$ is positive.

Consider next the ideal $\mathcal I=(f_i)$ generated by the functions appearing in the expression of $F\circ\pi$ above. We remark that the set of zeros of $\mathcal I$ is contained in $(\pi^{-1}(Z), w)$ (this is only true because we are "far" from $V$). The function $\psi_Z$ is obtained by gluing functions 
of type $\log(\sum |h_j|^2)$, where $(h_j)$ are the local equations of $Z$ and so there exists an integer $N>0$ such that $(h_j\circ \pi)^N\in \mathcal I$. In particular we have $\sum |h_j\circ \pi|^{2N}\leq C(\sum|f_i|^2)$ for some positive constant $C> 0$. By arranging the constants, and taking into account the fact that $\pi^{-1}(z)$ is a compact set, this implies \eqref{eq83}. 
 

\smallskip

We fix a value, say $\ep_0$, of $\ep$ small enough, such that \eqref{eq83} holds 
true for every point of $\partial U\cap Z$: this is possible, since $\partial U\cap V= \emptyset$.
Then we claim that
there exists a positive constant $C> 0$ such that the inequality 
\begin{equation}\label{eq84}\wt f+ \delta_0\psi_Z< F_0+ C\end{equation}
holds true pointwise in a small open subset containing the boundary $\partial U$ of $U$, where $\displaystyle F_0:= F_{\ep_0}$. 

\noindent In order to verify this claim, we note that $\partial U$ is a compact set, and let $x\in \partial U$ be one of its elements. If $\displaystyle x\in \partial U \cap Z$, we clearly have
$\displaystyle \delta_0\psi_Z< F_0$
locally near $x$, as consequence of the inequality \eqref{eq83}. Moreover, the function $\wt f$
is non-singular, so the existence of the constant $"C"$ such that \eqref{eq84} holds in a small open subset of $x$ is guaranteed. If $\displaystyle x\not\in \partial U \cap Z$, then things are clear. By compactness, we have established our claim.
\smallskip

\noindent The next claim is that the function 
\[\wh F:= \max(\wt f+ \delta_0\psi_Z, F_{0}+ C)\]
is defined on the whole space $X$, is bounded at the generic point of $V$, and moreover
\[\alpha+ \ddbar \wh F\]
is a Kähler current. Indeed, near the boundary $\partial U$ of the set $U$ where $\wt f$ is defined
the inequality \eqref{eq84} shows that $\wh F= F_0+C$, and thus 
we define $\wh F=F_0+C$ on the complement of $U$.

\medskip

 \noindent \emph{End of the proof.} 
In general, the function $f$ has no reason to be smooth, but nevertheless the line of arguments above remains valid, thanks to a very important remark in \cite{CT15}. The point is that instead of obtaining a smooth extension $\wt f$ of $f$ defined on an open subset $U\subset X$ as above, in the actual context we only get a function with log poles $\wt f$ and a "pinched" neighborhood $U$ of $V\setminus W$, where $W$ is an analytic subset of $V$, such that the analog expression \eqref{eq82} is a Kähler current. In our context, this can be seen as follows. 

Consider finitely many open subsets $(A_i)$ of $X$, such that each $A_i$ is an analytic subset the unit ball of some Euclidean space,
and such that $V\subset \cup A_i$. We can assume that the equality 
\[f|_{A_i\cap V}= \delta_i\log(\sum_{\alpha} |g_{i\alpha}|^2)+ \tau_i\]
holds, where $\delta_i> 0$, the functions $g_{i\alpha}$ defined on $A_i\cap V$ are holomorphic and $\tau_i$ are bounded (given that $f$ has analytic singularities). In particular, the equations $g_{i\alpha}=0$ define a global analytic set $W$ contained in $V$. 

Let $\mathcal I_W\subset \mathcal O_V$ be the ideal sheaf of $W$, and let 
$\mathcal J_W\subset \mathcal O_X$ be its pull-back via the projection map $\mathcal O_X\to \mathcal O_V$. Then there exists a birational map
\[\pi: \wh X\to X\]
obtained by blowing-up smooth centres contained in $V$, such that the following hold.
\begin{enumerate}

\item[(a)] \emph{The space $\wh X$ is normal and the proper transform $\wh V$ of V is non-singular. Moreover, the map $\pi$ is biholomorphic near the general point of $\wh V$.}

\item[(b)] \emph{The inverse image of $\mathcal J_W$ via the restriction of $\pi$ to $\wh V$
is equal to $\mathcal O_{\wh V}(-D)$, where $D$ is an effective divisor on $\wh V$ whose support is snc.}
\end{enumerate}

\noindent For the construction of the map $\pi$  we use the same argument as in \cite{DHP22}.
We start by constructing a principalization of the ideal $\mathcal I_W\subset \mathcal O_V$:
this is achieved by a finite sequence of blow-ups 
\[p_k:V_{k+1}\to V_k\]
of smooth centres $\Sigma_k\subset V_k$, for $k=0,\dots ,N-1$ with $V_0:= V$. Next we interpret  $\Sigma_0\subset V$ as analytic subspace of $X$, and we blow-up $X$ along $\Sigma_0$; let 
\[\pi_1: X_1\to X\]
be the corresponding map. We get a closed immersion $\displaystyle V_1\to X_1$ (whose image is simply the proper transform of $V$), and we repeat this operation with $\Sigma_1$. Notice that the space $X_N$ is not necessarily normal, but the map $\displaystyle X_N\to X$ is biholomorphic at the generic point of $V_N$. Since $X$ is normal, it follows that the complement of a proper analytic subset of $V_N$ is contained in the set of normal points of 
$X_N$. Therefore, the normalization $\nu: \wh X\to X_N$ is biholomorphic 
near the proper transform $\wh V$ of $V_N$ (cf. \cite{GR}, Corollary, page 164).
 It follows that the map \[\pi:\wh X\to X\] 
obtained by composing the normalization of $X_N$ with the previous sequence of blow-ups has all the properties we need. 

 Consider next the pull-back current $ (\pi|_{\wh V})^\star\Theta_V$. 
Its singularities are concentrated along the snc divisor $D$ in $\wh V$, whose support is $\Lambda_1+\dots + \Lambda_N$. For each $i=1,\dots, N$ we 
denote by $\beta_i$ an arbitrary, smooth representative of the first Chern class
of $\Lambda_i$.
By blowing up $\wh X$ along $\Lambda_i$, we can add the following item to the properties of $\pi$ above:
\begin{enumerate}

\item[(c)] \emph{There exists a set $(\rho_i)_{i=1,\dots, N}$ of smooth $(1,1)$-forms on $\wh X$ such that the equality
\[\rho_i|_{\wh V}\equiv \beta_i\]
holds for each $i= 1,\dots, N$, meaning that the restriction of $\rho_i$ to $\wh V$ belongs to the 
class $\beta_i$.}
\end{enumerate}
Indeed, locally analytically we blow up the non-singular set $\Lambda _i$ in $\mathbb C^N$. The exceptional set of $\widetilde {\mathbb C^N}\to \mathbb C^N$ is a smooth divisor $F_i\to \Lambda _i$. Then $E_i:=\widetilde X\cap F_i$ where $\widetilde X\to \wh X$ is the strict transform and even if $E_i$ is neither reduced nor irreducible, it is nevertheless a Cartier divisor. Then $\mathcal O(E_i)$ is locally free, and so we can endow it with a smooth metric denoted by $h_i$ (this notion is defined precisely as in the usual case of a line bundle on a manifold). The curvature form corresponding to $h_i$ will be our $\rho_i$, for each index $i$. Moreover, notice that these additional transformations do not affect $\wh V$, since $\Lambda_i$ has codimension one in $\wh V$.
\medskip 

\noindent Given the properties (a)--(c) above, we obtain the decomposition
\[(\pi|_{\wh V})^\star\Theta_V= \alpha_1+ \sum a_i[\Lambda_i]+ \ddbar \wh f\]
where the
notations/conventions are as follows:
\begin{itemize}

\item The same symbol e.g. $\alpha$ is used to denote a cohomology class and some fixed representative contained in it (in case we do not intend to emphasize a particular representative of the said class).

\item $\Lambda_i$ are the hypersurfaces of $\wh V$ introduced before.

\item The smooth $(1,1)$-form $\alpha_1$ is defined as \[\alpha_1= (\pi|_{\wh V})^\star(\alpha|_V)- \sum a_i\rho_i|_{\wh V},\]
where $\rho_i$ are the forms in (c). 
\item The function $\wh f$ is bounded. 
\end{itemize}

\noindent 
We note that since $\Theta_V$ is a K\"ahler current, the inequality
\[(\pi|_{\wh V})^\star\Theta_V\geq \delta (\pi|_{\wh V})^\star(\omega _X|_V)\]
holds, for some $\delta> 0$. It follows that we have 
\[\wh\Theta_V:= \alpha_1+ \ddbar \wh f\geq \delta(\pi|_{\wh V})^\star(\omega _X|_V)\]
as well, since $\wh\Theta_V$ is smooth, and the inequality above holds in the 
complement of an analytic set of $\wh V$.

Then we claim that
for each $\ep > 0$ there exists an open subset
$U_\ep \subset \wh X$ containing $\hat V$, together with a smooth function $\wh f_\ep: U_\ep\to \mathbb R$ such that  
\begin{equation}\label{eq85}\pi^\star(\alpha)- \sum a_i\rho_i+ \ddbar \wh f_\ep\geq \delta\pi^\star(\omega_X)- \ep\omega_{\wh X}\end{equation}
pointwise on $U_\ep$, where $\omega_{\wh X}$ is a fixed Kähler metric on $\wh X$. Indeed, this is 
done in two steps: we first apply the regularisation result in \cite{Dem92} on $\wh V$ in order to convert $\wh f$ to a smooth function. The price to pay is a loss of positivity, which can be assumed to be of size $\displaystyle \frac{\ep}{2}\omega_{\wh X}|_{\wh V}$.
Then by the argument in \cite{DeBook} already used in the particular case above, we obtain $U_\ep$
and $\wh f_\ep$.
\smallskip

\noindent The inequality \eqref{eq85} above is in particular true if we construct the metric $\omega_{\wh X}$ as follows:
\[\omega_{\wh X}:= \pi^\star\omega_X- \sum \ep_i\rho_i\]
where $\ep_i> 0$ are well-chosen real numbers. Then we take $\displaystyle \ep:= \frac{\delta}{2}$ and if we denote by $U$ and $\wt f$ the corresponding set and function, respectively, then
all in all we have
\begin{equation}\label{eq86}\pi^\star(\alpha) + \ddbar \big(\wt f+ \sum (a_i+ \frac{\delta}{2}\ep_i)\log|s_{E_i}|_{h_i}^2\big)\geq \frac{\delta}{2}\pi^\star(\omega_X)\end{equation}
in the sense of currents on $U$.
\smallskip

\noindent On the other hand, we also have at our disposal the inverse image current $\displaystyle \pi^\star \Theta\geq \delta_1\pi^\star(\omega_X)$
which has log poles along $\wh V\cup Z$, where $Z\subset \wh X$ is an analytic set. The procedure we have used in the previous particular case applies here: indeed, we have only used the fact that $f$ is smooth in order to construct the open subset $U$. Then adding the 
function $\delta_0\psi_Z$ to it has the effect of diminishing a bit more the lower bound in 
\eqref{eq86}, but we can afford this since $\delta> 0$.

The current obtained after the gluing procedure on $\wh X$ has log poles and it is non-singular at the generic point of $\wh V$. But as we have already mentioned, the birational map
$\pi$ is a biholomorphism at the general point of $V$, so the direct image of the said current will be smooth at the generic point of $V$. Moreover, it has weak analytic singularities by definition --actually this is the main reason why we have introduced this class of singularities.
   \end{proof}

\section{On subadjunction and the canonical bundle formula}

\noindent 
Let $X$ be a compact Kähler manifold, and consider a real $(1,1)$-class ${\alpha}$ on $X$. We assume that ${\alpha}$ contains a closed positive current $R\geq 0$ with admissible singularities.
This means that we can write
\begin{equation}\label{new70}
R= \alpha+ \ddbar \phi
\end{equation}
where (abusing notation) $\alpha$ is a smooth representative of the class $\alpha$ and the function $\phi$ verifies the conditions in Definition \ref{defn-adm-sing}. 

\noindent We denote by $\pi:\wh X\to X$ a log-resolution of the integral closure of the ideal generated locally by the functions
$(f_i)$ in Definition \ref{defn-adm-sing}. The pull back of the current $R$ admits the decomposition
\begin{equation}\label{new71}
\pi^\star R= [D]+ R_0 
\end{equation}  
where we denote by $[D]$ the current of integration along the effective divisor $D$ and $R_0\geq 0$ is a 
closed, smooth and semi-positive $(1,1)$-form on $\wh X$. By analogy with the case in which $R$ is induced by an effective divisor
we call such a map $\pi$ a log-resolution of $R$.
\medskip

\noindent We assume next that there exists a log-resolution of $R$, such that the following additional requirements are satisfied.
\begin{enumerate}

\item The support of the $\pi$-exceptional divisor $E$ has simple normal crossings, and $\pi$ is obtained as composition of 
blow-ups of smooth centers.

\item We have $\displaystyle \pi^\star(K_X+ \alpha)\simeq K_{\wh X}+ \beta + S+\Xi_1- \Xi_2$, where:
\begin{itemize}

\item The notation above means that the relative canonical class $K_{\wh X/X}$ of $\pi$ plus 
the divisor $S+ \Xi_1- \Xi_2$ coincides with $ \pi^\star(\alpha)-\beta$. 

\item $\beta$ is a nef $(1,1)$-class on $\wh X$. 

\item $S$ is a smooth hypersurface on $\wh X$ whose image is denoted by $T:= \pi(S)$. 
\item The restriction $\pi|_S$ admits a decomposition
$$\pi|_S= \pi_W\circ f$$
where $\pi_W:W\to T$ is a desingularisation of $T$ and $f:S\to W$ is holomorphic.
\item The $\Xi_i$'s are effective $\R$-divisors on $\wh X$ such that their supports do not have common components, $S+\Xi_1+ \Xi_2$ is snc, $(\wh X, S+\Xi_1)$ is plt and 
any hypersurface $Y$ contained in the support of $\Xi_2$ is $\pi$-exceptional. 
\end{itemize}
\end{enumerate}
\smallskip

\noindent 
We let \begin{equation}\label{eq2}
\gamma:= \pi_W^\star(K_X+\alpha)|_T- K_{W}, 
\end{equation}
and we write 
\begin{equation}\label{eq1}
f^\star(\gamma)+ \Xi_2|_S \simeq K_{S/W}+ (\beta + \Xi_1)|_S.
\end{equation}

Note that the pair $(\wh X, \Xi:=\{\Xi _1-\Xi _2\})$ is klt and that $T$ is the unique center of log canonical singularities for the generalized pair $(\hat X,S+\Xi _1-\Xi _2+\beta)$. In particular $T$ is normal. 
\medskip

\begin{enumerate}

\item[3.] The class $\beta$ contains a smooth, positive representative.

\item[4.] The coefficients of the divisors $\Xi_i$ are rational.
\end{enumerate}
\medskip

\noindent A first result we establish here is the following.

\begin{theorem}\label{L2/m} Assume that the requirements 1-4 above are satisfied. 
Then
the class $K_{S/W}+ (\beta + \Xi _1)|_S$ contains a closed positive current $\Theta\geq 0$ such that: 
\begin{itemize}
\item[\rm (1)] For each general fiber $S_w= f^{-1}(w)$, the restriction $\displaystyle \Theta|_{S_w}$ is induced by the space of sections of the line bundle associated to $m\Xi_2|_{S_w}$, for $m$ large and sufficiently divisible.
\item[\rm (2)] Consider the divisor $\Xi_2'\leq \Xi_2|_S$ obtained by discarding the components of $\Xi_2|_S$ whose image is contained in the singular subset of $T$. 
Then, $\Theta\geq [\Xi_2']$ -- that is to say, the current $\Theta$ is singular along the divisor $\Xi_2'$.
\end{itemize}
\end{theorem}
\smallskip

\medskip

\noindent The following results can be seen as "transcendental" versions of the 
canonical bundle formula. They can be used to refine Theorem \ref{L2/m}, but they are of independent interest as they apply to a variety of other contexts. 

\smallskip

\noindent Let $f:S\to W$ be a surjective map of compact Kähler manifolds. Let $P:= \sum P_i$ and $Q=\sum Q_j$ be two reduced, snc divisors on $S$ and $W$, respectively such that moreover $f^{-1}Q\subset P$. We decompose the divisor $P$
\[P= P^h+ P^v\]
into $f$-horizontal and $f$-vertical parts, and we assume moreover that 
the restriction of $f$ to the support of $P^h$ is relatively snc on the complement of the support of $Q$, and moreover $f(\supp P^v)= Q$.
\smallskip

\noindent Let $B= \sum d_iP_i$ be a $\mathbb Q$-divisor on $S$, 
and let $\beta$ be a $(1,1)$-class such that the following requirements are satisfied.
\begin{enumerate}

\item[(a)] The pair $(S, B)$ is sub-klt.

\item[(b)] The morphism $\mathcal O_W\to f_\star\mathcal O_S(\lceil-B\rceil)$
is surjective at general points of $W$. 

\item[(c)] We have $K_S+ B+\beta\simeq f^\star\gamma$ and moreover $\beta$ contains a smooth positive representative.

\item[($\star$)] For any point $z_0\in S$ and $w_0= f(z_0)\in W$ there exist
local coordinates $(x_1,\dots, x_{n+m})$ on $S$ centred at $z_0$ and $(t_1,\dots, t_m)$ on $W$ centred at $w_0$ such that 
$\displaystyle t_i\circ f(x)= \prod x_j^{k_{ij}}$ where
the $k_{ij}$ are non-negative integers such that $k_{ij}\neq 0$
 for at most one $i$ for each index $j$. 
\end{enumerate}
\medskip

\noindent Then the following result holds true -- the case $\beta=0$ corresponds to the original result of Y. Kawamata in \cite{Kawamata98}.

\begin{theorem}\label{lelong} Assume that conditions {\rm (a), (b), (c)} as well as $(\star)$ hold.
Then the class $\{\gamma\}$ can be decomposed as $K_W+ B_W+ \beta_W$ where 
$B_W$ is the discriminant divisor and $\beta_W$ is a cohomology class containing a closed positive current with zero Lelong numbers. In particular, $\beta_W$ is a nef class. 
\end{theorem}

\begin{remark}
Note that the hypothesis {\rm (a), (b), (c)} are very natural, identical to the set-up in \cite{Kawamata98}. We expect Theorem \ref{lelong} to hold without the additional hypothesis $(\star)$,
but there are serious technical difficulties to overcome. 
\end{remark}
\medskip

\noindent One could ask the same type of questions in a more flexible and natural context, in which $\beta$ is only assumed to be nef (so that we start with a nef class on $S$ and the "output" is a nef class on $W$). It turns out that the situation is a bit more complicated --the reason being that a perturbation of $\beta$ could destroy the first hypothesis in (c)--,  and in order to treat the nef case we consider the following assumptions.
\begin{enumerate}

\item[(d)] There exists a Kähler metric $\omega$ on $S$ such that \[\omega= f^\star g+ \theta,\]
where $g$ is a Kähler metric on $W$ and $\theta$ is a \emph{rational} $(1,1)$-form on $S$. Therefore, 
for any coordinate subset $\Omega\subset W$ biholomorphic to a ball we have a $\mathbb Q$-line bundle 
$A_\Omega$ on $V:= f^{-1}(\Omega)$ whose curvature equals $\omega|_V$.

\item[(b')] For any fixed, sufficiently big and divisible integer $m_0> 0$, there is an integer $k_0$  such that the natural inclusion $f_\star\mathcal O_V(m_0A_\Omega+ k_0\lceil-B\rceil)\subset f_\star\mathcal O_V(m_0A_\Omega+ k\lceil-B\rceil)$ is an isomorphism for all $k\geq k_0$ (i.e. the local sections vanish along $(k-k_0)\lceil-B\rceil$),
where $A_\Omega$ is defined above.

\item[(c')] We have $K_S+ B+\beta\simeq f^\star\gamma$ and $\beta$ is a nef class.
\end{enumerate}

\begin{theorem}\label{lelong_1}
Assume that conditions  {\rm (a), (b'), (c'), (d)}, as well as $(\star)$ hold. 
Then the class $\{\gamma\}$ can be decomposed as $K_W+ B_W+ \beta_W$ where
$B_W$ is the discriminant divisor and $\beta_W$ is a nef class. 
\end{theorem}

\begin{remark}
The proof of Theorem \ref{L2/m} will show that the hypothesis (b') is quite natural, 
in the sense that if the map $f:S\to W$ is induced by a birational transformation $\pi$, then these hypothesis hold true.
\end{remark}
\medskip

\noindent The content of the following sections is organized as follows.  There are two techniques of constructing closed positive currents in twisted relative classes of a map between compact K\"ahler manifolds. One can either use fiberwise holomorphic sections (normalized in a canonical manner), or fiberwise Kähler-Einstein metrics, cf. \cite{Gue20} and the references therein. Here we will use the former, since the latter is not sufficiently general to be implemented in our context.
\smallskip

\noindent Indeed, given a holomorphic surjective map $f:S\to W$ between two Kähler manifolds
and a Hermitian line bundle $(L, h_L)\to X$, the spaces 
\[H^0\big(S_w, (K_{S_w}+ L|_{S_w})\otimes\mathcal I(h_L|_{S_w})\big)\]
of $L^2$ sections (for $w\in W$ general) can be "pieced together" in order to construct a metric on $K_{S/W}+L$, which is semi-positively curved e.g. in case the curvature current $\sqrt{-1}\Theta(L, h_L)\geq 0$ is positive, cf. \cite{BP08}. The same is true in the pluricanonical case, i.e. we can construct a 
positively curved metric on $mK_{S/W}+L$, by replacing the $L^2$ normalization with 
an $L^{\frac{2}{m}}$ condition. As a result, the rational class $\displaystyle K_{S/W}+\frac{1}{m}L$
contains a closed positive current, whose restriction to the general fiber of $f$ is induced by the subspace of sections of
\[H^0\big(S_w, mK_{S_w}+ L|_{S_w})\]
which satisfy an $L^{\frac{2}{m}}$ integrability condition.

\noindent In section \ref{cons} we show that the same holds true if we replace $\frac{1}{m}L$
with a $(1,1)$-class $\alpha$, provided that the we can still define the space above, i.e. 
the restriction $\alpha|_{S_w}$ of our class to the fibers of $f$ is rational. This will settle the first part of Theorem \ref{L2/m}. The singularities of the current constructed are analyzed by using techniques borrowed from extension of pluricanonical forms. 
\smallskip

\noindent Concerning Theorem \ref{lelong}, recall that any nef class is pseudo-effective, but in general the two cones are different. Nevertheless, if a $(1,1)$ class contains a closed positive current whose Lelong numbers at each point of the ambient space are equal to zero, then the class in question is psef, cf. \cite{Dem92}. The nefness of the moduli part (in our notations, the class $\beta_W$) in the canonical bundle formula was established by S. Takayama in \cite{Taka} along these lines. Here we will adopt the same strategy --i.e., we will conclude by showing that $\beta_W$ contains a closed positive current with zero Lelong numbers--, and by the same token, simplify a little the arguments in \emph{loc. cit.}

\medskip

\section{Positivity of the relative adjoint transcendental classes}\label{s-pos}

\noindent We begin this section by recalling the following results. 

\begin{theorem} \label{mp}
Let $f:S\to W$ be a surjective map between two compact complex manifolds. Let $\alpha$ be a real class of type $(1,1)$ on $W$. Then
\begin{enumerate} \item ${\alpha}$ is nef if and only if ${f^\star \alpha}$ is nef.
\item ${\alpha}$ is nef if and only if $\alpha|_Z$ is pseudo-effective for all irreducible proper subvarieties $Z\subset W$ and ${f^\star \alpha}$ is pseudo-effective.
\end{enumerate}
\end{theorem}
\begin{proof} (1) follows immediately from \cite[Lemma 2.38]{DHP22}.

We will now show that (2) also follows from the proofs of  \cite[Theorem 2.36, Lemma 2.38]{DHP22}.
By \cite[Theorem 2.36]{DHP22}, we know that if $Z^\nu \to Z$ is the normalization of an irreducible proper subvariety of $W$, then $\alpha |_{Z^\nu}$ is nef. By \cite[Corollary 2.39]{DHP22}, it suffices to show that if $\omega$ is K\"ahler on $W$ and $\dim W=d$, then $\int _W\alpha ^k\wedge \omega ^{d-k}\geq 0$ for $0<k\leq d$. Suppose that $\alpha $ is not nef, then we let $t>0$ be the nef threshold so that $\alpha +t\omega$ is nef but not K\"ahler. Clearly  $(\alpha+t\omega) |_{Z^\nu}$ is K\"ahler and so by \cite[Corollary 2.39]{DHP22} $\int _W(\alpha+t\omega) ^k\wedge \omega ^{d-k}= 0$ for some $0<k\leq d$.
Let $F$ be a general fiber of $S\to W$, $\eta$ a K\"ahler class on $S$ and $\lambda =\int _F\eta ^{n-d}$ where $n=\dim S$. 
Then 
\[\lambda \cdot \int _W (\alpha+t\omega) ^k\wedge \omega ^{d-k} =\int_S f^* ((\alpha+t\omega) ^k\wedge \omega ^{d-k})\wedge \eta ^{n-d}\geq \int_S f^* (t^k\omega ^{d})\wedge \eta ^{n-d}=\lambda t^k \cdot \int _W\omega ^{d}\]
which is impossible as the LHS equals 0 and the RHS is strictly positive. Thus $t=0$ and $\alpha $ is nef.
\smallskip

\noindent Another way of establishing the point (2) is by using \cite[Corollary 2.32]{DHP22} : this shows that $\alpha$ contains a closed positive current, i.e. it is pseudo-effective. 
The conclusion follows by using \cite[Theorem 2.36]{DHP22}.
\end{proof}
\noindent Therefore, in order to show that the class $\beta_W$ in Theorem \ref{lelong} is nef, it is sufficient to 
show that this is true for its pull back via $f$. This will follow from the main results in the 
next two subsections. In the first one we collect a few results about the construction of closed positive currents.

\subsection{Closed positive currents in twisted relative canonical classes}\label{cons}

\noindent To start with, we introduce the following set of 
notations, \emph{which will only be used in this subsection}.
\begin{enumerate}

\item[(1)] $f:X\to Y$ is a surjective map with connected fibers between compact Kähler manifolds.

\item[(2)] $D= \sum a_i D_i$ is an effective, snc divisor with rational coefficients $0<a_1< 1$ and $L$ is a $\Q$-line bundle on $X$.  Thus there exists a positive integer $m_0$ such that the reflexive hull $L^{[m_0]}:=(L^{\otimes m_0})^{\vee \vee}$ is a genuine line bundle,
and a metric on $L$ will simply be given by a collection of functions $\displaystyle \Big(\frac{\varphi_i}{m_0}\Big)_{i\in I}$,
where $(\varphi_i)_{i\in I}$ are the weights of a metric on $L^{[m_0]}$. By abuse of notation we will often denote $L^{[m_0]}$ by $m_0L$. 

\item[(3)] $\alpha$ and $\gamma$ are real $(1,1)$--classes on $X$ and $Y$, respectively. Moreover, $\alpha$
contains a smooth positive representative denoted by $\theta$.

\item[(4)] The class $f^\star\gamma- \alpha$ coincides with the first Chern class of $K_{X/Y}+ D- L$.
\end{enumerate}
\smallskip

\noindent In this context we prove the next statement.

\begin{theorem}\label{L2/m1}
Assume that conditions {\rm (1), (2), (3)} and {\rm (4)} hold and that for some sufficiently big and divisible integer $m> 0$ we have $\displaystyle H^0(X_y, mL|_{X_y})\neq 0$ for general $y\in Y$. Then the class
\[K_{X/Y}+ D+ \alpha\]
contains a closed positive current $\Theta\geq 0$ whose restriction to the general fiber of $f$ is (well-defined and) induced by 
the sections of $mL$ restricted to the fibers of $f$.
\end{theorem}

\begin{proof}
To begin with, we remark that if $D=0$ and if some multiple of $\alpha$ belongs to $H^2(X, \mathbb Z)$, then the matter is clear. Indeed, in this case we can choose a $\Q$-line bundle $F$ on $X$ whose Chern class is $\alpha$ and such that 
\[(K_{X/Y}+ D+ F)|_{X_y}\simeq L|_{X_y}\]
for all general $y\in Y$. The results proved in \cite{PT18} show that the current $\Theta$ constructed  
fiber-wise by the $m^{\rm th}$ root of the sections of $mL|_{X_y}$ is positive.
\smallskip

\noindent Even though $\alpha$ may not be a rational class, hypothesis (4) implies that this is the case locally over $Y$. 
This will allow us to conclude via an approach similar to the one in \cite{PT18}, \cite{CH20} (with slight modifications). The details are as follows.
\smallskip

\noindent We denote by $\gamma$ any closed, real $(1,1)$-form contained in the class $\gamma$ given by 
hypothesis (3) -and apologize for the abuse of notation. As consequence of the hypothesis (4) above, the $(1,1)$-form $\mu$ defined by the equality
\begin{equation}\label{new69}
\mu := \theta-f^{\star}(\gamma)
\end{equation}
is closed, real, and its corresponding class is \emph{rational}.
\smallskip

\noindent Let $h$ be a metric on $L-(K_{X/Y}+D)$ whose corresponding curvature form equals $\mu$ (here we are using the convention in (2) above). We consider a finite open cover $(U_i)_{i\in I}$ of $Y$ such that 
\begin{equation}\label{new80}
\gamma|_{U_i}= dd^c\tau_i 
\end{equation}
for some smooth real function $\tau_i$ defined on $U_i$. 

\noindent For each index $i$ we endow the restriction 
\[(L-(K_{X/Y}+D))|_{f^{-1}(U_i)}\]
with the metric $h_i:= e^{-\tau_i\circ f}h$. The equality
\begin{equation}\label{new2}
\theta|_{f^{-1}(U_i)}= dd^c(f^\star\tau_i)+ \mu|_{f^{-1}(U_i)} 
\end{equation}
shows that the curvature corresponding to the metric $h_i$ is equal to $\theta|_{f^{-1}(U_i)}$.
\smallskip

\noindent All in all, we can define a Hermitian $\Q$-line bundle $(F_i, h_i)$ on the inverse image $f^{-1}(U_i)$ such that the following hold.
\begin{enumerate}
\smallskip

\item[(a)] For each $i\in I$ the following holds 
\begin{equation}\label{new3}
(K_{X/Y}+ D)|_{f^{-1}(U_i)}+ F_i\simeq L|_{f^{-1}(U_i)}
\end{equation}
especially the sections of $\displaystyle mL|_{X_y}$ correspond to 
sections of $m(K_{X/Y}+ D+ F_i|_{X_y})$ for each general point $y\in U_i$, where $X_y:= f^{-1}(y)$.

\item[(b)] We describe here more precisely the metric $h_i$. Let $(V_j)_{j\in J}$ be an open covering of $X$, such that the restriction of the bundles $K_X, f^\star K_Y, mD, mL$
to each $V_j$ is trivial. Recall that we have fixed a metric $h$ on 
$L- K_{X/Y}- D$, and denote by $\rho_j$ its weight on the set $V_j$. Then the 
weight of the metric $h_i$ on the set $V_j\cap f^{-1}(U_i)$ is
\[\varphi_{ij}:= \rho_j|_{V_j\cap f^{-1}(U_i)}+ \tau_i\circ f|_{V_j\cap f^{-1}(U_i)}.\]
We stress the fact that the only "non-global" part of the metric $h_i$ corresponds to the pull-back of $\tau_i$.

\item[(c)] It follows that we have 
\begin{equation}\label{new4}
\sqrt{-1}\Theta(F_i, h_i)= \theta|_{f^{-1}(U_i)}. 
\end{equation}
and remark that even if $(F_i, h_i)$ is only locally defined (with respect to the base $Y$), the corresponding curvature 
is a global form on $X$.
\end{enumerate}
\medskip

\noindent Relation \eqref{new3} allows us to define a metric $h_{X/Y, i}$ on $(K_{X/Y}+ D)|_{f^{-1}(U_i)}+ F_i$ whose corresponding 
curvature is positive. This was done in \cite{PT18}, and we recall next the construction. Let $x_0\in f^{-1}(U_i)$ be an arbitrary point. We fix coordinates $(t_k)$ and $(z_l)$ on $U_i$ and near
$x_0$, respectively. Assume that $f$ is smooth over $Y\setminus \Sigma$. For each $y\in U_i\setminus \Sigma_i$ and $\displaystyle \xi\in V_{m, y}$ let 
\begin{equation}\label{new5}\Vert \xi\Vert^{\frac{2}{m}}_{y,i}:= \int_{X_y}|\xi|^{\frac{2}{m}}e^{-\varphi_D- \varphi_{i}} \end{equation}
be the $L^{2/m}$-seminorm on the space of sections
\[V_{m, y}:= H^0\big(X_y, mL_y\big)= H^0\big(X_y, m(K_{X_y}+ D_y+ F_{i,y})\big)\]
where the notations are explained below. The subscript $(\ldots )_y$ denotes restriction to the fiber $X_y$ and $m$ is assumed to be sufficiently divisible so that all divisors in question are Cartier. 
In \eqref{new5} the symbol $e^{-\varphi_i}$ means that we are using the metric $h_i$ (cf. (b) above) on the bundle $F_i$.
The section $\xi$ in \eqref{new5} is interpreted as a twisted pluricanonical form, so that the quantity under the integral is a $(n,n)$--form. 
\smallskip

\noindent Then the weight of the metric $h_{X/Y, i}$ at the point $x_0$ is equal to 
\begin{equation}\label{new6}
e^{\varphi_{X/Y, i}(x_0)}:= \sup_{\Vert \xi\Vert_{y_0, i}= 1}\left|\xi_0(x_0)\right|^{\frac{2}{m}}
\end{equation}
where $y_0:= f(x_0)$ and {$\xi_0$ is given by the equality}
\[\xi\wedge f^\star(dt^{\otimes m})= \xi_0dz^{\otimes m}\]
written locally near $x_0$. 
\smallskip

\noindent We have given such a detailed description of the metric $h_{X/Y, i}$ because thanks to it, it is easy to deduce its 
dependence on the index $i$. Indeed, we assume that we choose the same coordinates $t$ and $z$ on $U_i\cap U_k$
and near $x_0\in V_j$, respectively, where it is understood that $f(x_0)\in U_i\cap U_k$. Notice that the space of holomorphic sections $V_{m, y}$ involved in the definition of the relative metric is independent of $i$, but this may be not the case 
for the semi-norm \eqref{new5}. By (b) of \eqref{new3}, we can write
\begin{equation}\label{new8}
\int_{X_y}|\xi|^{\frac{2}{m}}e^{-\varphi_D- \varphi_{i}}= e^{-\tau_i(y)}
\int_{X_y}|\xi|^{\frac{2}{m}}e^{-\varphi_D- \rho_{j}},
\end{equation}
where the notation is indicating the weight $\varphi_D+ \rho_{j}$ we are using 
on the set $V_j\cap X_y$. This is a consequence of the definition of the metric $h_i$ in (b). Moreover,
we remark that the second factor of the product on the RHS of \eqref{new8} is independent on the index "$i$".

\noindent Thus, by \eqref{new5} we infer that the equality
\begin{equation}\label{new9}
\Vert \xi\Vert^{\frac{2}{m}}_{y,i}= e^{\tau_k(y)-\tau_i(y)}\Vert \xi\Vert^{\frac{2}{m}}_{y,k}
\end{equation}
holds for any point $y\in U_i\cap U_k$ and $\xi\in V_{m, y}$. 

\noindent Moreover, we can assume that the difference
\begin{equation}\label{new10}
\tau_k-\tau_i= \Re(\tau_{ik}) 
\end{equation}
is the real part of some holomorphic function $\tau_{ik}$ defined on the intersection $U_i\cap U_k$, since their respective Hessian forms coincide by \eqref{new2}.
\medskip

\noindent It follows that
\begin{equation}\label{new11}
\sup_{\Vert \xi\Vert_{y_0, i}= 1}\left|\xi_0(x_0)\right|= \sup_{\Vert \xi\Vert_{y_0, k}= e^{-\Re(\tau_{ik}(y_0))}}\left|\xi_0(x_0)\right|= e^{-m\Re(\tau_{ik}(y_0))}\sup_{\Vert \xi\Vert_{y_0, k}= 1}\left|\xi_0(x_0)\right|.
\end{equation}
Finally, we get 
\begin{equation}\label{new12}
{\varphi_{X/Y, i}(x_0)}= {\varphi_{X/Y, k}(x_0)- \frac{1}{2}\Re(\tau_{ik}(y_0))}
\end{equation} 
and since the point $x_0$ was arbitrary and $y_0= f(x_0)$ it follows that we have
\begin{equation}\label{new14}
{\varphi_{X/Y, i}}= {\varphi_{X/Y, k}- \frac{1}{2}\Re(\tau_{ik}\circ f)}
\end{equation}
locally near a fixed point on $f^{-1}(U_i\cap U_k)$. In particular we obtain the equality 
\begin{equation}\label{new13}
dd^c\varphi_{X/Y, i}= dd^c\varphi_{X/Y, k},
\end{equation}
on the overlapping $U_i$'s.
\medskip

\noindent In conclusion, \eqref{new13} shows that the curvature currents we construct locally on the base agree on the intersection of the corresponding sets,
and the construction of $\Theta$ is finished, since the positivity of this current was already established in \cite[Theorem 4.2.7]{PT18}.
\end{proof}

\subsection{Singularities of the metric}

\noindent In order to prove Theorem \ref{L2/m} we can apply Theorem \ref{L2/m1} for the following data:
$X:= S, Y:= W, D= \Xi_1$ and $L= \Xi_2$, together with $\alpha:= \beta$. The output is a 
current 
\[\Theta\in c_1(K_{S/W})+ (\beta+ \Xi_1)|_S\]
with the properties stated in the point (1) of Theorem \ref{L2/m}. 
The assertion (2) will be established along the following lines. 
\medskip

\begin{proof}[Proof of Theorem \ref{L2/m}, (2)]

\noindent To begin with, we recall an important class of manifolds on which $L^2$ methods can be applied.
\begin{defn}\label{pc}
A manifold/complex space $X$ is called weakly pseudo-convex if it admits a smooth, plurisubharmonic exhaustion function $\psi$, so that the closure of the sets 
$(\psi< C)\Subset X$ are compactly contained in $X$, for any constant $C$. 
\end{defn}

\noindent Obviously, compact holomorphic manifolds have this property, but this is equally the case for any complex space which admits a proper map into a Stein manifold. In particular, consider the map $f:S\to W$ given in Theorem \ref{L2/m}; for any Stein open subset $U\subset W$
the inverse image $f^{-1}(U)\subset S$ is an example of weakly pseudo-convex manifold which will be important in what follows.
\medskip

\noindent Consider next the blow-up map $\pi:\wh X\to X$ introduced at the beginning of Section 5, and denote by $E= \sum E_j$ the corresponding 
exceptional divisor. We define the following form 
\begin{equation}\label{new1}
\omega_{\wh X}:= \pi^\star \omega_X+ \sum a_i\theta_i
\end{equation}
where $\omega_X$ is a Kähler metric on the base $X$, the coefficients $a_i$ are positive rational numbers and 
the forms $\theta_i$ belong to the Chern class of $\mathcal O_{\hat X}(-E_i)$. By an appropriate choice of the 
coefficients $a_i$, we can assume that $\displaystyle \omega_{\wh X}> 0$ -- so we have a Kähler metric on 
$\wh X$ for which the only "transcendental" part is pulled-back from the base $X$.
\smallskip

\noindent Let $w_0\in W$ be an arbitrary point, and let $\Omega\subset X$ be a Stein co-ordinate subset which contains
the image $\iota_W(w_0)$, cf. diagram \eqref{dia1} below.  
\begin{equation}\label{dia1} 
\begin{tikzcd}
S \arrow{r}{\iota_S} \arrow[swap]{d}{f} & \wh X \arrow{d}{\pi}\\
W \arrow{r}{\iota_W} & X
\end{tikzcd}
\end{equation}

\noindent Consider the following sets 
\begin{equation}\label{new15} 
\wh X_\Omega:= \pi^{-1}(\Omega), \qquad U:= \iota_W^{-1}(\Omega),\qquad S_U:= f^{-1}(U) 
\end{equation}
contained in $\hat X$, $W$, and $S$, respectively. We have the following statement.
\begin{lemma}\label{ample} There exist Hermitian line bundles $(A_\Omega, h)\to \wh X_\Omega$ and $(A_U, h)\to U$
on $\wh X_\Omega$ and $U$, respectively such that the corresponding curvature forms are multiple of Kähler forms, i.e.
\[\sqrt{-1}\Theta(A_\Omega, h)= N\omega_{\wh X}|_{\wh X_\Omega}, \qquad \sqrt{-1}\Theta(A_U, h)= N\omega_{W}|_{U}\] 
where $N$ is positive and sfficiently divisible.  
\end{lemma}
\begin{proof}
This follows by standard arguments. Since the restriction $\displaystyle \omega_X|_\Omega$ is $dd^c$-exact, it can be interpreted as 
trivial bundle over $\Omega$ endowed with a non-trivial metric whose corresponding curvature form is $\displaystyle \omega_X|_\Omega$. The rest follows as consequence of \eqref{new1} -- in particular we only need the positive integer $N$ in order to clear the denominators of the coefficients $a_i$. A similar argument applies for $\omega _W$.
\end{proof}

\begin{remark}\label{rmk1} Note that we may assume $\iota _W:W\to X$ is given by a finite sequence of blow ups whose centers are contained in the singular locus $T_{\rm sing}$ of $T$. The metric $\omega_W$ can be obtained by the same formula as in \eqref{new1},
so that the corresponding exceptional divisors $E_i\subset W$ map into the singular locus of the centre $X$.
\end{remark}
\medskip

\noindent After these preparations, we proceed with the second part of Theorem \ref{L2/m}. Let $\Omega\subset X$
be an open subset as above. By the same procedure as in the proof of Theorem \ref{L2/m1}, we can construct a $\Q$-line bundle 
$(F_\Omega, h_F)\to \wh X_\Omega$ such that the following relations hold
\begin{equation}\label{new16} 
K_{\wh X}+ S+ \Xi_1+ F_\Omega\simeq \Xi_2, \qquad \sqrt{-1}\Theta(F_\Omega, h_F)= \beta 
\end{equation}
on $\wh X_\Omega\subset \wh X$. 
\medskip

\noindent On the other hand, let $\rho$ be any non-singular $(1,1)$-form on $W$, such that $\rho\in c_1(K_W)$. We can assume that the 
constant $N$ in Lemma \ref{ample} is large enough, so that the following inequality
\begin{equation}\label{new17} 
\rho+ N\omega_W\geq 0
\end{equation}
holds point-wise on $W$. We then consider the closed positive current 
\begin{equation}\label{new18} 
\Theta+ f^\star(\rho+ N\omega_W)\geq 0
\end{equation}
which belongs to the class $\displaystyle K_{S}+ (\Xi_1+ \beta)|_S+ Nf^\star(\omega_W)$. 

\noindent Next, given the expression of the metric $\omega_W$ combined with Remark \ref{rmk1}, there exist integers $k_i$ and 
divisors $E_i$ such that the current
\begin{equation}\label{new19} 
\wh \Theta:= \Theta+ f^\star(\rho+ N\omega_W)+ \sum k_i[E_i|_S]
\end{equation}
has the following properties
\begin{itemize}

\item It belongs to the cohomology class $\displaystyle (K_{X}+ S+ \Xi_1+ \beta+ N\pi^\star(\omega_X))|_S$ 

\item The divisors $E_i$ appearing in \eqref{new19} project into $T_{\rm sing}$. 
\end{itemize}
\medskip

\noindent When restricted to the set 
\[S_U= S\cap \pi^{-1}\Omega\]
(cf. \eqref{new15} for the notations), the class $\beta$ corresponds to the line bundle $\displaystyle F_\Omega|_S$, so by abuse of 
notation we can write
\begin{equation}\label{new20} 
\wh \Theta|_{S_U}\simeq K_{X}+ S+ \Xi_1+ F_\Omega|_{S_U}
\end{equation} 
by which we mean that the bundle on the RHS admits a singular metric $h_\theta= e^{-\varphi_\theta}$ defined on $S_U$ whose curvature form is precisely the 
restriction $\displaystyle \wh \Theta|_{S_U}$. 
\smallskip

\noindent For each $k\geq 1$ sufficiently divisible 
we consider the line bundle \[L_k:= (k(K_{X}+ S+ \Xi_1+ F_\Omega)+ A_\Omega)|_{S_U}\]
and the corresponding Hilbert space of homolorphic sections
\begin{equation}\label{new21} 
\mathcal H_k:= \{s\in H^0(S_U, L_k) / \int_{S_U}|s|^2e^{-k\varphi_{\theta}- \varphi_A}dV< \infty\}
\end{equation}
in the multiplier ideal induced by the current $k\wh \Theta|_{S_U}$. Then we recall the following result,
basically proved in \cite[Section 13]{Dem09} and references therein.
\begin{theorem}\label{JP}
Let $\wh \Theta_k\geq 0$ be the closed positive current on $S_U$ given by a family of orthonormal sections of $\mathcal H_k$. Then we have 
\[\nu(\wh \Theta, x)= \lim_k\frac{1}{k}\nu({\wh \Theta_k, x})\]
where $x\in S_U$ is an arbitrary point and where we denote by $\nu(\wh \Theta, x)$ the Lelong number of 
$\wh \Theta$ at $x$.
\end{theorem}
\medskip

\begin{remark} We note that in \emph{loc. cit.} the result above is established in the setting of bounded
pseudo-convex subsets in $\C^n$, but the proof applies in the context of Theorem \ref{JP}, so we will not reproduce it here. As a matter of fact, it is at this point that the pseudo-convexity of the set 
$S_U$ (cf. Definition \ref{pc}) is very important.
\end{remark}
\smallskip

\noindent In other words, in order to evaluate the singularities of $\wh \Theta$ it would suffice to have a uniform lower bound for the vanishing orders of
the sections $s\in \mathcal H_k$ as $k\to \infty$. To this end, we recall that as a consequence of the results e.g. in \cite{BP10} the following local version of the \emph{invariance of plurigenera} holds true.

\begin{theorem}\label{thm-pg}
In the above set-up, any holomorphic section $s$ of the bundle $L_k$ extends to $\wh X_\Omega$ as section $\widetilde s$ of 
$k(K_{X}+ S+ \Xi_1|_{\wh X_\Omega}+ F_\Omega)+ A_\Omega.$
\end{theorem}
\noindent We offer next a few explanations about \ref{thm-pg} in the very particular case 
in which we have to extend a section $s$ of the bundle $\displaystyle k(K_{X_\Omega}+ S+ L)|_{S_U}$, where $(L, h_L)$ is a semi-positively curved line bundle, such that $h_L$ is non-singular and it is defined over $X_{\Omega'}$ for some $\Omega\Subset \Omega'$. As we have seen above, we have an ample line bundle $A_\Omega$ over $X_\Omega$, and thus, in order to 
construct the extension of $s$ we need the following.
\begin{itemize}

\item \emph{A local version of the Ohsawa-Takegoshi extension theorem}. The statement we need is 
available, cf. \cite{DemOT}.

\item \emph{A finite family of holomorphic sections for the bundles \[\displaystyle (k+r)(K_{X_\Omega}+ S+ L)+ C(k)A_\Omega|_{X_\Omega}\] for $r=0,\dots, k$ such that for each $r$, 
their common set of zeroes is empty.} This is easy to see, despite of the fact that $S_U$ and $X_\Omega$ are not compact: the point is that all the bundles/metrics extend over $X_{\Omega'}$ and we construct our sections by a quick compactness argument.  
\end{itemize}
These two points granted, one follows the usual algorithm, see e.g.\cite{BP10} and the references therein. 

However, in our case there is an additional level of difficulty, induced by the presence of the
$\Q$-divisor $\Xi_1$. This can also be treated by the known techniques (i.e. work by Hacon-McKernan, Ein-Popa ...), given the fact that 
$S+\Xi_1+ \Xi_2$ is snc and the relation \eqref{new16} (and therefore the condition (7) in \emph{loc. cit.} is automatically satisfied).

\medskip

\noindent Now, by relation\eqref{new16}, the extension $\widetilde s$ can be seen as section of the bundle
$\displaystyle k\Xi_2|_{\wh X_\Omega}+ A_\Omega$. Given that the support of $\Xi_2$ is $\pi$-contractible 
and that the set $\wh X_\Omega$ is the inverse image of $\Omega$ by the map $\pi$, it follows that the vanishing order of 
$\widetilde s$ along $\Xi_2$ is at least $k-k_0$ (we "loose" a fixed amount $m_0$ because of the ample bundle $A_\Omega$). 
\medskip

\noindent In conclusion, it follows that we have 
\begin{equation}\label{new22} 
\wh \Theta_k\geq (k-k_0)[\Xi_2]|_{S_U}
\end{equation}
and the proof is finished by using Theorem \ref{JP}. \end{proof}
\medskip


\section{Proof of Theorem \ref{lelong}}

\noindent The main steps of the proof of our version of the canonical bundle formula -- Theorem \ref{lelong} -- 
are as follows. Let $\Xi:=\{ B\}$ be the fractional part of $B$ and we write $B= \Xi+\lfloor B\rfloor =\Xi- \lceil -B\rceil$ as difference of two effective $\Q$-divisors. We assume that the discriminant divisor $B_W$ is equal to zero (we can do this without altering any of our hypothesis).  We then have the numerical identity 
\[K_{S/W} +\Xi+ \beta\simeq f^\star\beta_W+ \lceil-B\rceil.\]
Next, we apply the methods already used in the proof of Theorem \ref{L2/m}
in order to construct a closed positive current $\Theta\geq 0$ in the class corresponding to the LHS of the 
relation above. The said current is proved to be singular along $\lceil-B\rceil$: this follows as 
consequence of the hypothesis (b) (which replaces the fact that the map $f:S\to W$ might not be induced by a log-resolution $\pi$). 

The heart of the matter is to show the (highly non-trivial) fact that 
the Lelong numbers of the difference
\[\Theta- \lceil-B\rceil\]
are equal to zero. To this end we adapt the method used in \cite{Taka} in our context.
\medskip

\noindent We start with a general discussion --and a simple result-- concerning fiber integrals. 

\subsubsection{Fiber integrals}\label{fiber} Let $p: X\to Y$ be a proper, surjective holomorphic map, where $X$ is a $(n+m)$-dimensional 
Kähler manifold and $Y$ is the unit disk in $\C^m$. We denote by $Y_0\subset Y$ the set of regular values of $p$.
Let $t= (t_1,\dots, t_m)$ be coordinates on $\C^m$ induced by a fixed base. Consider a $\mathbb Q$-line bundle $(L, h_L)$ on the total space $X$, endowed with a metric 
$h_L$ eventually singular, but whose curvature is semi-positive. Let $s\in H^0\big(X, k(K_{X}+ L)\big)$ be a pluricanonical form with values in $kL$,
where $k$ is a positive, sufficiently divisible integer so that $kL$ is a line bundle. 
For each $y\in Y_0$ let $\displaystyle s_y\in H^0\big(X_y, k(K_{X_y}+ L_y)\big)$ be the induced form on $X_y$, in the sense that 
\begin{equation}\label{new40} 
s|_{X_y}= s_y\wedge p^\star (dt)^{\otimes k}.
\end{equation}
\smallskip

\noindent In this setting we show that the following holds true.
\begin{lemma}\label{OT} We assume moreover that there exists a section $\sigma$ of a line bundle $\Lambda$ such that the quotient 
$\displaystyle \frac{s}{\sigma}$ is a holomorphic section of $k(K_{X}+ L)- \Lambda$.
There exists a positive constant $C_0> 0$ independent of $s$ such that the inequality 
\begin{equation}\label{new27} 
\int_{X_y}|s_y|^{\frac{2}{k}}e^{-\varphi_L}\geq C_0\sup_{X_y} \left|\frac{s}{\sigma}\right|^{\frac{2}{k}}
\end{equation} 
holds for any $y\in Y_0$ such that $|y|<\frac{1}{2}$. The norm on the RHS is with respect to a fixed, smooth metric on $k(K_X+ L)- \Lambda$,
and an upper bound for the constant $C_0$ can be obtained from the proof that follows. 
\end{lemma}

\begin{proof}

\noindent Let $z_0\in X_y$ be a point such that \[
\sup_{X_y} \left|\frac{s}{\sigma}\right|^{\frac{2}{k}}= \left|\frac{s}{\sigma}\right|^{\frac{2}{k}}(z_0)\]
and let $y= f(z_0)$ be its image. We take the local coordinates $z=(z_1, \dots, z_{n+m})$ and $t= (t_1,\dots, t_m)$
centred at $z_0$ and $y$ respectively. The $t$-coordinates are defined on some open set $\Omega\subset Y$, and the 
$z$-coordinates are defined on $V\subset f^{-1}(\Omega)$ biholomorphic to the unit ball in $\mathbb C^{n+m}$. Let $\omega$ be an arbitrary Kähler metric on $X$. 
\smallskip

\noindent Corresponding to this data
we define the function $\psi: V\to \mathbb R\cup \{-\infty\}$ as follows
\begin{equation}\label{new38}
\omega^n\wedge f^\star (\sqrt{-1}dt\wedge d\ol t)= e^\psi \sqrt{-1}dz\wedge d\ol z
\end{equation}
where we use the notations 
\[\sqrt{-1}dt\wedge d\ol t:= \prod_{i=1}^m\sqrt{-1}dt_i\wedge d\ol t_i, \qquad  \sqrt{-1}dz\wedge d\ol z:= \prod_{i=1}^{m+n}\sqrt{-1}dz_i\wedge d\ol z_i.\]
Therefore, the restriction of the form $\displaystyle e^{-\psi} \omega^n$ to the fiber $X_y\cap V$ is equal to the measure sometimes denoted with
$\displaystyle \left|\frac{dz}{f^\star dt}\right|^2$
on $X_y$. 
\smallskip

\noindent We assume that the bundles $L$ and $\Lambda$ are trivial when restricted to $V$, and let $u\in \mathcal O(V)$ be the local 
holomorphic function corresponding to the section $s|_V$. Then we clearly have the inequality
\begin{equation}\label{new40} 
\int_{X_y\cap V}|u|^{\frac{2}{k}}e^{-\varphi_L- \psi}\omega^n\leq \int_{X_y}|s_y|^{\frac{2}{k}}e^{-\varphi_L}.
\end{equation}
\smallskip

\noindent On the other hand, by the $L^{\frac{2}{k}}$-version of the Ohsawa-Takegoshi theorem established in \cite{PT18}, Proposition 1.2 there exists a 
function $U\in \mathcal O(V)$ such that 
\begin{equation}\label{new41} 
U|_{V\cap X_y}= u|_{V\cap X_y}, \qquad \int_{V}|U|^{\frac{2}{k}}e^{-\varphi_L}d\lambda\leq C_{\rm univ} \int_{X_y\cap V}|u|^{\frac{2}{k}}e^{-\varphi_L- \psi}\omega^n,
\end{equation} 
where $C_{\rm univ}$ is a numerical constant. 
\smallskip

\noindent Let now $\sigma_V\in \mathcal O(V)$ be the local holomorphic function induced by the section $\sigma$ and let $N\gg 0$
be a large enough integer so that the integral
\begin{equation}\label{new42} 
\int_V\frac{d\lambda}{|\sigma_V|^{\frac{2}{N}}}< \infty
\end{equation} 
is convergent. The {first part} of \eqref{new41} combined with the fact that $z_0$ is the maximum point and the mean-value inequality gives
\begin{equation}\label{new43} 
\sup_{X_y} \left|\frac{s}{\sigma}\right|^{\frac{2}{k}}\leq C\left|\frac{U}{\sigma_V}(z_0)\right|^{\frac{1}{kN}}\leq C\int_V\left|\frac{U}{\sigma_V}\right|^{\frac{1}{kN}}d\lambda
\end{equation}
where the first constant is due to the fixed metric on $k(K_X+ L)- \Lambda$ and it follows -thanks to Hölder inequality- that 
\begin{equation}\label{new44} 
\sup_{X_y} \left|\frac{s}{\sigma}\right|^{\frac{2}{k}} \leq C_0\int_{V}|U|^{\frac{2}{k}}e^{-\varphi_L}d\lambda,
\end{equation}
where $C_0$ depends on \eqref{new43}, and an upper bound for $\varphi_L$. This inequality, combined with 
the estimate in \eqref{new41} completes the proof of Lemma \ref{OT}.
\end{proof}

\subsubsection{Pseudo-effectivity}

\noindent We remark that we can assume $B_W= 0$, by simply replacing $B$ with $B- f^\star(B_W)$ and noticing that 
under the transversality hypothesis in our statement, the new pair $(S, B)$ is sub-klt and moreover the hypothesis (b) still 
holds. 



\medskip

\noindent Under the assumption that $B_W= 0$, it follows from the hypothesis (c) of Theorem \ref{lelong} that we have 
\begin{equation}\label{new28}
K_{S/W} +B+ \beta\simeq f^\star\beta_W.
\end{equation} 
Since the pair $(S, B)$ is sub-klt, we can write 
$B= 
\{B\}+\lfloor B\rfloor := \Xi- \lceil-B\rceil$ and therefore we obtain 
\begin{equation}\label{new29}
K_{S/W} +\Xi+ \beta\simeq f^\star\beta_W+ \lceil-B\rceil,
\end{equation} 
where $(S, \Xi)$ is klt and $\lceil-B\rceil$ is effective, with integer coefficients. We apply Theorem \ref{L2/m1} to the following data: $X:= S, Y:= W,
\alpha:= \beta, \gamma:= \beta_W, L:= \mathcal O_S(\lceil-B\rceil)$ and finally $D:= \Xi$. It follows that there is a closed positive current 
$\Theta\geq 0$ in the class \eqref{new29}, induced by the sections of $\displaystyle \mathcal O_S(\lceil-B\rceil)|_{S_w}$ for $w\in W$ general. 
\smallskip

\noindent We then formulate our next assertion:
\begin{claim}\label{singsing}
The inequality \[\Theta\geq \lceil-B\rceil\] holds in the sense of currents on $S$, where the RHS is interpreted as current of integration on the divisor 
$\lceil-B\rceil$.
\end{claim}

\begin{proof}[Proof of the Claim]
We start with a little comment: if a hypersurface $Y\subset S$ belongs to the support of the divisor $\lceil-B\rceil$ is such that $f(Y)= W$ (i.e. $Y$ is horizontal
with respect to the map $f$), then the hypothesis (c) together with the construction of $\Theta$ show immediately that $\Theta\geq \mu[Y]$,
where $\mu$ is the multiplicity of $\lceil-B\rceil$ along $Y$. However, things are less clear for the vertical part of the support of $\lceil-B\rceil$, 
since we only have the explicit expression of $\Theta$ over general points of $W$. It is at this point that the techniques from the subsection
\ref{fiber} come into play. The argument which follows has its origins in \cite{BP08}, as well as in \cite{CH20}. The reason why we review it here is to show that it can be easily adapted to the pluricanonical case, needed a bit later. 
\smallskip

\noindent Let $w\in W$ be any regular value of the map $f$, and let $w\in \Omega\subset W$ be a coordinate set of $W$, biholomorphic with the unit ball in $\mathbb C^m$. Recall from the proof of Theorem \ref{L2/m1} that there exists a Hermitian line bundle $(F, h_F)$ defined over 
$f^{-1}(\Omega)$, whose associated curvature form is equal to $\beta$, and such that restricting to $f^{-1}(\Omega)$ we have
\begin{equation}\label{new30}
K_{S/W} +\Xi+ F\simeq \lceil-B\rceil
 .\end{equation}   

\noindent Next, let $u$ be any holomorphic section of $\displaystyle K_{S_w} +(\Xi+ F)|_{S_w}$, such that 
\begin{equation}\label{new31}
\int_{S_w}|u|^2e^{-\varphi_\Xi- \varphi_F}= 1.
\end{equation}
As recalled in \ref{fiber}, there exists a section $U$ of $K_{S} +\Xi+ F$ such that 
\begin{equation}\label{new32}
U|_{S_w}= u\wedge f^\star(dt), \qquad \int_{f^{-1}(\Omega)}|U|^2e^{-\varphi_\Xi- \varphi_F}\leq C_0.
\end{equation}
Since the canonical bundle of $W$ is trivial when restricted to $\Omega$, we can interpret $U$ as a section of $K_{S/W} +\Xi+ F$, which is the same as
$\displaystyle \lceil-B\rceil|_{f^{-1}(\Omega)}$ thanks to $\eqref{new30}$. In particular, by (b) the quotient 
\[\tau:= \frac{U}{s_{\lceil-B\rceil}}\]
    becomes a \emph{holomorphic} function on $f^{-1}(\Omega)$, where $\displaystyle s_{\lceil-B\rceil}$ is the canonical section of 
$\mathcal O(\lceil-B\rceil)$.  

\noindent By Lemma \ref{OT}
we infer the following inequality
\begin{equation}\label{new44}
\sup_{S_w}|\tau|\leq C,
\end{equation}
--because of the normalisation \eqref{new31}-- where the constant $C$ in \eqref{new44} is independent of $u$.
\smallskip

\noindent As consequence of \eqref{new31} combined with the 
definition of the relative metric we obtain
\begin{equation}\label{new36}
\varphi_{S/W}\leq C+ \log |s_{\lceil-B\rceil}|^2,
\end{equation}
from which our claim follows.
\end{proof}
\medskip

\subsubsection{Lelong numbers}

\noindent Next we show that the Lelong numbers of the closed positive current \begin{equation}\label{new45}T:= \Theta- \lceil-B\rceil\end{equation} are equal to zero. To this end
we will use an important result due to S. Takayama. Actually we will "extract" from the proof in \cite{Taka} the result below (which will be useful for the proof of Theorem \ref{lelong_1} as well). To begin with, we recall the construction of a natural metric on $K_{S/W}$, cf. \cite{MP12}.
\smallskip

\noindent Let $z_0\in S$ be an arbitrary point, and $t_0= f(z_0)$ be its image. We take coordinates $z=(z_1, \dots, z_{n+m})$ and $t= (t_1,\dots, t_m)$
centred at $z_0$ and $t_0$ respectively. The $t$-coordinates are defined on some open set $\Omega\subset W$, and the 
$z$-coordinates are defined on $V\subset f^{-1}(\Omega)$. Let $\omega$ be an arbitrary Kähler metric on $S$. Corresponding to this data
we define the function $\psi$ as follows
\begin{equation}\label{new38}
\omega^n\wedge f^\star (\sqrt{-1}dt\wedge d\ol t)= e^\psi \sqrt{-1}dz\wedge d\ol z
\end{equation}
where we use the notations 
\[\sqrt{-1}dt\wedge d\ol t:= \prod_{i=1}^m\sqrt{-1}dt_i\wedge d\ol t_i, \qquad  \sqrt{-1}dz\wedge d\ol z:= \prod_{i=1}^{m+n}\sqrt{-1}dz_i\wedge d\ol z_i.\]

\noindent We now consider a covering of $S$ and $W$ with coordinates sets as above. Given the equality \eqref{new38}, the resulting 
functions $e^{-\psi}$  
define a metric $h$ on the relative canonical bundle $K_{S/W}$, which is general is singular. Let $h_0= e^{-\psi_0}$ be an arbitrary, smooth
metric on $K_{S/W}$. The difference of the weights corresponding to the two metrics
\[\psi_f:= \psi-\psi_0\]
is a global function on $S$. 
\smallskip

\noindent For each regular value $t\in W$ of $f$
we define the function
\begin{equation}\label{new37}
F(w):= \int_{S_w}e^{-\psi_f- \phi_B}\omega^n
\end{equation} 
where $\phi_B:= \log|s_B|^2$ is the log of the norm of the canonical section of the divisor $B$ (which we recall, is not necessarily effective). 
\medskip

\noindent We remark that so far the hypothesis $(\star)$ was not used in our arguments. It comes into play through the following important result, established in \cite{Taka}. Although it is not stated in this form explicitly, it is a direct consequence of the 
proof of Theorem 3.1 in \emph{loc. cit.} 
\begin{theorem}\cite{Taka}\label{taka}
Assume that the hypothesis of Theorem \ref{lelong} are satisfied, as well as the following.
\begin{itemize}

\item The divisor $B_W=0$ is zero.

\item given any point $z_0\in S$ and $w_0= f(z_0)\in W$ there exist
local coordinates $(x_1,\dots, x_{n+m})$ on $S$ centred at $z_0$ and $(t_1,\dots, t_m)$ on $W$ centred at $w_0$ such that 
$\displaystyle t_i\circ f= \prod x_j^{k_{ij}}$ where
the $k_{ij}$ are non-negative integers such that $k_{ij}\neq 0$
 for at most one $i$ for each index $j$.
 \end{itemize}
Then for any point $w_0\in W$ the following inequality holds
\[F(w)\leq C\prod_j \log\frac{1}{|w_j|}\]
where $C> 0$ is a positive constant and $w$ are coordinates centred at any $w_0$. 
\end{theorem}
\medskip

\begin{remark} For the comfort of the readers, we will provide a complete proof of 
Theorem \ref{taka} in the Appendix of this article. \end{remark} 
\medskip

\begin{remark}
We note that the inequality in Theorem \ref{taka}
holds for any $w$ such that $\mu_W(w)$ belongs to the set of regular values of $f$, and the constant $C$ is uniform.
\end{remark} 
\medskip

\noindent Consider $z_0\in S$ and $w_0:= f(z_0)$, together with the corresponding local coordinates chosen as in Theorem \ref{taka}. 
By the definition of the relative metric we have 
\begin{equation}\label{new46}
e^{\varphi_{S/W}(z)}\geq \frac{|f_{\lceil-B\rceil}|^2}{\int_{S_w}|s_{\lceil-B\rceil, w}|^2e^{-\varphi_\Xi-\varphi_F}}
\end{equation} 
where we denote by $s_{\lceil-B\rceil, w}$ the restriction of the section $s_{\lceil-B\rceil}$ to the fiber $S_w$, 
so that all in all the expression $\displaystyle |s_{\lceil-B\rceil, w}|^2e^{-\varphi_\Xi-\varphi_F}$ is a volume form on $S_w$.

\noindent Moreover, given the definition of the divisors $\Xi$ and $\lceil-B\rceil$, we have 
\begin{equation}\label{new47}
|s_{\lceil-B\rceil, w}|^2e^{-\varphi_\Xi-\varphi_F}\leq Ce^{-\psi_f- \phi_B}\omega^n|_{S_w}
\end{equation}
for some constant $C> 0$ (remark that the proximity of $w$ to the singular loci of $f$ 
is luckily irrelevant for the uniformity of $C$). 
\medskip

\noindent Then we have the following inequality for the potential $\varphi_T$ of the current $T$ introduced in \eqref{new45}
\begin{equation}\label{new48}
e^{\varphi_{T}(z)}\geq \frac{C}{F(w)}
\end{equation} 
where $w= f(z)$. The second bullet in Theorem \ref{taka} together with the upper bound for the function $F$ provided by this result 
show that 
\begin{equation}\label{new49}
\nu(T, z_0)= 0,
\end{equation}
and therefore Theorem \ref{lelong} is completely proved, modulo the regularisation theorem in \cite{Dem92} (a class containing a closed positive current whose Lelong numbers are equal to zero is nef)

\begin{remark}
In general, a nef cohomology class does not necessarily contain a closed positive current with zero Lelong numbers. Therefore, the property we are establishing in the proof of Theorem \ref{lelong} 
is stronger than neffness. Moreover, we construct the current $T$ is a very explicit manner, so in principle it should be possible to further analyze its singularities.
\end{remark}

\begin{remark}
We expect that a more general form of Theorem \ref{lelong} holds true, 
namely one should obtain a version of this result in the absence of the hypothesis $\star$. This promises to be a difficult problem (given the arguments invoked to prove it in \cite{Taka}).
\end{remark}

\bigskip

\section{Proof of Theorem \ref{lelong_1}}

The main steps of the proof that follows are the same as in the previous subsection. To begin with, recall that by hypothesis (d) we have a Kähler metric $\omega$ on $S$ such that 
\begin{equation}\label{new51}
\omega= f^\star g+ \theta,
\end{equation}
where $g$ is a Kähler metric on $W$ and $\theta$ is a \emph{rational} $(1,1)$-form on $S$. 

\noindent Consider a positive integer $m_0$ divisible enough such that $m_0\theta\in H^2(S, \mathbb Z)$. For each $k\geq 1$ the class
\begin{equation}\label{new52}
\beta_k:= \beta+ \frac{m_0}{k}\omega
\end{equation}
contains a positive representative, since $\beta$ is nef. Therefore, we can write  
\begin{equation}\label{new53}
K_{S/W}+ B+ \beta_k= f^\star (\gamma_k)+ A_k
\end{equation}
where \[ \gamma_k:= \gamma+ \frac{m_0}{k}g, \qquad \frac{m_0}{k} \theta\in c_1(A_k)\]
i.e. $A_k$ is a $\mathbb Q$-bundle such that $kA_k$ becomes a holomorphic line bundle which admits a metric $h_k$ whose curvature is 
precisely $m_0\theta$. 
\medskip

\noindent We therefore find ourselves in the framework of Theorem \ref{L2/m1}: we obtain a closed positive current $\Theta_k$ belonging to the class
$\displaystyle K_{S/W}+ \Xi+ \beta_k$, constructed by using the global sections of 
\[k(A_k+ \lceil-B\rceil)|_{S_w}\]
for $w\in W$ generic. 
\smallskip

\noindent Moreover, by hypothesis (b') together with the arguments in sub-section \ref{fiber} and the Claim \ref{singsing}, we infer that we have
\begin{equation}\label{new54}
\Theta_k\geq (1-\delta_k)\lceil-B\rceil
\end{equation}
where $\delta_k\to 0$ as $k\to \infty$. 

\noindent Finally, we analyse next the the singularities of the closed positive current 
\begin{equation}\label{new55}
T_k:= \Theta_k- (1-\delta_k)\lceil-B\rceil,\qquad  T_k\in K_{S/W}+ B+ \beta_k +\delta_k\lceil-B\rceil.
\end{equation}

\noindent To this end, we first observe that for each co-ordinate ball $\Omega\subset W$ the restriction \begin{equation}\label{new56} 
kA_k|_V\end{equation}
admits a metric whose curvature is equal to $m_0\omega|_V$, where $V:= f^{-1}(\Omega)$. We can assume that $m_0$ is large enough, 
so that the bundle in \eqref{new56} is generated by its global sections.
 
\noindent As in the proof of Theorem \ref{lelong} assume that the morphism $f$ satisfies the additional hypothesis in the statement of Theorem \ref{taka}.
We consider $z_0\in S$ such that $w_0:= f(z_0)\in \Omega$, and let $x$ and $t$ be coordinates having the second bullet property in Theorem \ref{taka}.
Let $u$ be a holomorphic section of the bundle $kA_k|_V$, such that $z_0\not\in (u=0)$. The product
\[\rho:= u\otimes s_{\lceil-B\rceil}^{\otimes k}\]  
can be interpreted as section of $\displaystyle k(K_{S/W}+ \Xi+ F_k|_V)$ (notations as in Section \ref{s-pos})  
 and by the definition of the $L^{2/k}$-metric we get
\begin{equation}\label{new57}
e^{\varphi_{S/W}(z)}\geq \frac{|f_\rho(z)|^{\frac{2}{k}}}{\int_{S_w}|\rho_w|^{\frac{2}{k}}e^{-\varphi_\Xi- \varphi_F}}
\end{equation}
for any point $z$ near $z_0$ and $w:= f(z)$. In \eqref{new56} we denote by $f_\rho$
the local holomorphic function corresponding to the section 
$\rho$. 

\noindent Inequality \eqref{new47} still applies, so we infer that
\begin{equation}\label{new58}
\nu(T_k, z_0)\leq \delta_k\nu(\lceil-B\rceil, z_0)
\end{equation}
given the definition of $\rho$, provided that we assume beforehand that $B_W= 0$ so that we can use Theorem \ref{taka}. Now the quantity $\delta_k$ is independent of the point $z_0$, and it follows that 
\begin{equation}\label{new59}
\sup_{z\in Z}\nu(T_k, z)\leq \delta_kC,
\end{equation}
where $C$ is the maximum  multiplicity of the divisor $\lceil-B\rceil$ at points of $S$. 

\noindent We next use \cite{Dem92} in order to obtain a smooth representative $\widetilde T_k\in K_{S/W}+ B+ \beta_k +\delta_k\lceil-B\rceil$ such that 
\[\widetilde T_k\geq -C_S\delta_k\omega\]
where the constant $C_S$ only depends on the geometry of $(S, \omega)$. Since $k$ was arbitrary, it follows that the class
\[K_{S/W}+ B+ \beta= f^\star\gamma\]
is nef, hence the same is true for $\gamma$ by Theorem \ref{mp}. 
 
\medskip

\section{Appendix}

\noindent The main result of this subsection is a direct argument for Theorem \ref{taka}. We first fix a few notations:

\begin{itemize}

\item $U\subset S$ is an open subset of $S$ small enough so that we have the coordinates $x= (x_1,\dots, x_{n+m})$ on $U$
with the property that 
\[f_1(x)= \prod_{i=1}^{l_1} x_i^{a_i}, \quad f_2(x)= \prod_{i=l_1+1}^{l_2} x_i^{a_i},\quad \dots, \quad f_m(x)= \prod_{i=l_{m-1}+1}^{l_m} x_i^{a_i}\]
where $a_i\geq 1$ and $0=l_0< l_1<\dots < l_m\leq n+m$ and $f_i:= t_i\circ f$ for $i= 1,\dots, m$. 

\item For every multi-index $I:= (i_1,\dots, i_m)$ such that $i_k\in J_k:= \{l_k+1,\dots l_{k+1}\}$ we define a form of by-degree $(n, n)$ through the formula
\[\sqrt{-1}dx_I\wedge d\ol x_I:= \Big(\prod_{k=1}^m\prod_{i\in J_k, i\neq i_k} \sqrt{-1}dx_i\wedge d\ol x_i \Big)\wedge\prod_{i= l_m+1}^{n+m} \sqrt{-1}dx_i\wedge d\ol x_i\]

\item $\omega:= \sum_i\sqrt{-1}dx_i\wedge d\ol x_i$ is the local version of the reference Kähler metric. We set 
\[\omega_l:= \sum_I\sqrt{-1}dx_I\wedge d\ol x_I\]
and up to a constant, we see that we have 
\[\omega_l= \big(\sum_{i=1}^{l_1}\sqrt{-1}dx_i\wedge d\ol x_i\big)^{l_1-1}\wedge\dots\wedge \big(\sum_{i=l_{m-1}+1}^{l_{m}}\sqrt{-1}dx_i\wedge d\ol x_i\big)^{l_m-1}\wedge\bigwedge_{i= l_m+1}^{n+m} \sqrt{-1}dx_i\wedge d\ol x_i\]
\end{itemize}

\noindent We now proceed to the evaluation of the function $\psi$ defined in \eqref{new38}. The first remark is that
\begin{equation}\label{new60}
\frac{df_i}{f_i}= \sum_{j= l_{i-1}+1}^{l_i}a_i\frac{dx_i}{x_i}
\end{equation}
and therefore a few simple calculations that we skip show that we have
\begin{equation}\label{new61}
e^\psi\simeq \big(\prod_{i=1}^m |f_i|^2\big)\prod_{k=1}^m\sum_{i=l_{k-1}+1}^{l_k}\frac{1}{|x_i|^2}
\end{equation}
where the symbol $"\simeq"$ in \eqref{new61} means that the quotient of the two functions is two-sided bounded away from zero.
\smallskip

\noindent The hypothesis $B_W= 0$ shows that the inequality 
\begin{equation}\label{new62}
e^{-\psi- \phi_B}\leq C \prod_{k=1}^m\frac{1}{|X_k|^{2d_k-2}}\prod_{j\in J}\frac{1}{|x_j|^{2-2\beta_j}}=: \Lambda(x)
\end{equation}
holds, where $\displaystyle d_k:= l_k- l_{k-1}$ and $\displaystyle |X_k|^2:= \sum_{i=l_{k-1}+1}^{l_k}|x_i|^2$.
Moreover, we have $J\subset \{l_m+1, \dots, m+n\}$ and $\beta_j> 0$ for each $j$.
\smallskip

\noindent On the other hand, by the Poincaré-Lelong formula, the local version of the quantity we have to analyse equals
\begin{equation}\label{new63}
F_U(t):= \int_U\theta(x)e^{-\psi- \phi_B}\bigwedge_{i=1}^m dd^c\log|t_i- f_i(x)|^2\wedge \omega^n
\end{equation}
where $\theta$ is a truncation function defined as follows
\[\theta(x):= \theta(|X'|^2)\prod_{i=1}^m\theta(|X_i|^2)\]
and $\displaystyle |X'|^2:= \sum_{i=l_{m}+1}^{n+m}|x_i|^2$.
\smallskip

\noindent Now, given the expression of the functions $f_i$, we infer that the integral \eqref{new63} is bounded by the following expression
\begin{equation}\label{new64}
\int_U\theta(x)\Lambda(x)\omega_l\wedge \bigwedge_{i=1}^m dd^c\log|t_i- f_i(x)|^2 
\end{equation} 
up to a constant independent of $t$, where we recall that the function $\Lambda$ was defined in \eqref{new62}. Indeed, 
the equality
\[\bigwedge_{i=1}^m dd^c\log|t_i- f_i(x)|^2\wedge \omega^n= \omega_l\wedge \bigwedge_{i=1}^m dd^c\log|t_i- f_i(x)|^2\] holds modulo a constant, because $f_i$ only depends on the variables $\displaystyle x_{l_{i-1}+ 1}, \dots, x_{l_i}$.

\smallskip

\noindent In order to evaluate \eqref{new64} we use integration by parts: this expression is the same as
\begin{equation}\label{new65}
\int_U\log|t_1- f_1(x)|^2 dd^c\big(\theta \Lambda \big)\wedge \omega_l\wedge \bigwedge_{i=2}^m dd^c\log|t_i- f_i(x)|^2 
\end{equation} 
and we will consider first the term containing $\displaystyle \theta(x) dd^c\Lambda$. 

\noindent Since $\displaystyle dx^\alpha\wedge d\ol x^\beta\wedge \omega_l\wedge \bigwedge_{i=2}^m dd^c\log|t_i- f_i(x)|^2= 0$ 
if $\max(\alpha, \beta)\geq l_1+1$, we see that the integral
\[dd^c\Lambda \wedge \omega_l\wedge \bigwedge_{i=2}^m dd^c\log|t_i- f_i(x)|^2\]
is equal to
\begin{equation}\label{new66}\small
\prod_{k=2}^m\frac{1}{|X_k|^{2d_k-2}}\prod_{j\in J}\frac{1}{|x_j|^{2-2\beta_j}}dd^c\frac{1}{|X_1|^{2d_1-2}}\wedge \omega_l\wedge \bigwedge_{i=2}^m dd^c\log|t_i- f_i(x)|^2.
\end{equation} 
\smallskip

\noindent We recall that we have the equality 
\[dd^c\frac{1}{|X_1|^{2d_1-2}}\wedge dd^c|X_1|^{2}= \delta_0,\]
the Dirac distribution at the origin in $\C^{l_1}$, so that we have 
\begin{equation}\label{new667}
\int_U\log|t_1- f_1(x)|^2\theta dd^c\big(\Lambda \big)\wedge \omega_l\wedge \bigwedge_{i=2}^m dd^c\log|t_i- f_i(x)|^2 =\end{equation} 
\begin{equation*}
\log\frac{1}{|t_1|^2}
\int_{U'}\theta_1 \Lambda_1\bigwedge_{i=2}^m dd^c\log|t_i- f_i(x)|^2\wedge\omega_{l'} 
\end{equation*} 
where $U'\subset \C^{n+m-l_1}$ is the unit ball, $x'= (x_{l_1+1},\dots, x_{n+m})$ and 
\[\theta_1(x'):= \theta(|X'|^2)\prod_{i=2}^m\theta(|X_i|^2), \quad l'= (l_2,\dots, l_m), \qquad \Lambda_1:= \prod_{k=2}^m\frac{1}{|X_k|^{2d_k-2}}\prod_{j\in J}\frac{1}{|x_j|^{2-2\beta_j}}.\]
A quick argument by induction gives the expected estimate for the RHS of \eqref{new667}.
\smallskip

\noindent The remaining terms involve the differential of $\displaystyle  X_1\to \theta(|X_1|^2)$, on the support of which 
$\displaystyle \frac{1}{|X_1|^{2d_1-2}}$ is smooth. 
\smallskip

\noindent In conclusion, after the first integration by parts we get
\begin{equation}\label{new67}
\int_U\theta_1(X_1)\psi_1(X_1)\log|t_1- f_1(x)|^2 \theta_1(x)\Lambda_1(X)\bigwedge_{i=1}^{l_1} dd^c|X_1|^2\wedge \omega_{l'}\wedge \bigwedge_{j=2}^m dd^c\log|t_j- f_j(x)|^2
\end{equation}
modulo the terms \eqref{new667}.
\smallskip

\noindent We now repeat this procedure, integrating by parts using the factor $\displaystyle dd^c\log|t_2- f_2(x)|^2$. Notice that, because of the 
form $\displaystyle \bigwedge_{i=1}^{l_1} dd^c|X_1|^2$ the derivatives of the function
\[\theta_1(X_1)\psi_1(X_1)\log|t_1- f_1(x)|^2\]
don't come into play. 
\medskip

\noindent Thus, after a finite number of steps the last term we still have to deal with equals
\begin{equation}\label{new68}
\int_U\theta(x)\psi(x)\prod_{i=1}^m \log|t_i- f_i(x)|^2 \frac{d\lambda}{\prod_{j\in J}{|x_j|^{2-2\beta_j}}}
\end{equation}
where $\psi$ is a truncation function. Since $\beta_j> 0$, the integral \eqref{new68} is uniformly bounded as soon as we fix a bound for $|t|$.
Collecting all the terms, Theorem \ref{taka} is proved.
\medskip

\begin{remark}
We consider the function \[f: \C^3\to \C^2,\qquad f(z)= (z_1z_2, z_1z_3).\] Let $\theta$ be a truncation function which equals 1 near the origin of $\C^3$. A simple calculation shows that we have 
\[\int_{\C^3}\theta e^{-\psi}dd^c\log|t_1- f_1(x)|^2\wedge dd^c\log|t_2- f_2(x)|^2\wedge\omega \simeq \frac{1}{|t_1|^2+ |t_2|^2}\]
therefore the hypothesis $(\star)$ is crucial.
\end{remark}
\medskip

\begin{remark}
Let $f:S\to W$ be a morphism such that all the hypothesis of Theorem \ref{lelong} except perhaps for 
$(\star)$ are satisfied. Then the techniques developed in our article show that the current $T$ in \eqref{new45} can only have positive Lelong numbers along an analytic subset of $X$ which projects in codimension two. The reason is that one can construct a subset $W_0\subset W$ whose codimension is at least two, and such that the morphism $f$ satisfies $(\star)$ in the complement of $W_0$. 
It follows that if $\dim W= 2$, Theorem \ref{lelong} holds true for morphisms which only satisfy
the assumptions {\rm (a), (b), (c)}, thanks to the following general fact.
\begin{theorem}\cite{Dem92}
Let $X$ be a compact complex manifold, and let $T$ be a $(1,1)$--closed positive current on $X$.
If the level sets
\[E_c(T):= \{x\in X : \nu(T, x)\geq c\}\]
have dimension zero for any $c>0$, then the cohomology class of $T$ is nef. 
\end{theorem}
\end{remark}
\medskip



\end{document}